\documentclass[12pt]{article}
\usepackage{amssymb}
\usepackage{amsfonts}
\usepackage{amsmath,latexsym,amssymb,amsfonts,amsbsy, amsthm}
\usepackage[usenames]{color}
\usepackage{xspace,colortbl}
\usepackage{mathrsfs}
\usepackage[colorlinks,linkcolor=blue,citecolor=blue]{hyperref}
\usepackage[titletoc]{appendix}
\allowdisplaybreaks

\newcommand{\beq}{\begin{equation}}
\newcommand{\R}{{\mathbb R}}
\newcommand{\eeq}{\end{equation}}
\newcommand{\ben}{\begin{eqnarray}}
\newcommand{\een}{\end{eqnarray}}
\newcommand{\beno}{\begin{eqnarray*}}
\newcommand{\eeno}{\end{eqnarray*}}

\newtheorem{thm}{Theorem}[section]
\newtheorem{defi}[thm]{Definition}
\newtheorem{lem}[thm]{Lemma}
\newtheorem{prop}[thm]{Proposition}
\newtheorem{coro}[thm]{Corollary}
\newtheorem{rmk}[thm]{Remark}

\renewcommand{\theequation}{\thesection.\arabic{equation}}

\title{\textbf{A new proof of Savin's theorem on Allen-Cahn equations}}
\author{Kelei Wang
\\
{\small  School of Mathematics and Statistics  \& Computational
Science Hubei Key Laboratory}
\\
{\small Wuhan University, Wuhan 430072, China}\\
{\small wangkelei@whu.edu.cn } \ }
\date{}

\begin{document}
\maketitle
\begin{abstract}
In this paper we establish an improvement of tilt-excess decay
estimate for the Allen-Cahn equation, and use this to give a new
proof of Savin's theorem on the uniform $C^{1,\alpha}$ regularity of
flat level sets.  This generalizes Allard's $\varepsilon$-regularity
theorem for stationary varifolds to the setting of Allen-Cahn
equations. A new proof of Savin's theorem on the one dimensional
symmetry of minimizers in $\mathbb{R}^n$ for $n\leq 7$ is also
given.
\end{abstract}


\vskip 0.2cm


\renewcommand{\theequation}{\thesection.\arabic{equation}}
\setcounter{equation}{0}

\tableofcontents

\section{Introduction}
\numberwithin{equation}{section}
 \setcounter{equation}{0}

This paper is devoted to generalize Allard's regularity theory in
{\it Geometric Measure Theory} to the setting of Allen-Cahn
equations and discuss its application to the De Giorgi conjecture.

The Allen-Cahn equation
\begin{equation}\label{Allen-Cahn}
\Delta u=u^3-u,
\end{equation}
is a typical model of phase transition. By now, it has been studied from various aspects.
 One particular feature of this equation is its close relation with the minimal surface theory, through its
 singularly perturbed version
\[\varepsilon\Delta u_\varepsilon=\frac{1}{\varepsilon}\left(u_\varepsilon^3-u_\varepsilon\right).\]

By this connection and in view of the Bernstein theorem for minimal hypersurfaces \cite{Simons}, De Giorgi made the following conjecture in \cite{D2}:

{\it \ \ Let $u\in C^2(\R^{n+1})$ be a solution of \eqref{Allen-Cahn} such that
\[|u|\leq 1,\ \ \ \frac{\partial u}{\partial x_{n+1}}>0\ \ \mbox{in}\ \R^{n+1}.\]
Then $u$ depends only on one variable, if $n\leq 7$. }

This conjecture has been considered by many authors, including
Ghoussoub and Gui \cite{G-G}, Ambrosio and Cabr\'{e} \cite{A-C} and
Savin \cite{Savin}. Counterexamples in $\R^9$ are also constructed
by del Pino, Kowalczyk and Wei in \cite{DKW}. In particular, Savin
proved an improvement of flatness type result for minimizing
solutions (i.e. a minimizer of the energy functional). This result
says, given any $\theta_0>0$, for a minimizer $u$, if in a ball
$\mathcal{B}_l$ with $l$ large, its zero level set is trapped in a
strip $\{|x_{n+1}|<\theta\}$ with $\theta\geq\theta_0$, which is
sufficiently narrow (i.e. $\theta l^{-1}$ small), then by shrinking
the radius of the ball, possibly after a rotation of coordinates,
the zero level set of $u$ is trapped in a flatter strip.

By using this estimate, Savin proved
\begin{thm}\label{main result}
Let $u$ be a minimizing solution of \eqref{Allen-Cahn} defined on
the entire space $\mathbb{R}^{n+1}$ where $n\leq 6$. Then $u$ is one
dimensional.
\end{thm}
For $n>6$, if we add some further assumptions on level sets of $u$,
e.g. the global Lipschitz regularity of $\{u=0\}$, it is still
possible to prove the one dimensional symmetry of $u$. This theorem
also implies the original De Giorgi conjecture, with an additional
assumption that
\[\lim_{x_{n+1}\to\pm \infty}u(x,x_{n+1})=\pm 1.\]

This type of improvement of flatness result appears in the partial
regularity theory for various elliptic problems, although sometimes
in rather different forms. One main ingredient to establish this
improvement of flatness is the {\em blow up (or harmonic
approximation)} technique, first introduced by De Giorgi in his work
\cite{D} on the almost everywhere regularity of minimal
hypersurfaces.

Although the statement of Savin's improvement of flatness result
bears many similarities with De Giorgi theorem, the proof in
\cite{Savin} employs some new ideas. Indeed, it is based on
Caffarelli-Cordoba's proof of De Giorgi theorem in \cite{C-C}. This
approach uses the ``viscosity" side of the problem, and relies
heavily on a Krylov-Safanov type argument. In particular, Savin
first obtained a Harnack inequality (hence some kind of uniform
H\"{o}lder continuity) and then used this to prove that the blow up
sequence converges uniformly to a harmonic function.

Savin's approach can be applied to many other problems, even without
variational structure, see for example \cite{Savin 1, Savin 2, Savin
3}. However, it seems that the maximum principle and Harnack
inequality are crucial in this approach.
At present it is  still not clear how to get this kind of
improvement of flatness result for elliptic systems, where Harnack
inequality may fail. Thus, in view of the connection between
Allen-Cahn equations and minimal hypersurfaces, in this paper we
intend to explore the {\it variational} side of improvement of
flatness and establish some results paralleling classical regularity
theories in {\it Geometric Measure Theory}. As in Allard's
regularity theory \cite{Allard} (see also \cite[Section 6.5]{Lin}
for an account), we use the following {\it excess} (for more
details, see Section 2 and 3)
\[\int_{\mathcal{C}_1(0)}\left[1-\left(\nu_\varepsilon\cdot e_{n+1}\right)^2\right]\varepsilon|\nabla u_\varepsilon|^2,\]
where $\mathcal{C}_1(0)$ is the cylinder
$B_1(0)\times(-1,1)\subset\R^{n+1}$ and $\nu_\varepsilon=\nabla
u_\varepsilon/|\nabla u_\varepsilon|$ is the unit normal vector to
level sets of $u_\varepsilon$. This quantity was first used by
Hutchinson-Tonegawa \cite{H-T} to derive the integer multiplicity of
the limit varifold arising from general critical points in the
Allen-Cahn problem.

This quantity can be used to measure the flatness of level sets of
$u_\varepsilon$ (see Lemma \ref{lem excess small} below). Similar to
Allard's $\varepsilon$-regularity theorem, if the excess in a ball
is small, then after shrinking the radius of the ball and possibly
rotating the vector $e_{n+1}$ a little, the excess becomes smaller.
This {\it improvement of tilt-excess} is the main step in the proof
of Allard's $\varepsilon$-regularity theorem, and also in our
argument. In contrast to the quantity used in Savin's improvement of
flatness result,
 the excess is an energy type quantity. Indeed, if all the level sets $\{u_\varepsilon=t\}$ can be represented by graphs
 along the $(n+1)$-th direction, in the form $\{x_{n+1}=h(x,t)\}$,
 then the excess can be written as
\[\int_{\mathcal{C}_1(0)}\left[1-\left(\nu_\varepsilon\cdot e_{n+1}\right)^2\right]\varepsilon|\nabla u_\varepsilon|^2
=\int_{-1}^{1}\left(\int_{B_1(0)}\frac{|\nabla_x
h(x,t)|^2}{1+|\nabla_x h(x,t)|^2}\varepsilon|\nabla
u_\varepsilon(x,h(x,t))|dx\right)dt,\] which is almost a weighted
Dirichlet energy, provided $\sup|\nabla_x h(x,t)|$ small.

Thus the problem can be approximated by harmonic functions
(corresponding to critical points of the Dirichlet energy) if
$|\nabla_x h(x,t)|$ is small. (We will see that only the smallness
of the excess is sufficient for this purposes.) This is exactly the
content of {\it harmonic approximation} technique. Using the excess
allows us to work in the Sobolev spaces and apply standard compact
Sobolev embedding results to get the blow up limit, while in Savin's
version the main difficulty lies in the compactness for the blow up
sequence where his Harnack inequality enters.

We also note that this type of tilt-excess decay result was known by
Tonegawa, see \cite{Tonegawa}, where he showed that this result
implies the uniform $C^{1,\alpha}$ regularity of intermediate
transition layers in dimension $2$.

However,  in this tilt-excess decay estimate we need one more
assumption:
\begin{equation}\label{obstruction}
\int_{\mathcal{C}_1}\left[1-\left(\nu_\varepsilon\cdot
e\right)^2\right]\varepsilon|\nabla
u_\varepsilon|^2\gg\varepsilon^2.
\end{equation}
 Compared to Allard's $\varepsilon$-regularity theorem, this condition is not so satisfactory. It prevents us from applying this improvement of decay directly to
deduce the uniform $C^{1,\alpha}$ regularity of intermediate
transition layers. (This obstruction was also observed by Tonegewa
in \cite{Tonegawa}.) One reason for the appearance of this condition
\eqref{obstruction} is due to the fact that the energy, although
mostly concentrated on the transition part, is still distributed on
a layer of width $\varepsilon$. Note that this phenomena does not
appear in minimal surface theory.

In Savin's version of improvement of flat estimate, an assumption
similar to \eqref{obstruction} is also needed. Using our
terminology, it is equivalent to requiring that the excess is not of
the order $o(\varepsilon^2)$. Note that this is weaker than
\eqref{obstruction}. This weaker assumption is perhaps due to the
fact that in Savin's approach only a single level set is considered,
while our improvement of flat estimate involves a family of level
sets.

By exploiting the fact that $u_\varepsilon$ is close to a one
dimensional solution up to $O(\varepsilon)$ scales, an iteration of
the improvement of tilt-excess decay estimate gives a Morrey type
bound on level sets of $u_\varepsilon$, which then implies that
these level sets are graphs. Here, once again due to the obstruction
\eqref{obstruction}, this Morrey type bound does not imply the
$C^{1,\alpha}$ regularity of $\{u_\varepsilon=0\}$, but only a
Lipschitz one. However, under the condition that
$\{u_\varepsilon=0\}$ is a Lipschitz graph, Caffarelli and Cordoba
\cite{C-C 3} have shown that transition layers are uniformly bounded
in $C^{1,\alpha}$ for some $\alpha\in(0,1)$. Thus we get a full
analogue of Allard's $\varepsilon$-regularity theorem in the
Allen-Cahn setting (see Theorem \ref{main result loc}).

In this paper we do not fully avoid the use of maximum principle.
For example, it seems that the Modica inequality is indispensable in
our argument, because we need it to derive a monotonicity formula
with the correct exponent. We also need to apply the moving plane
(or sliding) method (as in Farina \cite{F}) to deduce the one
dimensional symmetry of some entire solutions. There a distance type
function is used to control the behavior of $u$ at the place far
away from the transition part. This function behaves like a distance
function and this fact follows from the Modica inequality. However,
we do avoid the use of any Harnack inequality. It may be possible to
remove the above mentioned deficiency by strengthening the
tilt-excess decay estimate, but as explained above, the current
version of Theorem \ref{thm tilt excess decay} is already sufficient
for proving Theorem \ref{main result}, see Section 11.

The above approach was first used by the author in \cite{W2}, where
we consider a De Giorgi type conjecture for the elliptic system
\[\Delta u=uv^2,\ \ \Delta v=vu^2,\ \ u,v>0~~\text{in}~~\mathbb{R}^n.\]
For the corresponding singularly perturbed system
\begin{equation*}
\left\{
\begin{aligned}
 &\Delta u_\kappa=\kappa u_\kappa v_\kappa^2, \\
 &\Delta v_\kappa=\kappa v_\kappa u_\kappa^2,
                          \end{aligned} \right.
\end{equation*}
an improvement of flatness result was established by using the
quantity
\[\int_{B_1}\big|\nabla \left( u_\kappa-v_\kappa-e\cdot x\right)\big|^2dx.\]

These two proofs are similar in the spirit. In particular, in order
to show that the blow up limit is a harmonic function,
 we mainly use the {\it stationary condition} associated to the equation, not the equation itself.
This is more apparent in the current setting, because the stationary
condition for the singularly perturbed Allen-Cahn equation is
directly linked to the corresponding one in their limit problem, the
stationary condition for varifolds (in the sense of Allard
\cite{Allard}, see also \cite{H-T}). Furthermore, since the excess
is a kind of the $H^1$ norm, to prove the strong convergence in
$H^1$ space, we implicitly use a Caccioppoli type inequality, which
is again deduced from the stationary condition by choosing a
suitable test function (see Remark \ref{rmk Caccioppoli} below).

Finally, although in our improvement of tilt-excess decay result
(Theorem \ref{thm tilt excess decay}) and $\varepsilon$-regularity
result (Theorem \ref{main result loc}), we do not assume the
solution to be a minimizer, the multiplicity one property of
transition layers is still needed here. (This is associated to the
unit density property of the limit varifold.) Thus our results do
not remove the no folding assumption in Savin's result. However, we
feel that a generalization of our technique to the case with
multiple transition layers is possible, which should be of more
interest.

The organization of this paper can be seen from the table of
contents. Part I is devoted to prove the tilt excess decay estimate,
Theorem \ref{thm tilt excess decay}. In Part II, we establish an
Allard type $\varepsilon$-regularity theorem, the uniform
$C^{1,\alpha}$ regularity of intermediate layers, see Theorem
\ref{main result loc}. In the proof of this theorem, a De Giorgi
type conjecture for a class of entire solutions (see Theorem
\ref{main result 3}) is also obtained, which includes Theorem
\ref{main result} as a special case.

\section{Settings and notations}
\numberwithin{equation}{section}
 \setcounter{equation}{0}

We shall work in the following settings. Consider the Allen-Cahn
equation in the general form as
\begin{equation}\label{equation 0}
\Delta u=W^\prime(u),
\end{equation}
where $W$ is a double well potential, that is, $W\in C^3(\R)$,
satisfying
\begin{itemize}
\item $W\geq0$, $W(\pm 1)=0$ and $W>0$ in $(-1,1)$;
\item  for some $\gamma\in(0,1)$, $W^\prime<0$ on $(\gamma,1)$ and $W^\prime>0$ on $(-1,-\gamma)$;
\item $W^{\prime\prime}\geq\kappa>0$ for all $|x|\geq\gamma$.
\end{itemize}
A typical example is $W(u)=(1-u^2)^2/4$, which gives
\eqref{Allen-Cahn}.

Through a scaling $u_\varepsilon(X):=u(\varepsilon^{-1}X)$, we get
the singularly perturbed version of the Allen-Cahn equation:
\begin{equation}\label{equation}
\varepsilon\Delta u_\varepsilon=\frac{1}{\varepsilon}W^\prime(u_\varepsilon).
\end{equation}
This equation arises as the Euler-Lagrange equation of the energy
functional (after adding suitable boundary conditions)
\begin{equation}\label{functional}
E_\varepsilon(u_\varepsilon)=\int\frac{\varepsilon}{2}|\nabla u_\varepsilon|^2+\frac{1}{\varepsilon}W(u_\varepsilon).
\end{equation}
We say $u_\varepsilon$ is a minimizer (or a minimizing solution), if
for every ball $\mathcal{B}$ in the definition domain of
$u_\varepsilon$,
\[\int_{\mathcal{B}}\frac{\varepsilon}{2}|\nabla u_\varepsilon|^2+\frac{1}{\varepsilon}W(u_\varepsilon)
\leq\int_{\mathcal{B}}\frac{\varepsilon}{2}|\nabla
v|^2+\frac{1}{\varepsilon}W(v),\] for any $v\in H^1(\mathcal{B})$
satisfying $v=u_\varepsilon$ on $\partial\mathcal{B}$.

We will always assume $|u_\varepsilon|\leq 1$, and it satisfies the
Modica inequality
\begin{equation}\label{Modica inequality}
\frac{\varepsilon}{2}|\nabla
u_\varepsilon|^2\leq\frac{1}{\varepsilon}W(u_\varepsilon).
\end{equation}
This inequality (in the exact form as above) may not be essential in
the argument, but we prefer to assume it to make the arguments
clean. (These estimates can be relaxed, cf. \cite{H-T}.) By standard
elliptic estimates, there exists a universal constant $C$ such that
\begin{equation}\label{gradient bound}
\varepsilon|\nabla
u_\varepsilon|+\varepsilon^2|\nabla^2u_\varepsilon|\leq C.
\end{equation}
In particular, $u_\varepsilon$ is a classical solution.

For any smooth vector field $Y$ with compact support, by considering the domain variation in the form
\[u_\varepsilon^t(X):=u_\varepsilon(X+tY(X)), \ \ \mbox{for}\ |t| \ \mbox{small},\]
from the definition of critical points we get
\[\frac{d}{dt}\int\frac{\varepsilon}{2}|\nabla u_\varepsilon|^2+\frac{1}{\varepsilon}W(u_\varepsilon)\Big|_{t=0}=0.\]
After some integration by parts, we obtain the stationary condition
for $u_\varepsilon$:
\begin{equation}\label{stationary condition}
\int\left(\frac{\varepsilon}{2}|\nabla u_\varepsilon|^2+\frac{1}{\varepsilon}W(u_\varepsilon)\right)\mbox{div}Y-\varepsilon DY(\nabla u_\varepsilon,\nabla u_\varepsilon)=0.
\end{equation}

Finally, by the assumption on $W$, there exists a one dimensional
solution $g(t)$ defined on $t\in(-\infty,+\infty)$, satisfying
\begin{equation}\label{1-d solution 1}
g^{\prime\prime}(t)=W^\prime(g(t)),  \ \ \ \ \lim_{t\to\pm\infty}g(t)=\pm1,
\end{equation}
where the convergence rate is exponential.

The first integral for $g$ can be written as
\begin{equation}\label{1-d solution 2}
g^\prime(t)=\sqrt{2W(g(t))}>0.
\end{equation}
For any $\varepsilon>0$, we denote $g_\varepsilon(t)=g(\varepsilon^{-1}t)$, which satisfies
\begin{equation}\label{1-d solution perturbed}
\varepsilon
g_\varepsilon^{\prime\prime}(t)=\frac{1}{\varepsilon}W^\prime(g_\varepsilon(t)).
\end{equation}

Throughout this paper, $\sigma_0$ denotes the constant defined by
\begin{equation}\label{energy of 1-d solution}
\sigma_0:=\int_{-\infty}^{+\infty}g^\prime(t)^2dt=\int_{-\infty}^{+\infty}\frac{1}{2}g^\prime(t)^2+W(g(t))dt.
\end{equation}

In this paper we adopt the following notations.
\begin{itemize}
\item A point in $\R^{n+1}$ will be denoted by $X=(x,x_{n+1})\in\R^n\times\R$.
\item $\mathcal{B}_r(X)$ denotes an open ball in $\R^{n+1}$ and $B_r(x)$ an open ball in $\R^n$. If the center is the origin $0$, we write
it as $\mathcal{B}_r$ (or $B_r$).
\item $\mathcal{C}_r(x)=B_r(x)\times (-1,1)\subset\R^{n+1}$ the finite cylinder over $B_r(x)\subset\R^{n}$.
\item $e_i$, $1\leq i\leq n+1$ denote the standard basis in $\R^{n+1}$.
\item $P$ denotes a hyperplane in $\R^{n+1}$ and $\Pi_P$  (or simply $P$) the orthogonal projection onto it. If $P=\R^n$, we simply use $\Pi$.

\item $G(n)$ denotes the Grassmann manifold of unoriented
$n$-dimensional hyperplanes in $\R^{n+1}$.

\item A varifold $V$ is a Radon measure on $\R^{n+1}\times G(n)$. We use $\|V\|$ to denote the weighted measure of $V$, that is, for any measurable set $A\subset\R^{n+1}$,
\[\|V\|(A)=V(A\times G(n)).\]

\item For a measure $\mu$, $\mbox{spt}\mu$ denotes its support.

\item $\nu_\varepsilon(X)=\frac{\nabla u_\varepsilon(X)}{|\nabla u_\varepsilon(X)|}$ if $\nabla u_\varepsilon(X)\neq0$, otherwise we take it to be an arbitrary unit vector.
\item $\mu_\varepsilon:=\varepsilon|\nabla u_\varepsilon|^2dX$.
\item $\mathcal{H}^s$ denotes the $s$-dimensional Hausdorff measure.
\item $\omega_n$ denotes the volume of the unit ball $B_1$ in $\R^n$.
\item $H^1$ is the Sobolev space with the norm $\left(\int|\nabla u|^2+|u|^2\right)^{1/2}$.
\item $\mbox{dist}_H$ is the Hausdorff distance between sets in $\R^{n+1}$.
\item Unless otherwise stated, universal constants $C$, $C_i$ and $K_i$ (large) and $c_i$ (small) depend only on the
dimension $n$ and the potential function $W$.
\end{itemize}

Throughout this paper $u_\varepsilon$ always denotes a solution of \eqref{equation}. 
 We use $\varepsilon$  to denote a sequence of parameters converging to $0$, which should be written as $\varepsilon_i$
 if we want to be precise.

\part{Tilt-excess decay}

\section{Statement}
\numberwithin{equation}{section}
 \setcounter{equation}{0}

The following quantity will play an important role in our analysis.
\begin{defi}[\bf Excess]
Let $P$ be an $n$-dimensional hyperplane in $\R^{n+1}$ and $e$ one of its unit normal vector, $B_r(x)\subset P$ an open ball and $\mathcal{C}_r(x)=B_r(x)\times(-1,1)$ the cylinder over $B_r(x)$. The excess of $u_\varepsilon$ in $\mathcal{C}_r(x)$ with respect to $P$ is
\begin{equation}\label{excess}
E(r;x,u_\varepsilon,P):=r^{-n}\int_{\mathcal{C}_r(x)}\left[1-\left(\nu_\varepsilon\cdot e\right)^2\right]\varepsilon|\nabla u_\varepsilon|^2dX.
\end{equation}
\end{defi}
If $P=\R^n$ and $e=e_{n+1}$, the excess equals
\[E(r;x,u_\varepsilon)=r^{-n}\int_{\mathcal{C}_r(x)}\varepsilon\sum_{i=1}^n\left(\frac{\partial u_\varepsilon}{\partial x_i}\right)^2dX.\]

\begin{rmk}\label{rmk 1}
For any unit vector $\nu$ and $e$, we have
\[|\nu-e||\nu+e|\geq\sqrt{2}\min\{|\nu-e|,|\nu+e|\}.\]
Therefore
\begin{eqnarray*}
1-\left(\nu\cdot e\right)^2&=&\left(1-\nu\cdot e\right)\left(1+\nu\cdot e\right)\\
&=&\frac{1}{4}|\nu-e|^2|\nu+e|^2\\
&\geq&\frac{1}{2}\min\{|\nu-e|^2,|\nu+e|^2\}.
\end{eqnarray*}

By projecting the unit sphere $\mathbb{S}^n$ to the real projective
space $\mathbb{RP}^n$ (both with the standard metric), we get
\[1-\left(\nu\cdot e\right)^2\geq c \mbox{ dist}_{\mathbb{RP}^n}(\nu,e)^2,\]
for some universal constant $c$.
\end{rmk}

Our main objective in Part I is to prove the following decay estimate.
\begin{thm}[\bf Tilt-excess decay]\label{thm tilt excess decay}
Given a constant $b\in(0,1)$, there exist five universal constants
$\delta_0,\tau_0,\varepsilon_0>0$, $\theta\in(0,1/4)$ and $K_0$
large so that the following holds. Let $u_\varepsilon$ be a solution
of \eqref{equation} with $\varepsilon\leq \varepsilon_0$ in
$\mathcal{B}_4$, satisfying the Modica inequality \eqref{Modica
inequality}, $|u_\varepsilon(0)|\leq 1-b$, and
\begin{equation}\label{close to plane 0}
4^{-n}\int_{\mathcal{B}_4}\frac{\varepsilon}{2}|\nabla u_\varepsilon|^2+\frac{1}{\varepsilon}W(u_\varepsilon)\leq \left(1+\tau_0\right)\sigma_0\omega_n.
\end{equation}
Suppose the excess with respect to $\R^n$ satisfies
\begin{equation}\label{small excess 0}
\delta_\varepsilon^2:=E(2;0,u_\varepsilon,\R^n)\leq\delta_0^2,
\end{equation}
where $\delta_\varepsilon\geq K_0\varepsilon$.
Then there exists another plane $P$, such that
\begin{equation}\label{excess decay}
E(\theta;0,u_\varepsilon,
P)\leq\frac{\theta}{2}E(2;0,u_\varepsilon,\R^n).
\end{equation}
Moreover, there exists a universal constant $C$ such that
\[\|e-e_{n+1}\|\leq CE(2;0,u_\varepsilon,\R^n)^{1/2},\]
where $e$ is the unit normal vector of $P$ pointing to the above.
\end{thm}

Roughly speaking, this theorem says, if the excess (with respect to
some hyperplane) in a ball is small enough, then after shrinking the
radius of the ball and perhaps tilting the hyperplane a little, the
excess becomes smaller. This decay estimate will be used in Part II
to prove the uniform Lipschitz regularity of intermediate layers.

The condition \eqref{close to plane 0} says there is only a single
transition layer, which corresponds to the unit density assumption
in Allard's $\varepsilon$-regularity theorem. In the next section we
shall see that \eqref{small excess 0} always holds (with respect to
a suitable hyperplane), provided that \eqref{close to plane 0} is
satisfied with $\tau_0$ sufficiently small (depending on
$\delta_0$). However, the assumption that
$\delta_\varepsilon\gg\varepsilon$ is crucial here, which is not so
satisfactory compared to Allard's and Savin's version.

We shall prove this theorem by contradiction. 
Thus assume there exists a sequence of $\varepsilon_i$ (for
simplicity the subscript $i$ will be dropped), and a sequence of
solutions $u_{\varepsilon}$ satisfying all of the assumptions in
Theorem \ref{thm tilt excess decay}, that is,
\begin{enumerate}
\item there exists a sequence of $\tau_\varepsilon\to0$ such that
\begin{equation}\label{close to plane 1}
4^{-n}\int_{\mathcal{B}_4}\frac{\varepsilon}{2}|\nabla
u_{\varepsilon}|^2+\frac{1}{\varepsilon}W(u_{\varepsilon})\leq
\left(1+\tau_\varepsilon\right)\sigma_0\omega_n,
\end{equation}

\item  the excess satisfies
\begin{equation}\label{small excess}
\delta_{\varepsilon}^2:=E(2;0,u_{\varepsilon},\R^n)\to0,
\end{equation}
where
\begin{equation}\label{absurd assumption 3}
\frac{\delta_{\varepsilon}}{\varepsilon}\to+\infty,\ \ \ \mbox{as}\
\varepsilon\to0,
\end{equation}
\end{enumerate}
but for any unit vector $e$ with
\begin{equation}\label{absurd assumption}
\|e-e_{n+1}\|\leq CE(2;0,u_{\varepsilon},\R^n)^{\frac{1}{2}},
\end{equation}
where the constant $C$ will be determined below (by the constant in \eqref{8.01}), we have
\begin{equation}\label{absurd assumption 2}
E(\theta;0,u_{\varepsilon},
P)\geq\frac{\theta}{2}E(2;0,u_{\varepsilon},\R^n).
\end{equation}
Here $P$ is the hyperplane orthogonal to $e$ and $\theta$ is also a
constant to be determined later (see \eqref{8.3.1} and
\eqref{8.3.2}).

The remaining part, up to and including Section 8, will be devoted
to derive a contradiction from the assumptions \eqref{close to plane
1}-\eqref{absurd assumption 2}. The proof is divided into four
steps:
\begin{enumerate}
\item [{\bf Step 1.}] It is shown that $\{u_\varepsilon=t\}$ (for $t\in(-1+b,1-b)$) can be represented by Lipschitz graphs
 over $\R^n$,
$x_{n+1}=h_\varepsilon^t(x)$, except a bad set of small measure
(controlled by $E(2;0,u_\varepsilon,\R^n)$). This is achieved by the
weak $L^1$ estimate for Hardy-Littlewood maximal functions.

\item [{\bf Step 2.}] By writing the excess using the $(x,t)$
coordinates ($t$ as in Step 1), $h_\varepsilon^t/\delta_\varepsilon$
are uniformly bounded in $H^1_{loc}(B_1)$. Then we can assume that
they converge weakly to a limit $h_\infty$. Here we need the
assumption $\delta_\varepsilon\gg \varepsilon$ to guarantee the
limit is independent of $t$.

\item [{\bf Step 3.}] By choosing $X=\varphi\psi e_{n+1}$ in the
stationary condition \eqref{stationary condition}, where $\varphi\in
C_0^\infty(B_1)$ and $\psi\in C_0^\infty((-1,1))$, and then passing
to the limit, it is shown that $h_\infty$ is harmonic in $B_1$.

\item [{\bf Step 4.}] By choosing $X=\varphi\psi x_{n+1}e_{n+1}$ in the
stationary condition \eqref{stationary condition} and then passing
to the limit, it is shown that (roughly speaking)
$h_\varepsilon^t/\delta_\varepsilon$ converges strongly in
$H^1_{loc}(B_1)$. The tilt-excess decay estimate then follows from
some basic estimates on harmonic functions.
\end{enumerate}

After establishing some preliminary results in the next section,
Step 1-4 will be given in Section 5-8 respectively.

\section{Compactness results}
\numberwithin{equation}{section}
 \setcounter{equation}{0}

In this section, we study the convergence of various quantities
associated to $u_\varepsilon$ and establish some preliminary results
for the proof of Theorem \ref{thm tilt excess decay}.

Recall that we have assumed the Modica inequality \eqref{Modica
inequality}. An important consequence of this inequality is the
following monotonicity formula (see for example \cite{Mocica 3}).
\begin{prop}[\bf Monotonicity formula]\label{monotonicity formula}
For any $X\in\mathcal{B}_3$,
\[r^{-n}\int_{\mathcal{B}_r}\frac{\varepsilon}{2}|\nabla u_\varepsilon|^2+\frac{1}{\varepsilon}W(u_\varepsilon)\]
is non-decreasing in $r\in(0,1)$.
\end{prop}

By combining Proposition \ref{monotonicity formula} with
\eqref{close to plane 1}, we see
\begin{coro}
For any $\mathcal{B}_r(X)\subset\mathcal{B}_3$, we have
\begin{equation}\label{energy bound on ball}
\int_{\mathcal{B}_r(X)}\frac{\varepsilon}{2}|\nabla u_\varepsilon|^2+\frac{1}{\varepsilon}W(u_\varepsilon)\leq 8^n\sigma_0\omega_nr^n.
\end{equation}
\end{coro}

We use the main result in Hutchinson-Tonegawa \cite{H-T} to study
the convergence of $u_\varepsilon$.
 Define the varifold
$V_\varepsilon$ by
\[<V_\varepsilon,\Phi(X,S)>=\int\Phi(x,I-\nu_\varepsilon\otimes\nu_\varepsilon)\varepsilon|\nabla u_\varepsilon|^2dx,\ \ \ \forall \   \Phi\in  C_0^\infty(\mathcal{C}_2\times G(n)).\]
 Hutchinson and Tonegawa proved:
\begin{enumerate}
\item
  As $\varepsilon\to 0$, $V_\varepsilon$ converges in the sense of varifolds to a stationary, rectifiable varifold $V$ with integer density (modulo division by the constant
  $\sigma_0$).
\item The measures
$\mu_\varepsilon$ converge to $\|V\|$ weakly.
\item  The discrepancy quantity
\begin{equation}\label{discrepancy}
\left(\frac{1}{\varepsilon}W(u_\varepsilon)-\frac{\varepsilon}{2}|\nabla
u_\varepsilon|^2\right)\rightarrow0,\ \mbox{in}\ L^1_{loc}.
\end{equation}
\item For any $t\in(-1,1)$ fixed, $\{u_\varepsilon=t\}$
converges to $\mbox{spt}\|V\|$ in the Hausdorff distance.
\end{enumerate}
 The last statement implies $0\in\mbox{spt}\|V\|$, because $0\in\{|u_\varepsilon|\leq
1-b\}$ ($b$ as in Theorem \ref{thm tilt excess decay}).

With the help of the bound \eqref{close to plane 1}, we can give a
description of the limit varifold.
\begin{prop}[{\bf Limit varifold}] \label{prop lim varifold}
The limit measure satisfies
$\|V\|=\sigma_0\mathcal{H}^n\lfloor_{\R^n}$.
\end{prop}
\begin{proof}
By taking the limit in \eqref{close to plane 1} and using the
integer multiplicity of $V$, we get
\[4^{-n}\|V\|(\mathcal{B}_4)=\sigma_0\omega_n.\]
On the other hand, the integer multiplicity of $V$ implies
\[\lim_{r\to0}r^{-n}\|V\|\geq\sigma_0\omega_n.\]
By the monotonicity formula for stationary varifolds (cf.
\cite[Theorem 6.3.2]{Lin}), we deduce that $V$ is a cone.

Recall that $V$ is a rectifiable, stationary varifold with integer
multiplicity. What we have shown says that $V$ has density one at
the origin. Hence Allard's $\varepsilon$-regularity theorem implies
that $\mbox{spt}\|V\|$ is a smooth hypersurface in a neighborhood of
the origin.

Combining the cone property with this smooth regularity, we see $V$
is the standard varifold associated to a hyperplane with unit
density.
\end{proof}

Now we show that away from $\R^n$, $u_\varepsilon$ is exponentially
close to $\pm 1$.
\begin{prop}\label{prop exponential decay}
For any $h>0$, if $\varepsilon$ is sufficiently small, we have
\[\left(1-u_\varepsilon^2\right)+|\nabla u_\varepsilon|\leq C(h)e^{-\frac{|x_{n+1}|}{C(h)\varepsilon}}\ \ \  \mbox{in}\ \ \ \mathcal{C}_2\setminus\{|x_{n+1}|\leq h\}.\]
In particular, $\{u_\varepsilon=0\}\cap \mathcal{C}_2$ lies in an
$h$-neighborhood of $\R^n\cap \mathcal{C}_2$.
\end{prop}
\begin{proof}
By \cite{H-T}, $u_\varepsilon^2$ converges to $1$ uniformly on any
compact set outside $\mbox{spt}\|V\|=\R^n$. In particular, for all
$\varepsilon$ small,
\[u_\varepsilon^2\geq\gamma \quad \mbox{in}\ \mathcal{C}_3\setminus\{|x_{n+1}|\geq \frac{h}{2}\}.\]
By a direct calculation, there exists a universal constant $c$ such
that
\[\Delta(1-u_\varepsilon^2)\geq \frac{c}{\varepsilon^2}(1-u_\varepsilon^2) \quad
\mbox{in}\  \mathcal{C}_3\setminus\{|x_{n+1}|\geq\frac{h}{2}\}.\] By
this inequality we can apply Lemma \ref{lem B1} to get the
exponential decay of $1-u_\varepsilon^2$ in $\{|x_{n+1}|>h\}$. The
estimate for $|\nabla u_\varepsilon|$ follows from standard interior
gradient estimates.
\end{proof}
\begin{rmk}\label{rmk Hausdorff convergence}
If $u_\varepsilon$ converges to $1$ (or $-1$) on both sides of
$\R^n$, the multiplicity of $V$ will be greater than $1$ (see
\cite[Theorem 1, (4)]{H-T}). This contradicts Proposition \ref{prop
lim varifold}.

Thus $u_\varepsilon$ converges to $1$ locally uniformly on one side
of $\{x_{n+1}=0\}$, say in $\mathcal{C}_2\cap\{x_{n+1}>0\}$, and to
$-1$ locally uniformly in $\mathcal{C}_2\cap\{x_{n+1}<0\}$. Together
with the previous proposition, this implies
 \[\mbox{dist}_H\left(\{u_\varepsilon=0\}\cap \mathcal{C}_1, \R^n\cap\mathcal{C}_1\right)\to 0.\]
\end{rmk}

The following lemma says that \eqref{small excess} is a consequence
of \eqref{close to plane 1}.
\begin{lem}\label{lem excess small}
Let $u_\varepsilon$ be a sequence of solutions satisfying
\eqref{close to plane 1} and the Modica inequality \eqref{Modica
inequality} in $\mathcal{B}_4$. Then the excess with respect to
$\R^n$, satisfies
\[\lim_{\varepsilon\to0}E(2;0,u_\varepsilon)\to 0.\]
\end{lem}
\begin{proof}
For any $\eta\in C_0^\infty(\mathcal{C}_2)$, take the vector field
$Y=(0,\cdots,0,\eta x_{n+1})$ and substitute it into the stationary
condition \eqref{stationary condition}. This leads to
\begin{eqnarray}\label{stationary with normal variation}
0=\int_{\mathcal{C}_2}&&\left(\frac{\varepsilon}{2}|\nabla u_\varepsilon|^2+\frac{1}{\varepsilon}W(u_\varepsilon)\right)\left(\eta+\frac{\partial\eta}{\partial x_{n+1}}x_{n+1}\right)\\
&&-\eta\nu_{\varepsilon,n+1}^2\varepsilon|\nabla u_\varepsilon|^2
-x_{n+1}\sum_{i=1}^{n+1}\frac{\partial\eta}{\partial
x_i}\nu_{\varepsilon,i}\nu_{\varepsilon,n+1}\varepsilon|\nabla
u_\varepsilon|^2.\nonumber
\end{eqnarray}
By \eqref{discrepancy} and our assumptions on $u_\varepsilon$, both
the measures
\[\left(\frac{\varepsilon}{2}|\nabla u_\varepsilon|^2+\frac{1}{\varepsilon}W(u_\varepsilon)\right)dx,
\ \ \ \nu_{\varepsilon,i}\nu_{\varepsilon,n+1}\varepsilon|\nabla
u_\varepsilon|^2dx\] converge to some measures supported on $\R^n$.
Thus
\[\lim_{\varepsilon\to0}\int_{\mathcal{C}_2}
\left(\frac{\varepsilon}{2}|\nabla u_\varepsilon|^2
+\frac{1}{\varepsilon}W(u_\varepsilon)\right)\frac{\partial\eta}{\partial
x_{n+1}}x_{n+1} -x_{n+1}\sum_{i=1}^{n+1}\frac{\partial\eta}{\partial
x_i}\nu_{\varepsilon,i}\nu_{\varepsilon,n+1}\varepsilon|\nabla
u_\varepsilon|^2=0.\] Substituting this into \eqref{stationary with
normal variation} and applying the Modica inequality \eqref{Modica
inequality}, we finish the proof.
\end{proof}

\begin{rmk}\label{rmk Caccioppoli}
Although we will not use the Caccioppoli type inequality explicitly,
here we show how to use the stationary condition \eqref{stationary
condition} to derive it.

Take a $\psi\in C_0^\infty((-1,1))$ satisfying $0\leq \psi\leq 1$, $\psi\equiv1$ in $(-1/2,1/2)$, $|\psi^\prime|\leq 3$.
For any $\phi\in C_0^\infty(B_1)$, take $\eta(x,x_{n+1})=\phi(x)^2\psi(x_{n+1})^2$ and replace $x_{n+1}$ by $x_{n+1}-\lambda$ in \eqref{stationary with normal variation}, where $\lambda\in(-1,1)$ is an arbitrary constant. By this choice we get
\begin{eqnarray}\label{1.5.1}
0=\int_{\mathcal{C}_1}&&\left[\frac{\varepsilon}{2}|\nabla u_\varepsilon|^2+\frac{1}{\varepsilon}W(u_\varepsilon)\right]
\left[\phi^2\psi^2+2\phi^2\psi\psi^\prime\left(x_{n+1}-\lambda\right)\right]                \nonumber\\
&&-\phi^2\psi^2\nu_{\varepsilon,n+1}^2\varepsilon|\nabla u_\varepsilon|^2-\left(x_{n+1}-\lambda\right)\sum_{i=1}^{n}2\phi\psi^2\frac{\partial\phi}{\partial x_i}\nu_{\varepsilon,i}\nu_{\varepsilon,n+1}\varepsilon|\nabla u_\varepsilon|^2       \\
&&-\left(x_{n+1}-\lambda\right)2\phi^2\psi\psi^\prime\nu_{\varepsilon,n+1}^2\varepsilon|\nabla u_\varepsilon|^2.        \nonumber
\end{eqnarray}

First consider those terms containing $\psi^\prime(x_{n+1})$. Since $\psi^\prime\equiv 0$ in $B_1\times\{|x_{n+1}|<1/2\}$, with the help of Proposition \ref{prop exponential decay}, we get
\[
\int_{\mathcal{C}_1}\left[\frac{\varepsilon}{2}|\nabla
u_\varepsilon|^2+\frac{1}{\varepsilon}W(u_\varepsilon)\right]
2\phi^2\psi\psi^\prime\left(x_{n+1}-\lambda\right)
-\left(x_{n+1}-\lambda\right)2\phi^2\psi\psi^\prime\nu_{\varepsilon,n+1}^2\varepsilon|\nabla
u_\varepsilon|^2 =O(e^{-c\varepsilon^{-1}}).
\]
Substituting this into \eqref{1.5.1} leads to
\begin{eqnarray}\label{1.5.2}
&&\int_{\mathcal{C}_1}\left[\frac{\varepsilon}{2}|\nabla u_\varepsilon|^2-\nu_{\varepsilon,n+1}^2\varepsilon|\nabla u_\varepsilon|^2+\frac{1}{\varepsilon}W(u_\varepsilon)\right]
\phi^2\psi^2         \\
&=&\int_{\mathcal{C}_1}\left(x_{n+1}-\lambda\right)\sum_{i=1}^{n}2\phi\psi^2\frac{\partial\phi}{\partial x_i}\nu_{\varepsilon,i}\nu_{\varepsilon,n+1}\varepsilon|\nabla u_\varepsilon|^2+O(e^{-c/\varepsilon}).        \nonumber
\end{eqnarray}
By the Cauchy inequality,
\begin{eqnarray*}
&&\int_{\mathcal{C}_1}\left(x_{n+1}-\lambda\right)\sum_{i=1}^{n}2\phi\psi^2\frac{\partial\phi}{\partial x_i}\nu_{\varepsilon,i}\nu_{\varepsilon,n+1}\varepsilon|\nabla u_\varepsilon|^2\\
&\leq&\frac{1}{4}\int_{\mathcal{C}_1}\phi^2\psi^2\sum_{i=1}^{n}\nu_{\varepsilon,i}^2\varepsilon|\nabla u_\varepsilon|^2+64\int_{\mathcal{C}_1}|\nabla\phi|^2\psi^2\left(x_{n+1}-\lambda\right)^2\varepsilon|\nabla u_\varepsilon|^2.
\end{eqnarray*}
Substituting this into \eqref{1.5.2}, by noting that
\[\sum_{i=1}^{n}\nu_{\varepsilon,i}^2=1-\left(\nu_\varepsilon\cdot e_{n+1}\right)^2,\]
and
\[\frac{\varepsilon}{2}|\nabla u_\varepsilon|^2-\nu_{\varepsilon,n+1}^2\varepsilon|\nabla u_\varepsilon|^2+\frac{1}{\varepsilon}W(u_\varepsilon)\geq\left[1-\left(\nu_\varepsilon\cdot e_{n+1}\right)^2\right]\varepsilon|\nabla u_\varepsilon|^2,\]
we obtain the following Caccioppoli type inequality
\begin{equation}\label{Caccioppoli inequality}
\int_{\mathcal{C}_1}\phi^2\psi^2\left[1-\left(\nu_\varepsilon\cdot
e_{n+1}\right)^2\right]\varepsilon|\nabla u_\varepsilon|^2 \leq
2^8\int_{\mathcal{C}_1}|\nabla\phi|^2\psi^2\left(x_{n+1}-\lambda\right)^2\varepsilon|\nabla
u_\varepsilon|^2 +Ce^{-c\varepsilon^{-1}}.
\end{equation}

\end{rmk}

\begin{rmk}
Since we only have a control on $1-(\nu_\varepsilon\cdot
e_{n+1})^2$, in view of Remark \ref{rmk 1}, $\nu_\varepsilon$ may be
close to $e_{n+1}$ or $-e_{n+1}$. To exclude one of these two
possibilities, we need to use the unit density assumption
\eqref{close to plane 1}. A subtle point here is that, without such
a assumption, we cannot say that $1-\nu_\varepsilon\cdot e_{n+1}$
(or $1+\nu_\varepsilon\cdot e_{n+1}$) is small everywhere. This is
related to the possible interface foliation (and consequently the
higher multiplicity of the limit varifold $V$), see the examples
constructed by del Pino-Kowalczyk-Wei-Yang in \cite{DKWY}.
\end{rmk}

\section{Lipschitz approximation}
\numberwithin{equation}{section}
 \setcounter{equation}{0}

Let
\[f_\varepsilon(x)=\int_{-1}^1\left[1-\left(\nu_\varepsilon(x,x_{n+1})\cdot e_{n+1}\right)^2\right]
\varepsilon|\nabla u_\varepsilon(x,x_{n+1})|^2dx_{n+1}.\]
By Lemma \ref{lem excess small}, $f_\varepsilon\to 0$ in $L^1(B_1)$.
Consider the Hardy-Littlewood maximal function
\[Mf_\varepsilon(x):=\sup_{r\in(0,1)}r^{-n}\int_{B_r(x)}f_\varepsilon(y)dy.\]
For any $l>0$, by the weak $L^1$ estimate, there exists a universal
constant $C$ such that
\begin{equation}\label{weak L1 bound}
\mathcal{H}^n\left(\{Mf_\varepsilon\geq l\}\cap
B_1\right)\leq\frac{C}{l}\|f_\varepsilon\|_{L^1(B_1)}
=C\frac{\delta_\varepsilon^2}{l}.
\end{equation}
Denote the set $B_1\setminus\{Mf_\varepsilon\geq l\}$ by
$W_\varepsilon$. (Its dependence on the constant $l$ will not be
written explicitly.) Note that since the integrand in the definition
of $f_\varepsilon$ and hence $f_\varepsilon(x)$ are continuous
functions, $W_\varepsilon$ is an open set.

Given $b\in(0,1)$ and $l>0$, we say a point $X\in \{|u_\varepsilon|<1-b\}\cap\mathcal{C}_1$ is good, if
\[\sup_{0<r<1}r^{-n}\int_{\mathcal {B}_r(X)}\left[1-\left(\nu_\varepsilon\cdot e_{n+1}\right)^2\right]\varepsilon
|\nabla u_\varepsilon|^2dX<l.\] Good points form a set
$A_\varepsilon$ and we let
$B_\varepsilon=(\{|u_\varepsilon|<1-b\}\cap\mathcal{C}_1)\setminus
A_\varepsilon$ be the set of bad points. Note that since
$\left[1-\left(\nu_\varepsilon(x,x_{n+1})\cdot
e_{n+1}\right)^2\right] \varepsilon|\nabla
u_\varepsilon(x,x_{n+1})|^2$ is continuous, $A_\varepsilon$ is an
open set and $B_\varepsilon$ is relatively closed in
$\{|u_\varepsilon|<1-b\}\cap\mathcal{C}_1$. Clearly
$W_\varepsilon\subset\Pi(A_\varepsilon)$.

Similar to the weak $L^1$ estimate for Hardy-Littlewood maximal
function, $B_\varepsilon$ is small in the following sense.
\begin{lem}\label{lem 4.3}
There exists a universal constant $C$ such that
\[\mu_\varepsilon(B_\varepsilon)\leq C\frac{\delta_\varepsilon^2}{l}.\]
\end{lem}
\begin{proof}
For any $X\in B_\varepsilon$, by definition there exists an $r_X\in
(0,1)$ satisfying
\[r_X^n\leq\frac{1}{l}\int_{\mathcal{B}_{r_X}(X)}\left[1-\left(\nu_\varepsilon\cdot e_{n+1}\right)^2\right]\varepsilon
|\nabla u_\varepsilon|^2dX.\] By Vitali covering lemma, choose a
countable set of $X_i\in B_\varepsilon$ such that
$\mathcal{B}_{r_i}(X_i)$ (here $r_i:=r_{X_i}$) are disjoint, and
\[B_\varepsilon\subset \bigcup_{i}\mathcal{B}_{5r_i}(X_i).\]
Then
\begin{eqnarray*}
\mu_\varepsilon(B_\varepsilon)&\leq&\sum_i\mu_\varepsilon(\mathcal{B}_{5r_i}(X_i))\\
&\leq&C\sum_ir_i^n \ \ \ \ \ \ \ \ \ \ \ \ \ \ \ \ \ \ \ (\mbox{by \eqref{energy bound on ball}}) \\
&\leq&Cl^{-1}\int_{\mathcal{B}_{r_i}(X_i)}\left[1-\left(\nu_\varepsilon\cdot
e_{n+1}\right)^2\right]\varepsilon |\nabla u_\varepsilon|^2dX\\
&\leq &Cl^{-1}\delta_\varepsilon^2.  \qedhere
\end{eqnarray*}
\end{proof}

Another fact about $B_\varepsilon$ is
\begin{lem}\label{lem 3.2}
There exists a universal constant $C$ such that $\mathcal{H}^n(\Pi
(B_\varepsilon))\leq Cl^{-1}\delta_\varepsilon^2$.
\end{lem}
\begin{proof}
This is because $\Pi(B_\varepsilon)\subset B_1\setminus
W_\varepsilon$. Hence \eqref{weak L1 bound} applies.
\end{proof}

Next we show that in $A_\varepsilon$,
level sets of $u_\varepsilon$ are essentially Lipschitz graphs.
\begin{lem}\label{lem 5.3}
Given a constant $b\in(0,1)$, if $l$ is small enough, for any
$t\in(-1+b,1-b)$, $\{u_\varepsilon=t\}\cap A_\varepsilon$ can be
represented locally by
 a Lipschtz graph
$\{x_{n+1}=h_\varepsilon^t(x)\}$. The Lipschitz constant of
$h_\varepsilon^t$ is controlled by a constant $c_0(b,l)$ depending
on $b$ and $l$, which satisfies $\lim_{l\to0}c_0(b,l)=0$.
\end{lem}
\begin{proof}
Fix a point $X_0\in A_\varepsilon$ with $u_\varepsilon(X_0)=t$. After a rescaling
\[v(X)=u_\varepsilon(X_0+\varepsilon X),\]
we are in the situation that
\begin{equation}\label{3.1}
\Delta v=W^\prime(v),\ \ \ \mbox{in}\ \mathcal{B}_{\varepsilon^{-1}},
\end{equation}
\begin{equation}\label{3.2}
\int_{\mathcal{B}_R(0)}\frac{1}{2}|\nabla v|^2+W(v)\leq CR^n,\ \ \forall \ R\in(0,\varepsilon^{-1}),
\end{equation}
\begin{equation}\label{3.3}
\int_{\mathcal{B}_1(0)}\sum_{i=1}^n\left(\frac{\partial v}{\partial x_i}\right)^2\leq l.
\end{equation}
We claim that there exists an $l_0$ small such that for all $l\leq
l_0$, there exist two constants $c_1(b,l)\in(0,1/2)$ and $c_2(b)$ so
that
\begin{equation}\label{Lip graph at O(1)}
\Big|\frac{\partial v}{\partial x_{n+1}}\Big|\geq
\left(1-c_1(b,l)\right)|\nabla v|\geq c_2(b)\ \ \mbox{in}\
\mathcal{B}_1.
\end{equation}

Assume by the contrary, there exists a sequence of $v_i$ satisfying
all of these conditions \eqref{3.1}-\eqref{3.3} with $l$ replaced by
$l_i$, which goes to $0$ as $i\to0$. By standard elliptic estimates
and Arzel\`{a}-Ascoli theorem, $v_i$ converges to a function $v$ in
$C^2_{loc}(\R^{n+1})$. $v$ is still a solution of \eqref{3.1} in
$\mathbb{R}^{n+1}$. Because $|v|\leq 1$ and
\[|v(0)|=\lim_{i\to+\infty}|v_i(0)|\leq 1-b,\]
by the strong maximum principle, $|v|<1$ in $\R^{n+1}$. After
passing to the limit in \eqref{3.3} (where $l$ is replaced by $l_i$)
and \eqref{3.2}, we see $v(X)\equiv g(x_{n+1}+t)$ for some $t\in\R
$. (For more details about this derivation, see the proof of Lemma
\ref{lem B2}.) Then by \eqref{1-d solution 2},
\[\Big|\frac{\partial v}{\partial x_{n+1}}(X)\Big|=|\nabla v(X)|=\sqrt{2W(v(X))}\geq c(b)\  \ \ \ \mbox{in}\ \mathcal{B}_1.\]
Thus for all $i$ large, $v_i$ satisfies \eqref{Lip graph at O(1)}.
This also implies that $c_1(b,l)$ converges to $0$ as $l\to0$.


By \eqref{Lip graph at O(1)}, the level set
$\{v=v(0)\}\cap\mathcal{B}_1(0)$ is locally a Lipschitz graph in the
form $\{x_{n+1}=h(x)\}$, with its Lipschitz constant $c_0(b,l)\leq
2c_1(b,l)$. Coming back to $u_\varepsilon$ we finish the proof.
\end{proof}

By Lemma \ref{lem B2} and \cite[Proposition 5.6]{H-T}, for any $L>0$
and $X=(x,x_{n+1})\in A_\varepsilon$, if we have chosen $l$
sufficiently small,
\begin{equation}\label{lem behavior at O(e)}
\Pi^{-1}(x)\cap\{u_\varepsilon=u_\varepsilon(X)\}\cap\mathcal{B}_{L\varepsilon}(X)=\{X\}.
\end{equation}

The above results only provide a clear picture of
$\{u_\varepsilon=t\}\cap A_\varepsilon$ at $O(\varepsilon)$ scales.
By the unit density assumption \eqref{close to plane 1}, we can
further claim that
\begin{lem}\label{lem one graph}
Given $b\in(0,1)$, for every $t\in(-1+b,1-b)$ and $x\in
\Pi(A_\varepsilon)$, in $\Pi^{-1}(x)\cap\{u_\varepsilon=t\}$ there
exists exactly one point, $(x, h_\varepsilon^t(x))$ (as in Lemma
\ref{lem 5.3}). Moreover, $h_\varepsilon^t$ is Lipschitz on
$\Pi(A_\varepsilon)$.
\end{lem}
This lemma is a consequence of the following lemma, provided that we
have chosen first $R_0$ large in the following lemma and then $l$
sufficiently small in the definition of $A_\varepsilon$. The proof
of Lemma \ref{lem one graph} will be completed after Remark \ref{rmk
5.7}.

\begin{lem}\label{lem continuation from O(1) to O(e)}
For any $b\in(0,1)$ and $\delta>0$, there exist three constants
$R_0$ large and $\tau_1, l_2$ small so that the following holds.
Suppose that $u_\varepsilon$ is a solution of \eqref{equation} in
$\mathcal{B}_{R_0}$, where $\varepsilon\leq 1$, satisfying
$|u_\varepsilon(0)|\leq 1-b$, the Modica inequality \eqref{Modica
inequality},
\[R_0^{-n}\int_{\mathcal{B}_{R_0}}\frac{\varepsilon}{2}|\nabla u_\varepsilon|^2+\frac{1}{\varepsilon}W(u_\varepsilon)\leq \left(1+\tau_1\right)\sigma_0\omega_n,\]
 and
\[R_0^{-n}\int_{\mathcal{B}_{R_0}}\left[1-\left(\nu_\varepsilon\cdot e_{n+1}\right)^2\right]
\varepsilon|\nabla u_\varepsilon|^2\leq l_2,\] then
$\{u_\varepsilon=u_\varepsilon(0)\}\cap \mathcal{B}_1$ is contained
in the $\delta$-neighborhood of $\R^n\cap\mathcal{B}_1$.
\end{lem}
This result can be seen as a quantitative version of the
multiplicity one property for the limit varifold $V$.

\begin{proof}
Assume that there exists a sequence of solutions $u_k$ satisfying
the assumptions in this lemma, with $\varepsilon$ replaced by
$\varepsilon_k\in(0,1]$,
\begin{equation}\label{3.4.0}
R_0^{-n}\int_{\mathcal{B}_{R_0}}\frac{\varepsilon_k}{2}|\nabla
u_k|^2+\frac{1}{\varepsilon_k}W(u_{\varepsilon_k})\leq
\left(1+\tau_k\right)\sigma_0\omega_n,
\end{equation}
where $\tau_k\to0$, and
\begin{equation}\label{3.4}
R_0^{-n}\int_{\mathcal{B}_{R_0}}\varepsilon_k\sum_{i=1}^n\left(\frac{\partial
u_k}{\partial x_i}\right)^2dX\to0,
\end{equation}
but there exists $X_k=(x_k,x_{n+1,k})\in
\{u_k=u_k(0)\}\cap\mathcal{B}_1$ with $|x_{n+1,k}|\geq\delta$. The
constant $R_0$ will be determined below. Without loss of generality,
assume that $X_k$ converges to some point
$X_\infty=(x_\infty,x_{n+1,\infty})\in \mathcal{B}_1$ with
$|x_{n+1,\infty}|\geq\delta$.

The proof is divided into two cases.\\
{\bf Case 1.} $\varepsilon_k$ converges to some $\varepsilon_0>0$
(after subtracting a subsequence).

 By standard elliptic estimates and Arzel\`{a}-Ascoli theorem, $u_k$ converges to a function $u_\infty$
 in $C^2(\mathcal{B}_{R_0-1})$. Because $|u_k|<1$,
$|u_\infty|\leq 1$ in $\mathcal{B}_{R_0-1}$. $u_\infty$ is a
solution of \eqref{equation} with $\varepsilon$ replaced by
$\varepsilon_0$. Since $|u_\infty(0)|\leq 1-b<1$, by the strong
maximum principle, $|u_\infty|<1$ strictly in $\mathcal{B}_{R_0-1}$.
Passing to the limit in \eqref{3.4.0} leads to
\begin{equation}\label{5.001}
\int_{\mathcal{B}_{R_0-1}}\frac{\varepsilon_0}{2}|\nabla
u_\infty|^2+\frac{1}{\varepsilon_0}W(u_\infty)\leq
\sigma_0\omega_n\left(R_0-1\right)^n.
\end{equation}
Since $\varepsilon_0\leq 1$, we cannot have $u_\infty\equiv
u_\infty(0)$, because otherwise
\[\int_{\mathcal{B}_{R_0-1}}\frac{\varepsilon_0}{2}|\nabla u_\infty|^2+\frac{1}{\varepsilon_0}W(u_\infty)
\geq
\frac{1}{\varepsilon_0}W(u_\infty(0))\omega_{n+1}\left(R_0-1\right)^{n+1}>\sigma_0\omega_n\left(R_0-1\right)^n,\]
if we choose $R_0$ large to fufill the last inequality. (It depends
only on $b$, the dimension $n$ and the potential $W$.)

By passing to the limit in \eqref{3.4} we obtain
\[R_0^{-n}\int_{\mathcal{B}_{R_0}}\varepsilon_0\sum_{i=1}^n\left(\frac{\partial u_\infty}{\partial x_i}\right)^2dX=0.\]
Thus $u_\infty(X)\equiv \tilde{u}(x_{n+1})$.

$\tilde{u}$ is a one dimensional solution. By \eqref{5.001}, we have
\begin{equation}\label{1d energy bound}
\int_{-R_0/2}^{R_0/2}\frac{\varepsilon_0}{2}\Big|\frac{d
\tilde{u}}{dt}\Big|^2+\frac{1}{\varepsilon_0}W(\tilde{u})dt \leq
C\sigma_0,
\end{equation}
 for some universal constant $C$ independent of
$R_0$. Since $\varepsilon_0\leq 1$, we claim that if $R_0$ is
sufficiently large,
\begin{equation}\label{1d monotonicity}
\frac{\partial u_\infty}{\partial x_{n+1}}(X)\neq0,\ \ \ \mbox{in}\
(-1,1).
\end{equation}
This can be proved by a contradiction argument, using the following
fact: Except the heteroclinic solution $g$, all the other solutions
of \eqref{equation 0} in $\R^1$ are periodic, hence their energy on
$\R$ is infinite.

 By \eqref{1d monotonicity}, $u_\infty\neq u_\infty(0)$ in
$\mathcal{B}_1\setminus\R^n$. However, by the convergence of $X_k$
and uniform convergence of $u_k$, $u_\infty(X_\infty)=u_\infty(0)$.
Because $X_\infty\in \mathcal{B}_1\setminus\R^n$, this is a a
contradiction.

{\bf Case 2.} $\varepsilon_k\to0$.

Let $V_k$ be the varifold associated to $u_k$ as defined in Section
4. For any $\eta\in C_0^\infty(\mathcal{B}_{R_0})$, let
\[\Phi(X,S)=\eta(X)<Se_{n+1},e_{n+1}>\in C_0^\infty(\mathcal{B}_{R_0}\times G(n)).\]
By \eqref{3.4},
\[<V_k,\Phi>=\int_{\mathcal{B_{R_0}}}\eta(X)\varepsilon_k\sum_{i=1}^n\left(\frac{\partial u_k}{\partial x_i}\right)^2dX\to0.\]
Let $V_\infty$ be the limit varifold of $V_k$, which is stationary
rectifiable with unit density by Hutchinson-Tonegawa theorem. Then
\[0=<V,\Phi>=\int \eta(X)<Te_{n+1},e_{n+1}>d\|V_\infty\|,\]
where $T$ is the weak tangent plane of $V$ at $X$. Hence $T=\R^n$
$\|V_\infty\|$-a.e. and $V_\infty=\sigma_0\sum_j i(T_j)$ in
$B_{R_0/2}\times(-R_0/2,R_0/2)$, where $T_j=\R^n\times \{(0, t_j)\}$
for some $j$ and $i(T_j)$ is the standard varifold associated to it
with unit density. By our assumptions, there are at least two
components, say $T_0$ and $T_1$, containing the points $0$ and
$X_\infty$ respectively.

However, passing to the limit in \eqref{3.4.0} gives
\[\|V\|(\mathcal{B}_{R_0})\leq\sigma_0\omega_nR_0^n=\sigma_0\|T_0\|(\mathcal{B}_{R_0}).\]
Thus we cannot have any more components other than $T_0$. This also
leads to a contradiction.
\end{proof}
\begin{rmk}\label{rmk 5.7}
It will be useful to write the dependence of $\delta$ and $l_2$
reversely as $\delta=c_2(l_2)$. This function is a modulus of
continuity, i.e. a non-decreasing function satisfying
$\lim_{l_2\to0}c_2(l_2)=0$.
\end{rmk}

For any $X_0=(x_0,x_{0,n+1})\in A_\varepsilon$ and
$r\in(\varepsilon,1/R_0)$, applying the previous lemma to
\[\tilde{u}_{\varepsilon,r}(X)=u_\varepsilon(X_0+r X)\]
gives
\begin{equation}\label{cone of monotonicity}
\{u_\varepsilon=u_\varepsilon(X_0)\}\cap
\left(\mathcal{B}_{1/2}(X_0)\setminus
\mathcal{B}_\varepsilon(X_0)\right)\subset\{|x_{n+1}-x_{0,n+1}|\leq
c_2(l)|x-x_0|\}.
\end{equation}
Together with \eqref{lem behavior at O(e)}, this implies that for
every $t\in(-1+b,1-b)$ and $x\in \Pi(A_\varepsilon)$, there exists
at most one point in $\Pi^{-1}(x)\cap\{u_\varepsilon=t\}$.

On the other hand, by Remark \ref{rmk Hausdorff convergence}, for
each $x\in B_1$,
\[u_\varepsilon(x,1)>1-b, \ \ \ \ u_\varepsilon(x,-1)<-1+b.\]
Thus, by continuity of $u_\varepsilon$, there must exist one
$x_{n+1}\in(-1,1)$ satisfying $u_\varepsilon(x,x_{n+1})=t$.

 In conclusion, for any $x\in\Pi(A_\varepsilon)$, there exists a unique point $(x,x_{n+1})\in \Pi^{-1}(x)\cap\{u_\varepsilon=t\}$.
Combining \eqref{lem behavior at O(e)} with \eqref{cone of
monotonicity}, it can be seen that $h_\varepsilon^t$ is Lipschitz on
$\Pi(A_\varepsilon)$. This completes the proof of Lemma \ref{lem one
graph}.


Recall that we have assumed $u_\varepsilon>0$ in
$\mathcal{C}_1\cap\{x_{n+1}>h\}$ and $u_\varepsilon<0$ in
$\mathcal{C}_1\cap\{x_{n+1}<-h\}$ for some $h>0$, see Remark
\ref{rmk Hausdorff convergence}. (This $h$ can be made arbitrarily
small as $\varepsilon\to0$.) Hence for any
$r\in(\varepsilon,1/{R_0})$, by continuous dependence on $r$,
\eqref{cone of monotonicity} can be improved to
\begin{equation}\label{cone of monotonicity 1}
\{x_{n+1}-x_{0,n+1}>c_2(l)|x-x_0|\}\cap
\left(\mathcal{B}_{1/2}(X_0)\setminus
\mathcal{B}_\varepsilon(X_0)\right)\subset\{u_\varepsilon>u_\varepsilon(X_0)\}.
\end{equation}
When $r=\varepsilon$, combining this with Lemma \ref{lem behavior at
O(e)}, we obtain
\begin{equation}\label{derivative monotinicity}
\frac{\partial u_\varepsilon}{\partial x_{n+1}}(X)\geq
\left(1-c_1(b,l)\right)|\nabla
u_\varepsilon(X)|\geq\frac{c(b)}{\varepsilon}, \ \forall \ X\in
A_\varepsilon.
\end{equation}

In the following, for $t\in(-1+b,1-b)$, we denote the Lipschitz
functions by $h_\varepsilon^t$. By \eqref{lem behavior at O(e)}, the
definition domains of $h_\varepsilon^t$ can be made to be a common
one, $\Pi(A_\varepsilon)$. By \eqref{derivative monotinicity},
$h^t_\varepsilon$ is strictly increasing in $t\in(-1+b,1-b)$.

In the above construction, $h_\varepsilon^t$ is only defined on a
subset of $B_1$, but we can extend it to $B_1$ without increasing
its Lipschitz constant by letting (see for example \cite[Theoerm
3.1.3]{Lin})
\begin{equation}\label{Lip extension}
h_\varepsilon^t(x):=\inf_{y\in
\Pi(A_\varepsilon)}\left(h_\varepsilon^t(y)+c_3(b,l)|y-x|\right),\ \
\forall x\in B_1.
\end{equation}
Here $c_3(b,l)=\max\{c_0(b,l),c_2(l)\}$. This extension preserves
the monotonicity of $h^t_\varepsilon$ in $t$.

In Section 7 and 8, $b$ and hence $l$ may be decreased further. Thus
it is worthy to note the dependence of these Lipschitz functions on
$b$ and $l$.
\begin{rmk}\label{rmk 2}
If $l$ is decreased, the definition domain of $h^t_\varepsilon$ is
also decreased. But on the common part, these two constructions give
the same function. If we define two families by choosing two
$0<b_1<b_2<1$, these two families also coincide on $(-1+b_2,
1-b_2)$.
\end{rmk}

Notation: $D_\varepsilon=\Pi(A_\varepsilon)$.

In the following it will be useful to keep in mind that, on
$\{u_\varepsilon=t\}\cap A_\varepsilon$,
\begin{equation}\label{x,t coordinates}
\frac{\partial u_\varepsilon}{\partial x_{n+1}}=\left(\frac{\partial
h_\varepsilon^t}{\partial t}\right)^{-1},\ \ \ \frac{\partial
u_\varepsilon}{\partial x_i}=-\left(\frac{\partial
h_\varepsilon^t}{\partial t}\right)^{-1}\frac{\partial
h_\varepsilon^t}{\partial x_i},\ \ 1\leq i\leq n.
\end{equation}

\section{Estimates on $h_\varepsilon^t$}
\numberwithin{equation}{section}
 \setcounter{equation}{0}

First we give an $H^1$ bound.
\begin{lem}\label{lem H1 bound}
There exists a constant $C(b)$ independent of $\varepsilon$ such
that,
\[\int_{-1+b}^{1-b}\int_{B_1}|\nabla h_\varepsilon^t|^2dxdt\leq C(b)\delta_\varepsilon^2.\]
\end{lem}
\begin{proof}
By Lemma \ref{lem 3.2}, $\mathcal{H}^n(B_1\setminus
D_\varepsilon)\leq C\delta_\varepsilon^2$. The construction in the
previous section implies that $|\nabla h^t_\varepsilon|\leq
c_3(b,l)$ in $B_1$ for all $t\in(-1+b,1-b)$. Hence
\begin{equation}\label{4.1}
\int_{-1+b}^{1-b}\int_{B_1\setminus D_\varepsilon}|\nabla
h_\varepsilon^t|^2dxdt\leq C\delta_\varepsilon^2.
\end{equation}
Next, by noting that on $A_\varepsilon$, $\varepsilon|\nabla
u_\varepsilon|\geq c(b)$ (see \eqref{Lip graph at O(1)}), we have
\begin{eqnarray*}
\delta_\varepsilon^2&\geq&\int_{A_\varepsilon}\left[1-\left(\nu_\varepsilon\cdot e_{n+1}\right)^2\right]\varepsilon|\nabla u_\varepsilon|^2dX\\
&=&\int_{-1+b}^{1-b}\left[\int_{\{u_\varepsilon=t\}\cap
A_\varepsilon} \left(1-\left(\nu_\varepsilon\cdot
e_{n+1}\right)^2\right)\varepsilon|\nabla
u_\varepsilon|d\mathcal{H}^n\right]dt
\quad\mbox{(by the co-area formula)}\\
&\geq&c(b)\int_{-1+b}^{1-b}\left(\int_{D_\varepsilon}\left[1-\frac{1}{1+|\nabla h_\varepsilon^t|^2}\right]
\sqrt{1+|\nabla h_\varepsilon^t|^2}dx\right)dt \quad\mbox{(by \eqref{x,t coordinates})}\\
&\geq&c(b)\int_{-1+b}^{1-b}\left(\int_{D_\varepsilon}|\nabla
h_\varepsilon^t|^2dx\right)dt,
\end{eqnarray*}
if we have chosen $l$ so small that the Lipschitz constants of
$h^t_\varepsilon$, $c_3(b,l)\leq 1/2$.
\end{proof}

With this lemma in hand, we first choose a
$t_\varepsilon\in(-1+b,1-b)$ such that
\[\int_{B_1}|\nabla h^{t_\varepsilon}_\varepsilon|^2dx\leq C(b)\delta_\varepsilon^2,\]
and then take a $\lambda_\varepsilon$ so that the function defined by
\begin{equation}\label{defintion blow up sequence}
\bar{h}_\varepsilon:=\frac{1}{\delta_\varepsilon}h_\varepsilon^{t_\varepsilon}-\lambda_\varepsilon,
\end{equation}
satisfies $\int_{B_1}\bar{h}_\varepsilon=0$. By this choice and the
Poincar\'{e} ienquality,
\begin{equation}\label{6.3}
\int_{B_1}\bar{h}_\varepsilon(x)^2dx\leq C\int_{B_1}|\nabla
h_\varepsilon(x)|^2dx\leq C(b).
\end{equation}
Thus we can assume, after passing to a subsequence of
$\varepsilon\to0$, $\bar{h}_\varepsilon$ converges to a function
$\bar{h}$, weakly in $H^1(B_1)$ and strongly in $L^2(B_1)$.

Let
\[\bar{h}^t_\varepsilon:=\frac{1}{\delta_\varepsilon}h_\varepsilon^{t}-\lambda_\varepsilon.\]
In $D_\varepsilon$,
\begin{equation}\label{5.8}
0\leq\frac{\partial h^t_\varepsilon}{\partial t}=\left(\frac{\partial u_\varepsilon}{\partial x_{n+1}}\right)^{-1}\leq C(b)\varepsilon,
\end{equation}
with a constant $C(b)$ depending only on $b$.
Hence
\begin{equation}\label{5.8.1}
0\leq h^{1-b}_\varepsilon-h^{-1+b}_\varepsilon\leq C(b)\varepsilon, \ \ \ \mbox{in}\ D_\varepsilon.
\end{equation}
This also holds for $x\in B_1\setminus D_\varepsilon$ by \eqref{Lip extension}.

Hence for any $-1+b<t_1<t_2<1-b$,
\begin{equation}\label{4.2}
\int_{B_1}\left(h_\varepsilon^{t_1}-h_\varepsilon^{t_2}\right)^2
\leq C(b)\varepsilon^2.
\end{equation}
Because $\delta_\varepsilon\gg\varepsilon$, for any sequence of $\tilde{t}_\varepsilon\in(-1+b,1-b)$, $\bar{h}^{\tilde{t}_\varepsilon}_\varepsilon$ still converges to $\bar{h}$ in $L^2(B_1)$.

Since $\delta_\varepsilon^{-1}\nabla h^t_\varepsilon$ are uniformly
bounded in $L^2(B_1\times (-1+b,1-b),\mathbb{R}^n)$, it can be
assumed to converge weakly to some limit in  $L^2(B_1\times
(-1+b,1-b),\mathbb{R}^n)$. By the above discussion, this limit must
be $\nabla\bar{h}$.

By Remark \ref{rmk 2}, $\bar{h}$ is independent of the choice of
$b$. Hence we have a universal constant $C$,
which is independent of
$b$ and $l$, such that
\begin{equation}\label{H1 bound on the blow up limit}
\int_{B_1}|\nabla\bar{h}|^2+\bar{h}^2\leq C.
\end{equation}

Concerning the size of $\lambda_\varepsilon$, we have
\begin{lem}\label{lem 5.2}
$\lim_{\varepsilon\to 0}|\lambda_\varepsilon\delta_\varepsilon|=0$.
\end{lem}
\begin{proof}
Note that
\begin{equation}\label{5.9}
\lambda_\varepsilon\delta_\varepsilon=\int_{B_1}h_\varepsilon^{t_\varepsilon}.
\end{equation}

By Proposition \ref{prop exponential decay},
\[\lim_{\varepsilon\to0}\sup_{\mathcal{C}_1\cap\{|u_\varepsilon|\leq 1-b\}}|x_{n+1}|=0.\]
Thus
\begin{equation}\label{5.10}
\lim_{\varepsilon\to0}\sup_{t\in(-1+b,1-b)}\sup_{x\in D_\varepsilon}|h^t_\varepsilon(x)|=0.
\end{equation}

For any $x\in B_1\setminus D_\varepsilon$, by Lemma \ref{lem 3.2},
\[\mbox{dist}(x,D_\varepsilon)\leq Cl^{-\frac{1}{n}}\delta_\varepsilon^{\frac{2}{n}}.\]
Because the Lipschitz constant of $h^t_\varepsilon$ is smaller than
$c_3(b,l)\leq 1$, we obtain
\[\sup_{t\in(-1+b,1-b)}\sup_{x\in B_1\setminus D_\varepsilon}|h^t_\varepsilon(x)|\leq
\sup_{t\in(-1+b,1-b)}\sup_{x\in
D_\varepsilon}|h^t_\varepsilon(x)|+C(l)\delta_\varepsilon^{\frac{2}{n}}.\]
Combining this with \eqref{5.10} we see
\[\lim_{\varepsilon\to0}\sup_{t\in(-1+b,1-b)}\sup_{x\in B_1}|h^t_\varepsilon(x)|=0.\]
Substituting this into \eqref{5.9} we finish the proof.
\end{proof}

Next we establish a bound for the height excess
\[\int_{\mathcal{C}_{3/4}}\left(x_{n+1}-\lambda_\varepsilon\delta_\varepsilon\right)^2\varepsilon|\nabla u_\varepsilon|^2.\]
 This can be viewed
as a Poincar\'{e} inequality on the varifold $V_\varepsilon$. (The
Caccioppoli type inequality in Remark \ref{rmk Caccioppoli} is a
reverse Poincar\'{e} inequality.)
\begin{lem}\label{lem Poincare inequality}
There exists a universal constant $C$ such that
\begin{equation}\label{L2 bound}
\int_{\mathcal{C}_{3/4}}\left(x_{n+1}-\lambda_\varepsilon\delta_\varepsilon\right)^2\varepsilon|\nabla u_\varepsilon|^2
\leq C\delta_\varepsilon^2.
\end{equation}
\end{lem}
\begin{proof}
The proof is divided into two steps. In the following we shall fix
two numbers $0<b_2<b_1<1$ so that $W^{\prime\prime}\geq \kappa$ in
$(-1,-1+b_1)\cup(1-b_1,1)$.

{\bf Step 1.} Here we give an estimate in the part
$\{|u_\varepsilon|< 1-b_2\}\cap\mathcal{C}_1$, i.e.
\begin{equation}\label{L2 bound 1}
\int_{\{|u_\varepsilon|<1-b_2\}\cap\mathcal{C}_1}\left(x_{n+1}-\lambda_\varepsilon\delta_\varepsilon\right)^2\varepsilon|\nabla u_\varepsilon|^2
\leq C\delta_\varepsilon^2.
\end{equation}

First, by \eqref{defintion blow up sequence} and \eqref{4.2},
\begin{equation}\label{6.4}
\int_{-1+b_2}^{1-b_2}\int_{B_1}\left(h_\varepsilon^t-\lambda_\varepsilon\delta_\varepsilon\right)^2dxdt\leq C\delta_\varepsilon^2.
\end{equation}
Then by a change of variables, the gradient bound \eqref{gradient
bound} and the Lipschitz bound on $h_\varepsilon^t$, we obtain
\begin{eqnarray}\label{6.5}
\int_{A_\varepsilon}\left(x_{n+1}-\lambda_\varepsilon\delta_\varepsilon\right)^2\varepsilon|\nabla u_\varepsilon|^2
&=&\int_{-1+b_2}^{1-b_2}\int_{D_\varepsilon}\left(h_\varepsilon^t-\lambda_\varepsilon\delta_\varepsilon\right)^2
\left(1+|\nabla h_\varepsilon^t|^2\right)\varepsilon\frac{\partial u_\varepsilon}{\partial x_{n+1}}dxdt   \nonumber\\
&\leq&C\int_{-1+b_2}^{1-b_2}\int_{D_\varepsilon}\left(h_\varepsilon^t-\lambda_\varepsilon\delta_\varepsilon\right)^2
dxdt\\
&\leq& C\delta_\varepsilon^2.\nonumber
\end{eqnarray}

In $B_\varepsilon$, by Lemma \ref{lem 4.3} and Lemma \ref{lem 5.2},
\begin{equation}\label{6.6}
\int_{B_\varepsilon}\left(x_{n+1}-\lambda_\varepsilon\delta_\varepsilon\right)^2\varepsilon|\nabla
u_\varepsilon|^2 \leq
C\left[\sup_{\{|u_\varepsilon|<1-b\}}\left(x_{n+1}-\lambda_\varepsilon\delta_\varepsilon\right)^2\right]\mu_\varepsilon(B_\varepsilon)
\leq C\delta_\varepsilon^2.
\end{equation}
Combining \eqref{6.5} and \eqref{6.6} we get \eqref{L2 bound 1}.

{\bf Step 2.} We claim that in the part $\{|u_\varepsilon|>1-b_2\}\cap\mathcal{C}_{3/4}$,
\begin{equation}\label{L2 bound 2}
\int_{\{|u_\varepsilon|>1-b_2\}\cap\mathcal{C}_{3/4}}\left(x_{n+1}-\lambda_\varepsilon\delta_\varepsilon\right)^2\varepsilon|\nabla u_\varepsilon|^2
\leq C\delta_\varepsilon^2.
\end{equation}

 Choose a function $\zeta\in C^\infty(\R)$, satisfying
\begin{equation*}
\left\{
\begin{aligned}
 &\zeta(t)\equiv 1, \ &\mbox{in}\ \{|t|>1-b_2\}, \\
  &\zeta(t)\equiv 0, \ &\mbox{in}\ \{|t|<1-b_1\}, \\
 &|\zeta^\prime|\leq\frac{2}{b_1-b_2},\quad |\zeta^{\prime\prime}|\leq \frac{8}{\left(b_1-b_2\right)^2} \ \ &\mbox{in}\ \{1-b_1\leq|t|\leq1-b_2\}.
                          \end{aligned} \right.
\end{equation*}
We also fix a function $\eta\in
C_0^\infty(B_1\times\{|x_{n+1}|<4/3\})$ so that $0\leq\eta\leq 1$,
$\eta\equiv 1$ in $\mathcal{C}_{3/4}$.

It can be directly checked that
\begin{equation}\label{equation for gradient estiamtes}
\Delta\left(\varepsilon|\nabla u_\varepsilon|^2\right)\geq\frac{\kappa}{\varepsilon^2}\left(\varepsilon|\nabla u_\varepsilon|^2\right),\ \ \ \mbox{in}\ \ \{|u_\varepsilon|>1-b_1\}.
\end{equation}
Multiplying this equation by $\left(x_{n+1}-\lambda_\varepsilon\delta_\varepsilon\right)^2\eta\zeta(u_\varepsilon)$ and integrating by parts, we obtain
\begin{eqnarray}\label{6.7}
&&\int_{\mathcal{B}_2}\left(x_{n+1}-\lambda_\varepsilon\delta_\varepsilon\right)^2\eta\zeta(u_\varepsilon)\varepsilon|\nabla u_\varepsilon|^2      \nonumber\\
&\leq&\frac{\varepsilon^2}{\kappa}
\int_{\mathcal{B}_2}\Delta\left[\left(x_{n+1}-\lambda_\varepsilon\delta_\varepsilon\right)^2\eta\right]
\zeta(u_\varepsilon)\varepsilon|\nabla u_\varepsilon|^2   \nonumber\\
&+&\frac{\varepsilon^2}{\kappa}\int_{\mathcal{B}_2}4\left(x_{n+1}-\lambda_\varepsilon\delta_\varepsilon\right)\frac{\partial u_\varepsilon}{\partial x_{n+1}}\eta\zeta^\prime(u_\varepsilon)\varepsilon|\nabla u_\varepsilon|^2 \\
&+&\frac{\varepsilon^2}{\kappa}\int_{\mathcal{B}_2}2\left(x_{n+1}-\lambda_\varepsilon\delta_\varepsilon\right)^2\left(\nabla\eta\cdot\nabla u_\varepsilon\right)\zeta^\prime(u_\varepsilon)\varepsilon|\nabla u_\varepsilon|^2   \nonumber\\
&+&\frac{\varepsilon^2}{\kappa}
\int_{\mathcal{B}_2}\left(\zeta^{\prime\prime}(u_\varepsilon)|\nabla u_\varepsilon|^2+\zeta^\prime(u_\varepsilon)\Delta u_\varepsilon\right)
\left(x_{n+1}-\lambda_\varepsilon\delta_\varepsilon\right)^2\eta\varepsilon|\nabla u_\varepsilon|^2. \nonumber
\end{eqnarray}
In the right hand side, the first term is bounded by
\begin{equation}\label{6.11}
\frac{\varepsilon^2}{\kappa}
\int_{\mathcal{B}_2}\Delta\left[\left(x_{n+1}-\lambda_\varepsilon\delta_\varepsilon\right)^2\eta\right]
\zeta(u_\varepsilon)\varepsilon|\nabla u_\varepsilon|^2\leq C\varepsilon^2,
\end{equation}
because both
$\Delta\left[\left(x_{n+1}-\lambda_\varepsilon\delta_\varepsilon\right)^2\eta\right]$ and $\zeta(u_\varepsilon)$
are bounded by a universal constant.

Note that the supports of $\zeta^\prime(u_\varepsilon)$ and $\zeta^{\prime\prime}(u_\varepsilon)$ belong to $\{|u_\varepsilon|<1-b_2\}$. By the
Cauchy inequality, the second term is bounded by
\begin{eqnarray}\label{6.8}
&&\frac{\varepsilon^2}{\kappa}\int_{\mathcal{B}_2}4\left(x_{n+1}-\lambda_\varepsilon\delta_\varepsilon\right)\frac{\partial u_\varepsilon}{\partial x_{n+1}}\eta\zeta^\prime(u_\varepsilon)\varepsilon|\nabla u_\varepsilon|^2\\
&\leq&C\varepsilon\left[\int_{\{|u_\varepsilon|<1-b_2\}\cap\mathcal{B}_2}
\left(x_{n+1}-\lambda_\varepsilon\delta_\varepsilon\right)^2\eta^2\varepsilon|\nabla u_\varepsilon|^2\right]^{\frac{1}{2}}
\left[\int_{\{|u_\varepsilon|<1-b_2\}\cap\mathcal{B}_2}
\left(\varepsilon\frac{\partial u_\varepsilon}{\partial x_{n+1}}\right)^2\varepsilon|\nabla u_\varepsilon|^2\right]^{\frac{1}{2}} \nonumber\\
&\leq&C\varepsilon\delta_\varepsilon.\nonumber
\end{eqnarray}
Here we have used Proposition \ref{prop exponential decay},
\eqref{L2 bound 1} and the fact that $\varepsilon|\frac{\partial
u_\varepsilon}{\partial x_{n+1}}|\leq C$.

Similarly, the third term can be controlled as
\begin{eqnarray}\label{6.9}
&&\frac{\varepsilon^2}{\kappa}\int_{\mathcal{B}_2}\left(x_{n+1}-\lambda_\varepsilon\delta_\varepsilon\right)^2\left(\nabla\eta\cdot\nabla u_\varepsilon\right)\zeta^\prime(u_\varepsilon)\varepsilon|\nabla u_\varepsilon|^2  \nonumber\\
&\leq&C\left(\sup|\nabla\eta|\right)\left(\sup\varepsilon|\nabla
u_\varepsilon|\right)\varepsilon\int_{\{|u_\varepsilon|<1-b_2\}\cap(B_1\times(-4/3,4/3))}
\left(x_{n+1}-\lambda_\varepsilon\delta_\varepsilon\right)^2\varepsilon|\nabla u_\varepsilon|^2 \\
&\leq&C\varepsilon\delta_\varepsilon^2.\nonumber
\end{eqnarray}

Finally, in the last term, by employing \eqref{gradient bound}, we
obtain
\begin{eqnarray}\label{6.10}
&&\frac{\varepsilon^2}{\kappa}
\int_{\mathcal{B}_2}\left(\zeta^{\prime\prime}(u_\varepsilon)|\nabla u_\varepsilon|^2+\zeta^\prime(u_\varepsilon)\Delta u_\varepsilon\right)
\left(x_{n+1}-\lambda_\varepsilon\delta_\varepsilon\right)^2\eta\varepsilon|\nabla u_\varepsilon|^2  \nonumber\\
&\leq&C\int_{\{|u_\varepsilon|<1-b_2\}\cap\mathcal{B}_2}
\left(x_{n+1}-\lambda_\varepsilon\delta_\varepsilon\right)^2\eta\varepsilon|\nabla u_\varepsilon|^2 \\
&\leq&C\delta_\varepsilon^2.\nonumber
\end{eqnarray}

Substituting \eqref{6.11}-\eqref{6.10} into \eqref{6.7}, and noting the fact that $\delta_\varepsilon\gg\varepsilon$, we obtain
\eqref{L2 bound 2}.

Combining \eqref{L2 bound 1} and \eqref{L2 bound 2} we finish the proof.
\end{proof}

Once we have this bound, we can further sharpen several estimates in
the above proof to show that
\begin{coro}\label{coro 6.4}
For any $\sigma>0$, there exists a constant $b>0$ such that
\[\int_{\{|u_\varepsilon|>1-b\}\cap\mathcal{C}_{3/4}}\left(x_{n+1}-\lambda_\varepsilon\delta_\varepsilon\right)^2\varepsilon|\nabla u_\varepsilon|^2
\leq \sigma\delta_\varepsilon^2+C\varepsilon^2.\]
\end{coro}
\begin{proof}
The starting point is the estimate \eqref{6.7}, where $\xi$ is now
assumed to have its support in $(-1+b,1-b)$, and satisfies
$\xi\equiv 1$ in $(-1+2b,1-2b)$,
$|\zeta^\prime|^2+|\zeta^{\prime\prime}|\leq 64b^{-2}$. The constant
$b$ will be determined later.

 We only
need to give a better control in \eqref{6.10}. Estimates in
\eqref{6.11}, \eqref{6.8} and \eqref{6.9} will be kept. Using the
Cauchy inequality, they can be bounded by
$\sigma\delta_\varepsilon^2+C\varepsilon^2$.

Replace \eqref{6.10} by
\begin{eqnarray}\label{6.10.1}
&&\frac{\varepsilon^2}{\kappa}
\int_{\mathcal{B}_2}\left(\zeta^{\prime\prime}(u_\varepsilon)|\nabla
u_\varepsilon|^2+\zeta^\prime(u_\varepsilon)\Delta
u_\varepsilon\right)
\left(x_{n+1}-\lambda_\varepsilon\delta_\varepsilon\right)^2\eta\varepsilon|\nabla u_\varepsilon|^2  \nonumber\\
&\leq&C\int_{\{1-2b<|u_\varepsilon|<1-b\}\cap\mathcal{B}_2}
\left(x_{n+1}-\lambda_\varepsilon\delta_\varepsilon\right)^2\eta\varepsilon|\nabla
u_\varepsilon|^2.
\end{eqnarray}
Here the constant $C$ is independent of $b$. This is because,
instead of using the bound \eqref{gradient bound} as in the proof of
the previous lemma, we can use
\[\varepsilon^2|\nabla u_\varepsilon|^2\leq 2W(u_\varepsilon),\quad
\varepsilon^2|\Delta u_\varepsilon| \leq
|W^\prime(u_\varepsilon)|,\] which follow from the Modica inequality
and the equation \eqref{equation}.
Thus
\[\varepsilon^2\left(\zeta^{\prime\prime}(u_\varepsilon)|\nabla
u_\varepsilon|^2+\zeta^\prime(u_\varepsilon)\Delta
u_\varepsilon\right)\] is bounded independent of $b\in(0,1)$.

In view of this, to complete the proof we only need to prove that,
for any $\sigma$, there exists a constant $b\in(0,1)$ such that,
\begin{equation}\label{6.10.1}
\int_{\{1-2b<|u_\varepsilon|<1-b\}\cap\mathcal{B}_2}
\left(x_{n+1}-\lambda_\varepsilon\delta_\varepsilon\right)^2\varepsilon|\nabla
u_\varepsilon|^2\leq
\sigma\int_{\{|u_\varepsilon|<1-b\}\cap\mathcal{B}_2}
\left(x_{n+1}-\lambda_\varepsilon\delta_\varepsilon\right)^2\varepsilon|\nabla
u_\varepsilon|^2.
\end{equation}

To this aim, first note that, because (see Proposition \ref{prop
exponential decay})
\[\lim_{\varepsilon\to0}\sup_{\{|u_\varepsilon|<1-b\}}\left(x_{n+1}-\lambda_\varepsilon\delta_\varepsilon\right)^2=0,\]
\eqref{6.6} can be improved to
\begin{equation}\label{6.10.2}
\int_{B_\varepsilon}\left(x_{n+1}-\lambda_\varepsilon\delta_\varepsilon\right)^2\varepsilon|\nabla
u_\varepsilon|^2 \leq \frac{\sigma}{2}\delta_\varepsilon^2,\quad
\forall \ \varepsilon \mbox{ small}.
\end{equation}

Next, by \eqref{4.2}, \eqref{6.5} can be rewritten as
\begin{eqnarray}\label{6.5.1}
&&\int_{A_\varepsilon}\left(x_{n+1}-\lambda_\varepsilon\delta_\varepsilon\right)^2\varepsilon|\nabla
u_\varepsilon|^2\nonumber\\
&=&\int_{-1+b}^{1-b}\int_{D_\varepsilon}\left(h_\varepsilon^t-\lambda_\varepsilon\delta_\varepsilon\right)^2
\left(1+|\nabla h_\varepsilon^t|^2\right)\varepsilon\frac{\partial u_\varepsilon}{\partial x_{n+1}}dxdt   \\
&=&\left[\int_{D_\varepsilon}\left(h_\varepsilon^{t_\varepsilon}
-\lambda_\varepsilon\delta_\varepsilon\right)^2\right]\left[\int_{-1+b}^{1-b}\int_{D_\varepsilon}
\left(1+|\nabla h_\varepsilon^t|^2\right)\varepsilon\frac{\partial
u_\varepsilon}{\partial x_{n+1}}dxdt\right]+O(\varepsilon^2).
\nonumber\\
&\geq&
c\left[\int_{D_\varepsilon}\left(h_\varepsilon^{t_\varepsilon}
-\lambda_\varepsilon\delta_\varepsilon\right)^2\right]+O(\varepsilon^2).\nonumber
\end{eqnarray}
In the last step we have used the fact that
$\varepsilon\frac{\partial u_\varepsilon}{\partial x_{n+1}}\geq c$
in $A_\varepsilon\cap\{|u_\varepsilon|<1/2\}$.

Now consider the integral on $(1-2b,1-b)$. By noting that
$\varepsilon\frac{\partial u_\varepsilon}{\partial x_{n+1}}$ is
small in $\{|u_\varepsilon|>1-2b\}$ (using the Modica inequality
\eqref{Modica inequality}), we obtain
\begin{eqnarray}\label{6.5.2}
&&\int_{A_\varepsilon\cap\{1-2b<|u_\varepsilon|<1-b\}}\left(x_{n+1}-\lambda_\varepsilon\delta_\varepsilon\right)^2\varepsilon|\nabla
u_\varepsilon|^2\nonumber\\
&=&\left[\int_{D_\varepsilon}\left(h_\varepsilon^{t_\varepsilon}
-\lambda_\varepsilon\delta_\varepsilon\right)^2\right]\left[\int_{1-2b<|t|<1-b}\int_{D_\varepsilon}
\left(1+|\nabla h_\varepsilon^t|^2\right)\varepsilon\frac{\partial
u_\varepsilon}{\partial x_{n+1}}dxdt\right]+O(\varepsilon^2)\\
&=&o_b(1)\left[\int_{D_\varepsilon}\left(h_\varepsilon^{t_\varepsilon}
-\lambda_\varepsilon\delta_\varepsilon\right)^2\right]+O(\varepsilon^2).\nonumber
\end{eqnarray}

Combining \eqref{6.5.1} and \eqref{6.5.2} we get
\begin{equation*}
\int_{A_\varepsilon\cap\{1-2b<|u_\varepsilon|<1-b\}}\left(x_{n+1}-\lambda_\varepsilon\delta_\varepsilon\right)^2\varepsilon|\nabla
u_\varepsilon|^2
=o_b(1)\int_{A_\varepsilon}\left(x_{n+1}-\lambda_\varepsilon\delta_\varepsilon\right)^2\varepsilon|\nabla
u_\varepsilon|^2+O(\varepsilon^2).
\end{equation*}
With \eqref{6.10.2} this implies \eqref{6.10.1}, if we have chosen
$b$ small enough. This completes the proof.
\end{proof}

Finally we give a uniform estimate on $\frac{\partial
h_\varepsilon^t}{\partial t}$ in a good set.
\begin{lem}\label{lem 4.2}
There exists a set $E_\varepsilon\subset D_\varepsilon$ with
$\mathcal{H}^n(D_\varepsilon\setminus E_\varepsilon)\leq
C\delta_\varepsilon$ so that the following holds. For any
$t\in(-1+b,1-b)$ and $X_\varepsilon\in\Pi^{-1}(E_\varepsilon)\cap
A_\varepsilon$ with $u_\varepsilon(X_\varepsilon)=t$,
\[\varepsilon\left[\frac{\partial h_\varepsilon^t}{\partial t}(X_\varepsilon)\right]^{-1}
=g^\prime(g^{-1}(t))+o_\varepsilon(1).\] Here $o_\varepsilon(1)$
means a quantity converging to $0$ as $\varepsilon\to0$, independent
of $X_\varepsilon$ and $t$.
\end{lem}
\begin{proof}
Let $E_\varepsilon=D_\varepsilon\cap\{Mf_\varepsilon<\delta_\varepsilon\}$. By \eqref{weak L1 bound},
\[\mathcal{H}^n(D_\varepsilon\setminus E_\varepsilon)\leq\mathcal{H}^n(B_1\setminus \{Mf_\varepsilon\geq\delta_\varepsilon\})\leq C\delta_\varepsilon\to0.\]

For any $X_\varepsilon\in\Pi^{-1}(E_\varepsilon)\cap A_\varepsilon$, consider
\[v_\varepsilon(X):=u_\varepsilon(X_\varepsilon+\varepsilon X), \ \ \ \mbox{for}\ X\in \mathcal{B}_{\varepsilon^{-1}/2}.\]
$v_\varepsilon$ is a solution of \eqref{equation 0}.  By definition,
$v_\varepsilon(0)=u_\varepsilon(X_\varepsilon)=t\in(-1+b,1-b)$
because $X_\varepsilon\in A_\varepsilon$. As usual assume
$v_\varepsilon$ converges to a function $v_\infty$ in
$C^2_{loc}(\mathbb{R}^{n+1})$, which is also a solution of
\eqref{equation 0} on $\mathbb{R}^{n+1}$.

By the definition of Hardy-Littlewood maximal function and our choice of $E_\varepsilon$,
\[\sup_{0<r<\varepsilon^{-1}/2}r^{-n}\int_{\mathcal{B}_r}\sum_{i=1}^n\left(\frac{\partial v_\varepsilon}{\partial x_i}\right)^2\leq\delta_\varepsilon\to0.\]
After passing to the limit, we see $v_\infty$ depends only on $x_{n+1}$. Then by \eqref{energy bound on ball}, we have the energy bound
\[\int_{\mathcal{B}_r}\frac{1}{2}|\nabla v_\infty|^2+W(v_\infty)\leq 8^n\sigma_0\omega_n r^n,\ \ \ \forall \ r>0.\]
From this we deduce that $v_\infty(X)\equiv g(x_{n+1}+g^{-1}(t))$
(see again the proof of Lemma \ref{lem B1}).

By definition and the $C^1_{loc}$ convergence of $v_\varepsilon$,
\[\varepsilon\frac{\partial u_\varepsilon}{\partial x_{n+1}}(X_\varepsilon)=\frac{\partial v_\varepsilon}{\partial x_{n+1}}(0)\to\frac{\partial v_\infty}{\partial x_{n+1}}(0)=g^\prime(g^{-1}(t)).\]
The claim then follows from \eqref{x,t coordinates}.
\end{proof}

\section{The blow up limit}
\numberwithin{equation}{section}
 \setcounter{equation}{0}

This section is devoted to prove
\begin{prop}\label{prop harmonic limit}
$\bar{h}$ is harmonic in $B_1$.
\end{prop}

Fix a $\psi\in C_0^\infty((-1,1))$, such that $0\leq \psi\leq 1$,
$\psi\equiv1$ in $(-1/2,1/2)$ and $|\psi^\prime|\leq 4$. For any
$\varphi\in C_0^\infty(B_1)$, let
$X(x,x_{n+1})=\varphi(x)\psi(x_{n+1})e_{n+1}$, which is a smooth
vector field with compact support in $\mathcal{C}_1$.

To prove Proposition \ref{prop harmonic limit}, we substitute this
vector field into the stationary condition \eqref{stationary
condition}. Roughly speaking, if we view the level set of
$u_\varepsilon$ as the graph of a function $h$, because $h$ almost
satisfies an elliptic equation, this procedure amounts to
multiplying the equation of $h$ by a $C_0^\infty$ function and then
integrating by parts, which of course is a standard method in the
elliptic equation theory.

Note that
\[DX(x,x_{n+1})=\psi(x_{n+1})\nabla\varphi(x)\otimes e_{n+1}+\varphi(x)\psi^\prime(x_{n+1}) e_{n+1}\otimes e_{n+1},\]
\[\mbox{div}X(x,x_{n+1})=\varphi(x)\psi^\prime(x_{n+1}).\]
Since $\mbox{div}X$ vanishes in $B_1\times(-1/2,1/2)$, by Proposition \ref{prop exponential decay},
\begin{equation}\label{5.1}
\int_{\mathcal{C}_1}\left[\frac{\varepsilon}{2}|\nabla u_\varepsilon|^2
+\frac{1}{\varepsilon}W(u_\varepsilon)\right]\mbox{div}X=O(e^{-\frac{c}{\varepsilon}}).
\end{equation}
Similarly,
\begin{equation}\label{5.2}
\int_{\mathcal{C}_1}\varphi(x)\psi^\prime(x_{n+1})\varepsilon\left(\frac{\partial
u_\varepsilon}{\partial
x_{n+1}}\right)^2=O(e^{-\frac{c}{\varepsilon}}).
\end{equation}
Thus from the stationary condition \eqref{stationary condition} we
deduce that
\begin{equation}\label{5.02}
\int_{\mathcal{C}_1}\varepsilon\psi\left(\sum_{i=1}^n\frac{\partial
u_\varepsilon}{\partial x_i}\frac{\partial \varphi}{\partial
x_i}\right)\frac{\partial u_\varepsilon}{\partial
x_{n+1}}=O(e^{-\frac{c}{\varepsilon}})=o(\delta_\varepsilon),
\end{equation}
where in the last equality we have used the assumption \eqref{absurd
assumption 3}.

First note that
\begin{eqnarray}\label{5.3}
&&\int_{\mathcal{C}_1\cap\{|u_\varepsilon|\geq 1-b\}}\varepsilon\psi\left(\sum_{i=1}^n\frac{\partial u_\varepsilon}{\partial x_i}\frac{\partial \varphi}{\partial x_i}\right)\frac{\partial u_\varepsilon}{\partial x_{n+1}}dxdx_{n+1}\nonumber\\
&\leq&C\left(\sup_{B_1}|\nabla\varphi|\right)\left[\int_{\mathcal{C}_1\cap\{|u_\varepsilon|\geq 1-b\}}\varepsilon\left(\sum_{i=1}^n\frac{\partial u_\varepsilon}{\partial x_i}\right)^2\right]^{1/2}\left[\int_{\mathcal{C}_1\cap\{|u_\varepsilon|\geq 1-b\}}\varepsilon\left(\frac{\partial u_\varepsilon}{\partial x_{n+1}}\right)^2\right]^{1/2}\\
&\leq&C(\varphi)o_b(1)\delta_\varepsilon,\nonumber
\end{eqnarray}
where $o_b(1)$ converges to $0$ as $b\to0$ (by Lemma \ref{lem B3}).

Next in $B_\varepsilon$,
\begin{eqnarray}\label{5.4}
&&\int_{B_\varepsilon}\varepsilon\psi\left(\sum_{i=1}^n\frac{\partial u_\varepsilon}{\partial x_i}\frac{\partial \varphi}{\partial x_i}\right)\frac{\partial u_\varepsilon}{\partial x_{n+1}}dxdx_{n+1}\nonumber\\
&\leq&C\left(\sup_{B_1}|\nabla\varphi|\right)\left[\int_{B_\varepsilon}\varepsilon\left(\sum_{i=1}^n\frac{\partial u_\varepsilon}{\partial x_i}\right)^2\right]^{1/2}\left[\int_{B_\varepsilon}\varepsilon\left(\frac{\partial u_\varepsilon}{\partial x_{n+1}}\right)^2\right]^{1/2}\\
&\leq&C(\varphi)\delta_\varepsilon\mu_\varepsilon(B_\varepsilon)^{1/2}\nonumber\\
&\leq&C(\varphi)\delta_\varepsilon^2 \quad\mbox{(by Lemma \ref{lem
4.3})}.\nonumber
\end{eqnarray}
Substituting \eqref{5.3} and \eqref{5.4} into \eqref{5.02} and
noting that $\psi\equiv 1$ on $A_\varepsilon$ (recall that
$A_\varepsilon\subset B_1\times\{|x_{n+1}|<1/2\}$), we see
\begin{equation}\label{5.5}
\int_{A_\varepsilon}\varepsilon\left(\sum_{i=1}^n\frac{\partial
u_\varepsilon}{\partial x_i}\frac{\partial \varphi}{\partial
x_i}\right)\frac{\partial u_\varepsilon}{\partial
x_{n+1}}dxdx_{n+1}=o(\delta_\varepsilon)+o_b(1)\delta_\varepsilon.
\end{equation}
By using the transformation $(x,x_{n+1})=(x,h_\varepsilon^t(x))$ and
\eqref{x,t coordinates}, this integral can be transformed into
\begin{equation}\label{5.6}
\int_{-1+b}^{1-b}\int_{D_\varepsilon}\varepsilon\left(\frac{\partial
h^t_\varepsilon}{\partial t}\right)^{-1}\sum_{i=1}^n\frac{\partial
h^t_\varepsilon}{\partial x_i}\frac{\partial \varphi}{\partial
x_i}dxdt=o(\delta_\varepsilon)+o_b(1)\delta_\varepsilon.
\end{equation}
We need to further divide $D_\varepsilon$ into two parts, using the
set $E_\varepsilon$ introduced in Lemma \ref{lem 4.2}.

In the first part $D_\varepsilon\setminus E_\varepsilon$, by
\eqref{gradient bound}, \eqref{x,t coordinates}, Lemma \ref{lem H1
bound} and Lemma \ref{lem 4.2},
\begin{eqnarray*}
&&\int_{-1+b}^{1-b}\int_{D_\varepsilon\setminus E_\varepsilon}\varepsilon\left(\frac{\partial h^t_\varepsilon}{\partial x_{n+1}}\right)^{-1}\sum_{i=1}^n\frac{\partial h^t_\varepsilon}{\partial x_i}\frac{\partial \varphi}{\partial x_i}dxdt\\
&\leq&\left(\sup_{A_\varepsilon}\Big|\varepsilon\left(\frac{\partial h^t_\varepsilon}{\partial t}\right)^{-1}\Big|\right)\left(\sup_{B_1}|\nabla\varphi|\right)\left[\int_{-1+b}^{1-b}\int_{D_\varepsilon\setminus E_\varepsilon}\sum_{i=1}^n\left(\frac{\partial h^t_\varepsilon}{\partial x_i}\right)^2dxdt\right]^{1/2}\mathcal{H}^n(D_\varepsilon\setminus E_\varepsilon)^{1/2}\\
&\leq&C(\varphi)\delta_\varepsilon^{\frac{3}{2}}=o(\delta_\varepsilon).
\end{eqnarray*}

In $E_\varepsilon$,
\begin{eqnarray*}
\int_{-1+b}^{1-b}\int_{E_\varepsilon}\varepsilon\left(\frac{\partial
h^t_\varepsilon}{\partial t}\right)^{-1}\sum_{i=1}^n\frac{\partial
h^t_\varepsilon}{\partial x_i}\frac{\partial \varphi}{\partial
x_i}dxdt
=\int_{-1+b}^{1-b}\int_{E_\varepsilon}g^{\prime}(g^{-1}(t))\sum_{i=1}^n\frac{\partial
h^t_\varepsilon}{\partial x_i}\frac{\partial \varphi}{\partial
x_i}dxdt+o(\delta_\varepsilon),
\end{eqnarray*}
where in the last equality we have used
\[\sup_{E_\varepsilon}\Big|\varepsilon\left(\frac{\partial h^t_\varepsilon}{\partial t}\right)^{-1}-g^\prime(g^{-1}(t))\Big|=o_\varepsilon(1)\to 0,\]
and the bound
\begin{eqnarray*}
\int_{-1+b}^{1-b}\int_{E_\varepsilon}\sum_{i=1}^n\frac{\partial h^t_\varepsilon}{\partial x_i}\frac{\partial \varphi}{\partial x_i}dxdt
&\leq&\left[\int_{-1+b}^{1-b}\int_{B_1}|\nabla h^t_\varepsilon|^2dxdt\right]^{1/2}\left[\int_{-1+b}^{1-b}\int_{B_1}|\nabla \varphi|^2dxdt\right]^{1/2}\\
&\leq&C(\varphi)\delta_\varepsilon.\quad \mbox{(by Lemma \ref{lem H1
bound})}
\end{eqnarray*}

By Cauchy inequality we also have
\begin{eqnarray*}
&&\int_{-1+b}^{1-b}\int_{B_1\setminus
E_\varepsilon}g^\prime(g^{-1}(t))\sum_{i=1}^n\frac{\partial
h^t_\varepsilon}{\partial x_i}\frac{\partial \varphi}{\partial
x_i}dxdt\\
&\leq& C\left(\sup_{B_1}|\nabla\varphi|\right)\left[\int_{-1+b}^{1-b}\int_{B_1}|\nabla h^t_\varepsilon|^2dxdt\right]^{1/2}\mathcal{H}^n(B_1\setminus E_\varepsilon)^{1/2}\\
&\leq&C(\varphi)\delta_\varepsilon^{3/2}=o(\delta_\varepsilon),
\end{eqnarray*}
where we have used Lemma \ref{lem H1 bound}, Lemma \ref{lem 3.2} and
Lemma \ref{lem 4.2}.

Putting these three integrals together, by \eqref{5.6} we get
\[\int_{-1+b}^{1-b}\int_{B_1}g^\prime(g^{-1}(t))\sum_{i=1}^n\frac{\partial h^t_\varepsilon}{\partial x_i}\frac{\partial \varphi}{\partial x_i}dxdt=o_b(1)\delta_\varepsilon+o(\delta_\varepsilon).\]
By the weak convergence of  $\delta_\varepsilon^{-1}\nabla
h^t_\varepsilon$ to $\nabla\bar{h}$ in $L^2(B_1\times (-1+b,1-b))$,
we can let $\varepsilon\to 0$ to obtain
\begin{equation}\label{5.7}
\left[\int_{-1+b}^{1-b}g^\prime(g^{-1}(t))dt\right]\left[\int_{B_1}\sum_{i=1}^n\frac{\partial
\bar{h}}{\partial x_i}\frac{\partial \varphi}{\partial
x_i}dx\right]=o_b(1).
\end{equation}
For $b\in(0,1/2)$,
\[\int_{-1+b}^{1-b}g^\prime(g^{-1}(t))dt=\int_{g^{-1}(-1+b)}^{g^{-1}(1-b)}g^\prime(s)^2ds\geq c\sigma_0.\]
At the first step, we can choose a smaller $\tilde{b}$ and get
another family $\tilde{h}^t_\varepsilon$ for
$t\in(-1+\tilde{b},1-\tilde{b})$. Assume its limit is $\tilde{h}$.
By Remark \ref{rmk 2}, $\tilde{h}^t_\varepsilon=h^t_\varepsilon$ for
$t\in(-1+b,1-b)$. Then by \eqref{4.2}, $\tilde{h}=\bar{h}$. In other
words, the limit $\bar{h}$ does not depend on $b$.

After taking $b\to0$ in \eqref{5.7}, we get
\[\int_{B_1}\sum_{i=1}^n\frac{\partial \bar{h}}{\partial x_i}\frac{\partial \varphi}{\partial x_i}dx=0.\]
Since $\varphi\in C_0^\infty(B_1)$ can be arbitrary and $\bar{h}\in
H^1(B_1)$, standard harmonic function theory implies that $\bar{h}$
is harmonic in $B_1$ and we finish the proof of Proposition
\ref{prop harmonic limit}.

\section{Proof of the tilt-excess decay}
\numberwithin{equation}{section}
 \setcounter{equation}{0}

Recall that $\bar{h}$ is a harmonic function satisfying (see \eqref{H1 bound on the blow up limit})
\[\int_{B_1}|\nabla\bar{h}|^2+\bar{h}^2\leq C.\]
By standard interior gradient estimates for harmonic functions we
get
\begin{equation}\label{8.01}
|\nabla\bar{h}(0)|\leq C,\ \ \ \ \sup_{B_r}|\nabla\bar{h}-\nabla\bar{h}(0)|\leq Cr, \ \ \ \forall r\in(0,1/2).
\end{equation}
Thus
\begin{equation}\label{8.02}
\int_{B_r}|\nabla\bar{h}-\nabla\bar{h}(0)|^2\leq Cr^{n+2}, \ \ \ \forall r\in(0,1/2).
\end{equation}

In this section we complete the proof of Theorem \ref{thm tilt
excess decay}. We first consider the special case when
$\nabla\bar{h}(0)=0$, and then reduce the general case to this one.

\subsection{The case $\nabla\bar{h}(0)=0$}

Take a $\psi\in C_0^\infty((-1,1))$ satisfying $0\leq \psi\leq 1$,
$\psi\equiv1$ in $(-1/2,1/2)$, $|\psi^\prime|\leq 3$. For any
$r\in(0,1/4)$, choose an $\phi\in C_0^\infty(B_{2r})$ such that $0\leq\phi\leq 1$,
$\phi\equiv1$ in $B_r$. In the stationary condition
\eqref{stationary condition}, take the vector field
\[Y=\phi(x)^2\psi(x_{n+1})^2\left(x_{n+1}-\lambda_\varepsilon\delta_\varepsilon\right)
e_{n+1},\]
 where $\lambda_\varepsilon$ is the constant appearing in \eqref{defintion blow up sequence}.

As in Section 7, by viewing the level set of $u_\varepsilon$ as the
graph of a function $h$, because $h$ almost satisfies an elliptic
equation, taking such a vector field as a test function corresponds
to the procedure of multiplying the equation of $h$ by $h\phi^2$ and
then integrating by parts, which is again a standard method in the
elliptic equation theory. (It is used to derive the Caccioppoli
inequality.)

By this choice of $Y$ we get
\begin{eqnarray}\label{stationary with normal variation 2}
0=\int_{\mathcal{C}_1}&&\left[\frac{\varepsilon}{2}|\nabla u_\varepsilon|^2+\frac{1}{\varepsilon}W(u_\varepsilon)\right]
\left[\phi^2\psi^2+2\phi^2\psi\psi^\prime\left(x_{n+1}-\lambda_\varepsilon\delta_\varepsilon\right)\right]                \nonumber\\
&&-\phi^2\psi^2\nu_{\varepsilon,n+1}^2\varepsilon|\nabla u_\varepsilon|^2-\left(x_{n+1}-\lambda_\varepsilon\delta_\varepsilon\right)\sum_{i=1}^{n}2\phi\psi^2\frac{\partial\phi}{\partial x_i}\nu_{\varepsilon,i}\nu_{\varepsilon,n+1}\varepsilon|\nabla u_\varepsilon|^2       \\
&&-\left(x_{n+1}-\lambda_\varepsilon\delta_\varepsilon\right)2\phi^2\psi\psi^\prime\nu_{\varepsilon,n+1}^2\varepsilon|\nabla u_\varepsilon|^2.        \nonumber
\end{eqnarray}

As in the proof of Caccioppoli inequality \eqref{Caccioppoli
inequality}, those terms containing $\psi^\prime$ are bounded by
$O(e^{-\frac{1}{C\varepsilon}})$.
By the Modica inequality \eqref{Modica inequality}, \eqref{stationary with normal variation 2} can be transformed to
\begin{eqnarray*}
&&\int_{\mathcal{C}_1}\phi^2\psi^2\left[1-\left(\nu_\varepsilon\cdot e_{n+1}\right)^2\right]\varepsilon|\nabla u_\varepsilon|^2\\
&\leq&\int_{\mathcal{C}_1}2\phi\psi^2\left(x_{n+1}-\lambda_\varepsilon\delta_\varepsilon\right)
\sum_{i=1}^{n}\frac{\partial\phi}{\partial
x_i}\nu_{\varepsilon,i}\nu_{\varepsilon,n+1}\varepsilon|\nabla
u_\varepsilon|^2+O(e^{-\frac{1}{C\varepsilon}}).
\end{eqnarray*}

Since $1-\psi^2\equiv 0$ in $\{|x_{n+1}|\leq 1/2\}$, as before we have
\[\int_{\mathcal{C}_1}\phi^2\left(1-\psi^2\right)\left[1-\left(\nu_\varepsilon\cdot e_{n+1}\right)^2\right]\varepsilon|\nabla u_\varepsilon|^2=O(e^{-\frac{1}{C\varepsilon}}).\]
Thus we obtain
\begin{eqnarray}\label{7.2}
&&\int_{\mathcal{C}_1}\phi^2\left[1-\left(\nu_\varepsilon\cdot
e_{n+1}\right)^2\right]\varepsilon|\nabla u_\varepsilon|^2\\
\nonumber
&\leq&\int_{\mathcal{C}_1}2\phi\psi^2\left(x_{n+1}-\lambda_\varepsilon\delta_\varepsilon\right)
\sum_{i=1}^{n}\frac{\partial\phi}{\partial
x_i}\nu_{\varepsilon,i}\nu_{\varepsilon,n+1}\varepsilon|\nabla
u_\varepsilon|^2+O(e^{-\frac{1}{C\varepsilon}}).
\end{eqnarray}

Now we consider the convergence of the integral in the right hand side of \eqref{7.2}.

\begin{lem}\label{lem 7.1}
We have
\begin{eqnarray*}
&&\lim_{\varepsilon\to0}\delta_\varepsilon^{-2}\int_{\mathcal{C}_1}2\phi\psi^2\left(x_{n+1}-\lambda_\varepsilon\delta_\varepsilon\right)
\sum_{i=1}^{n}\frac{\partial\phi}{\partial x_i}\nu_{\varepsilon,i}\nu_{\varepsilon,n+1}\varepsilon|\nabla u_\varepsilon|^2\\
&=&\left[\int_{-1}^1g^\prime(g^{-1}(t))dt\right]\left[
\int_{B_1}\phi^2|\nabla\bar{h}(x)|^2dx\right].
\end{eqnarray*}
\end{lem}
\begin{proof}
In $\{|u_\varepsilon|\geq 1-b\}$,
\begin{eqnarray*}
&&\Big|\int_{\{|u_\varepsilon|\geq 1-b\}\cap\mathcal{C}_1}2\phi\psi^2\left(x_{n+1}-\lambda_\varepsilon\delta_\varepsilon\right)
\sum_{i=1}^{n}\frac{\partial\phi}{\partial x_i}\nu_{\varepsilon,i}\nu_{\varepsilon,n+1}\varepsilon|\nabla u_\varepsilon|^2\Big|\\
&\leq&C\left(\sup_{B_1}\big|\phi\psi^2\nabla\phi
\big|\right)\left[\int_{\{|u_\varepsilon|\geq
1-b\}}\sum_{i=1}^{n}\nu_{\varepsilon,i}^2\varepsilon|\nabla
u_\varepsilon|^2\right]^{\frac{1}{2}}
\left[\int_{\{|u_\varepsilon|\geq 1-b\}}\left(x_{n+1}-\lambda_\varepsilon\delta_\varepsilon\right)^2\varepsilon|\nabla u_\varepsilon|^2\right]^{\frac{1}{2}}\\
&=&o_b(1)\delta_\varepsilon^2. \quad \mbox{(by the definition of
$\delta_\varepsilon$ and Corollary \ref{coro 6.4})}
\end{eqnarray*}

In $B_\varepsilon$,
\begin{eqnarray*}
&&\Big|\int_{B_\varepsilon}2\phi\psi^2\left(x_{n+1}-\lambda_\varepsilon\delta_\varepsilon\right)
\sum_{i=1}^{n}\frac{\partial\phi}{\partial x_i}\nu_{\varepsilon,i}\nu_{\varepsilon,n+1}\varepsilon|\nabla u_\varepsilon|^2\Big|\\
&\leq&C\left(\sup_{\{|u_\varepsilon|\leq
1-b\}}\big|x_{n+1}-\lambda_\varepsilon\delta_\varepsilon\big|\right)\left(\sup_{B_1}\big|\phi\psi^2\nabla\phi
\big|\right)\left[\int_{B_\varepsilon}\sum_{i=1}^{n}\nu_{\varepsilon,i}^2\varepsilon|\nabla
u_\varepsilon|^2\right]^{\frac{1}{2}}
\left[\int_{B_\varepsilon}\varepsilon|\nabla u_\varepsilon|^2\right]^{\frac{1}{2}}\\
&=&o(\delta_\varepsilon^2),
\end{eqnarray*}
where we have used the definition of excess, Lemma \ref{lem 4.3} and
the fact that $\{|u_\varepsilon|\leq 1-b\}$ belongs to a small
neighborhood of $\{x_{n+1}=0\}$ (see Proposition \ref{prop
exponential decay}), which together with Lemma \ref{lem 5.2} implies
that
\begin{equation}\label{7.4}
\lim_{\varepsilon\to0}\sup_{\{|u_\varepsilon|\leq 1-b\}}\big|x_{n+1}-\lambda_\varepsilon\delta_\varepsilon\big|=0.
\end{equation}

Because $A_\varepsilon\subset\{|x_{n+1}|\leq 1/2\}$,
$\psi(x_{n+1})\equiv 1$ in $A_\varepsilon$. Hence we have, by using
the $(x,t)$ coordinates,
\begin{eqnarray}\label{8.03}
&&\int_{A_\varepsilon}2\phi\psi^2\left(x_{n+1}-\lambda_\varepsilon\delta_\varepsilon\right)
\sum_{i=1}^{n}\frac{\partial\phi}{\partial x_i}\nu_{\varepsilon,i}\nu_{\varepsilon,n+1}\varepsilon|\nabla u_\varepsilon|^2\\
&=&-\int_{-1+b}^{1-b}\int_{D_\varepsilon}2\phi\left(\nabla\phi\cdot\nabla
h^t_\varepsilon\right)
\left(h^t_\varepsilon-\lambda_\varepsilon\delta_\varepsilon\right)
 \varepsilon\left(\frac{\partial h^t_\varepsilon}{\partial t}\right)^{-1}dxdt.      \nonumber
\end{eqnarray}

In $A_\varepsilon$, by \eqref{x,t coordinates} and \eqref{gradient
bound},
\begin{equation}\label{8.2.1}
\varepsilon\left(\frac{\partial h^t_\varepsilon}{\partial
t}\right)^{-1} =\varepsilon\frac{\partial u_\varepsilon}{\partial
x_{n+1}}\leq C.
\end{equation}
 Let $E_\varepsilon$ be the set defined in Lemma \ref{lem
4.2}. By the Cauchy inequality, Lemma \ref{lem H1 bound},
\eqref{8.2.1}, \eqref{defintion blow up sequence}, \eqref{4.2} and
Sobolev inequality,
\begin{eqnarray*}
&&\int_{-1+b}^{1-b}\int_{D_\varepsilon\setminus
E_\varepsilon}2\phi\left(\nabla\phi\cdot\nabla
h^t_\varepsilon\right)
\left(h^t_\varepsilon-\lambda_\varepsilon\delta_\varepsilon\right)\varepsilon\left(\frac{\partial
h^t_\varepsilon}{\partial t}\right)^{-1}
 dxdt\\
 &\leq&C\left[\int_{-1+b}^{1-b}\int_{D_\varepsilon\setminus E_\varepsilon}\left(\nabla\phi\cdot\nabla h^t_\varepsilon\right)^2dxdt\right]^{\frac{1}{2}}
\left[\int_{-1+b}^{1-b}\int_{D_\varepsilon\setminus E_\varepsilon}
\left(h^t_\varepsilon-\lambda_\varepsilon\delta_\varepsilon\right)^2\phi^2
 dxdt\right]^{1/2}\\
&\leq& C\delta_\varepsilon\mathcal{H}^n(D_\varepsilon\setminus
E_\varepsilon)^\frac{p-1}{2p}
\left[\int_{-1+b}^{1-b}\left(\int_{B_1}
\left(h^t_\varepsilon-\lambda_\varepsilon\delta_\varepsilon\right)^{2p}\phi^{2p}
 dx\right)^{1/p}dt\right]^{1/2}\\
 &\leq& C\delta_\varepsilon\mathcal{H}^n(D_\varepsilon\setminus E_\varepsilon)^\frac{p-1}{2p}
\left[\int_{-1+b}^{1-b}\int_{B_1}
\big|\nabla(h^t_\varepsilon-\lambda_\varepsilon\delta_\varepsilon)\phi\big|^2
 dxdt+O(\varepsilon^2)\right]^{1/2}\\
 &\leq&C\mathcal{H}^n(D_\varepsilon\setminus E_\varepsilon)^\frac{p-1}{2p}\delta_\varepsilon^2=o(\delta_\varepsilon^2).
\end{eqnarray*}
In the above $p>1$ is a constant depending only on the dimension $n$.
This estimate gives
\begin{equation}\label{8.0.6}
\int_{-1+b}^{1-b}\int_{D_\varepsilon\setminus E_\varepsilon}2\phi\left(\nabla\phi\cdot\nabla h^t_\varepsilon\right)
\left(h^t_\varepsilon-\lambda_\varepsilon\delta_\varepsilon\right)
 \varepsilon\left(\frac{\partial h^t_\varepsilon}{\partial t}\right)^{-1}dxdt=o(\delta_\varepsilon^2).
 \end{equation}
Hence by \eqref{x,t coordinates},
\begin{eqnarray}\label{7.3}
&&\delta_\varepsilon^{-2}\int_{\mathcal{C}_1}2\phi\psi^2\left(x_{n+1}-\lambda_\varepsilon\delta_\varepsilon\right)
\sum_{i=1}^{n}\frac{\partial\phi}{\partial x_i}\nu_{\varepsilon,i}\nu_{\varepsilon,n+1}\varepsilon|\nabla u_\varepsilon|^2\\
&=&-\delta_\varepsilon^{-2}\int_{-1+b}^{1-b}\int_{E_\varepsilon}2\phi\left(\nabla\phi\cdot\nabla
h^t_\varepsilon\right)
\left(h^t_\varepsilon-\lambda_\varepsilon\delta_\varepsilon\right)
\varepsilon\left(\frac{\partial h^t_\varepsilon}{\partial
t}\right)^{-1}
 dxdt+o_b(1)+o_\varepsilon(1).     \nonumber
\end{eqnarray}

In $E_\varepsilon$, by Lemma \ref{lem H1 bound}, \eqref{defintion blow up sequence}-\eqref{4.2} and the Cauchy inequality, we have
\[\Big|\int_{-1+b}^{1-b}\int_{E_\varepsilon}2\phi\left(\nabla\phi\cdot\nabla h^t_\varepsilon\right)
\left(h^t_\varepsilon-\lambda_\varepsilon\delta_\varepsilon\right)
 dxdt\Big|\leq C\delta_\varepsilon^2.\]
Then by Lemma \ref{lem 4.2},
\begin{eqnarray*}
&&\int_{-1+b}^{1-b}\int_{E_\varepsilon}2\phi\left(\nabla\phi\cdot\nabla h^t_\varepsilon\right)
\left(h^t_\varepsilon-\lambda_\varepsilon\delta_\varepsilon\right)
\varepsilon\left(\frac{\partial h^t_\varepsilon}{\partial t}\right)^{-1}
 dxdt\\
 &=&\int_{-1+b}^{1-b}\int_{E_\varepsilon}2\phi\left(\nabla\phi\cdot\nabla h^t_\varepsilon\right)
\left(h^t_\varepsilon-\lambda_\varepsilon\delta_\varepsilon\right)
g^{\prime}(g^{-1}(t)) dxdt+o(\delta_\varepsilon^2).
\end{eqnarray*}

Finally, similar to \eqref{8.0.6}, we have
\begin{equation}\label{8.0.7}
\int_{-1+b}^{1-b}\int_{B_1\setminus
E_\varepsilon}2\phi\left(\nabla\phi\cdot\nabla
h_\varepsilon^t\right)
\left(h^t_\varepsilon-\lambda_\varepsilon\delta_\varepsilon\right)g^\prime(g^{-1}(t))dxdt
=o(\delta_\varepsilon^2).
\end{equation}

This, combined with Lemma \ref{lem 4.2}, implies that
\begin{eqnarray*}
&&\delta_\varepsilon^{-2}\int_{\mathcal{C}_1}2\phi\psi^2\left(x_{n+1}-\lambda_\varepsilon\delta_\varepsilon\right)\left(\nabla\phi\cdot\nu_\varepsilon
\right)\nu_{\varepsilon,n+1}\varepsilon|\nabla u_\varepsilon|^2\\
&=&-\delta_\varepsilon^{-2}\int_{-1+b}^{1-b}\int_{B_1}2\phi\left(\nabla\phi\cdot\nabla
h^t_\varepsilon\right)
\left(h^t_\varepsilon-\lambda_\varepsilon\delta_\varepsilon\right)g^\prime(g^{-1}(t))
 dxdt+o_b(1)+o_\varepsilon(1).
\end{eqnarray*}

By the Rellich compactness embedding theorem, Lemma \ref{lem H1
bound} and \eqref{defintion blow up sequence}-\eqref{4.2}, it can be
directly checked that
\begin{eqnarray*}
&&\lim_{\varepsilon\to0}
\delta_\varepsilon^{-2}\int_{-1+b}^{1-b}\int_{B_1}2\phi\left(\nabla\phi\cdot\nabla h^t_\varepsilon\right)
\left[h^t_\varepsilon-\lambda_\varepsilon\delta_\varepsilon\right]g^\prime(g^{-1}(t))
 dxdt\\
 &=&\left[\int_{-1+b}^{1-b}g^\prime(g^{-1}(t))dt\right]\left[\int_{B_1}2\phi\left(\nabla\phi\cdot\nabla \bar{h}\right)
\bar{h}dx\right].
\end{eqnarray*}
Since $\bar{h}$ is a harmonic function (see Proposition \ref{prop
harmonic limit}), an integration by parts gives
\[\int_{B_1}2\phi\left(\nabla\phi\cdot\nabla \bar{h}\right)
\bar{h}dx=-\int_{B_1}\phi^2|\nabla \bar{h}|^2
dx.\]

Now we have proved that
\begin{eqnarray*}
&&\lim_{\varepsilon\to0}\delta_\varepsilon^{-2}\int_{\mathcal{C}_1}2\phi\psi^2\left(x_{n+1}-\lambda_\varepsilon\delta_\varepsilon\right)
\sum_{i=1}^{n}\frac{\partial\phi}{\partial x_i}\nu_{\varepsilon,i}\nu_{\varepsilon,n+1}\varepsilon|\nabla u_\varepsilon|^2\\
&=&\left[\int_{-1+b}^{1-b}g^\prime(g^{-1}(t))dt\right]\left[\int_{B_1}\phi^2|\nabla
\bar{h}|^2 dx\right]+o_b(1).
\end{eqnarray*}
As in the proof of Proposition \ref{prop harmonic limit}, we can let $b\to0$ to finish the proof.
\end{proof}
Note that
\[\int_{-1}^1g^\prime(g^{-1}(t))dt=\int_{-\infty}^{+\infty}g^\prime(s)^2ds=\sigma_0.\]

By \eqref{8.02}, we can choose a $\theta\in(0,1/2)$ so that
\begin{equation}\label{8.3.1}
\theta^{-n}\int_{B_{2\theta}}|\nabla\bar{h}|^2\leq C\theta^2\leq
\frac{\theta}{4\max\{\sigma_0,1\}}.
\end{equation}
 Then by choosing $r=2\theta$ in the
definition of $\phi$, \eqref{7.2} and Lemma \ref{lem 7.1} give, for
all $\varepsilon$ small,
\[
\theta^{-n}\int_{\mathcal{C}_{\theta}}\left[1-\left(\nu_\varepsilon\cdot
e_{n+1}\right)^2\right]\varepsilon|\nabla u_\varepsilon|^2
\leq\frac{\theta}{3}\delta_\varepsilon^2,\] which contradicts the
initial assumption \eqref{absurd assumption 2}. This completes the
proof of Theorem \ref{thm tilt excess decay} in the special case
$\nabla\bar{h}(0)=0$.

\subsection{The general case}

In general $\nabla\bar{h}(0)$ may not be $0$, and we only have an estimate as in \eqref{8.01}. Here we show how to reduce this problem to the special case treated in the previous subsection.

For each $\varepsilon$, take a rotation $T_\varepsilon\in SO(n+1)$
so that
\begin{equation}\label{8.4.1}
T_\varepsilon
e_{n+1}=e_\varepsilon:=\frac{e_{n+1}+\delta_\varepsilon\nabla\bar{h}(0)}
{\left(1+\delta_\varepsilon^2|\nabla\bar{h}(0)|^2\right)^{1/2}}.
\end{equation}
Next define
\[\tilde{u}_\varepsilon(X):=u_\varepsilon(T_\varepsilon X),\]
which is still a solution of \eqref{equation} in $\mathcal{B}_4$.

By \eqref{8.01},
\begin{equation}\label{7.5}
|e_\varepsilon-e_{n+1}|\leq C\delta_\varepsilon.
\end{equation}
We can also choose $T_\varepsilon$ so that it satisfies the
following estimates.
\begin{lem}
\begin{equation}\label{7.6}
\|T_\varepsilon-I\|\leq C\delta_\varepsilon,\ \ \ \  \|\Pi\circ
T_\varepsilon-I_{\R^n}\|\leq C\delta_\varepsilon^2.
\end{equation}
\end{lem}
\begin{proof}
Choose a basis in $\R^n$ so that
$\nabla\bar{h}(0)=|\nabla\bar{h}(0)|e_n$. We have defined
$T_\varepsilon e_{n+1}$. Now take
\[T_\varepsilon e_i=e_i,\ \ \ \mbox{for}\ 1\leq i\leq n-1,\]
\[T_\varepsilon e_n=\frac{e_n-\delta_\varepsilon|\nabla\bar{h}(0)|e_{n+1}}
{\left(1+\delta_\varepsilon^2|\nabla\bar{h}(0)|^2\right)^{1/2}}.\]
In particular, $T_\varepsilon$ is only a rotation in the $(e_n,
e_{n+1})$-plane.

Since $\delta_\varepsilon|\nabla\bar{h}(0)|\leq 1/2$ (recall that
$\delta_\varepsilon$ converges to $0$ and we have a universal bound
on $|\nabla\bar{h}(0)|$), the first inequality in \eqref{7.6} can be
directly verified. For the second one, first we have
\begin{eqnarray*}
|\Pi\circ T_\varepsilon e_n-e_n|&=&\Big|\frac{e_n}{\left(1+\delta_\varepsilon^2|\nabla\bar{h}(0)|^2\right)^{1/2}}-e_n\Big|\\
&=&1-\frac{1}{\left(1+\delta_\varepsilon^2|\nabla\bar{h}(0)|^2\right)^{1/2}}\\
&\leq&C\delta_\varepsilon^2|\nabla\bar{h}(0)|^2.
\end{eqnarray*}
For $1\leq i\leq n-1$, we have $\Pi\circ T_\varepsilon e_i=e_i$.
This finishes the proof.
\end{proof}

Similar to $\nu_\varepsilon$, define the unit normal vector
$\tilde{\nu}_\varepsilon$ associated to $\tilde{u}_\varepsilon$ as
in Section 2. We claim that
\begin{lem}\label{lem 7.2}
There exists a universal constant $C$ such that
\[\int_{\mathcal{C}_{3/4}}\left[1-\left(\tilde{\nu}_\varepsilon\cdot e_{n+1}\right)^2\right]
\varepsilon|\nabla\tilde{u}_\varepsilon|^2\leq
C\delta_\varepsilon^2.\]
\end{lem}
\begin{proof}
First by noting \eqref{7.6} and a change of variables, we have
\begin{eqnarray}\label{7.7}
&&\int_{\mathcal{C}_{3/4}}\left[1-\left(\tilde{\nu}_\varepsilon\cdot e_{n+1}\right)^2\right]\varepsilon|\nabla\tilde{u}_\varepsilon|^2      \nonumber \\
&=&\int_{T_\varepsilon^{-1}\left(B_{3/4}\times\{|x_{n+1}|<1/2\}\right)}
\left[1-\left(\nu_\varepsilon\cdot
e_\varepsilon\right)^2\right]\varepsilon|\nabla
u_\varepsilon|^2+O(e^{-c/\varepsilon})\\ \nonumber
&\leq&\int_{\mathcal{C}_1}\left[1-\left(\nu_\varepsilon\cdot
e_\varepsilon\right)^2\right]\varepsilon|\nabla
u_\varepsilon|^2+O(e^{-c/\varepsilon}),
\end{eqnarray}
where $O(e^{-c/\varepsilon})$ represents the contribution from the
part near $B_1\times\{\pm1\}$ where Proposition \ref{prop
exponential decay} applies.

By \eqref{8.4.1},
\begin{eqnarray*}
1-\left(\nu_\varepsilon\cdot e_\varepsilon\right)^2&\leq& 1-\left(\nu_\varepsilon\cdot e_{n+1}\right)^2+2\left(\nu_\varepsilon\cdot e_{n+1}\right)^2\left(1-\frac{1}{\left(1+\delta_\varepsilon^2|\nabla \bar{h}(0)|^2\right)^{1/2}}\right)\\
&&+2\delta_\varepsilon\big|\nu_\varepsilon\cdot e_{n+1}\big|\big|\nu_\varepsilon\cdot\nabla\bar{h}(0)\big|.
\end{eqnarray*}
By definition,
\[\int_{\mathcal{C}_1}\left[1-\left(\nu_\varepsilon\cdot e_{n+1}\right)^2\right]\varepsilon|\nabla u_\varepsilon|^2=\delta_\varepsilon^2.\]
Next, by \eqref{8.01},
\[2\left(\nu_\varepsilon\cdot e_{n+1}\right)^2
\left(1-\frac{1}{\left(1+\delta_\varepsilon^2|\nabla
\bar{h}(0)|^2\right)^{1/2}}\right)\leq C\delta_\varepsilon^2.\]
Finally, by noting that
\[
|\nu_\varepsilon\cdot\nabla\bar{h}(0)|\leq|\nabla\bar{h}(0)|\left(\sum_{i=1}^n\nu_{\varepsilon,i}^2\right)^{\frac{1}{2}}
\leq C\left[1-\left(\nu_\varepsilon\cdot
e_{n+1}\right)^2\right]^{\frac{1}{2}},
\]
we can use the Cauchy inequality to derive that
\begin{eqnarray*}
&&\delta_\varepsilon\int_{\mathcal{C}_1}\big|\nu_\varepsilon\cdot e_{n+1}\big|\big|\nu_\varepsilon\cdot\nabla\bar{h}(0)\big|\varepsilon|\nabla u_\varepsilon|^2\\
&\leq&C\delta_\varepsilon\left(\int_{\mathcal{C}_1}|\nu_\varepsilon\cdot
e_{n+1}|^2\varepsilon|\nabla u_\varepsilon|^2\right)^{\frac{1}{2}}
\left(\int_{\mathcal{C}_1}\left[1-\left(\nu_\varepsilon\cdot e_{n+1}\right)^2\right]\varepsilon|\nabla u_\varepsilon|^2\right)^{\frac{1}{2}}\\
&\leq& C\delta_\varepsilon^2.
\end{eqnarray*}

Putting these together we get
\[\int_{\mathcal{C}_1}\left[1-\left(\nu_\varepsilon\cdot e_\varepsilon\right)^2\right]\varepsilon|\nabla u_\varepsilon|^2\leq C\delta_\varepsilon^2.\]
Substituting this into \eqref{7.7}  and noting \eqref{absurd
assumption 3} we finish the proof.
\end{proof}

With this lemma in hand, we can proceed as before to construct the
Lipschitz functions $\tilde{h}^t_\varepsilon$, and prove that
$\delta_\varepsilon^{-1}\left(\tilde{h}^t_\varepsilon-\tilde{\lambda}_\varepsilon\delta_\varepsilon\right)$
converge to a harmonic function $\tilde{h}$ (the constant
$\tilde{\lambda}_\varepsilon$ is defined as $\lambda_\varepsilon$),
weakly in $H^1(B_{3/4})$ and strongly in $L^2(B_{3/4})$.

However by the definition of $\tilde{u}_\varepsilon$, the graph of $\tilde{h}^t_\varepsilon$ is only a rotation of the one of
$h^t_\varepsilon$. More precisely, for any $x\in B_{3/4}$ and $t\in(-1+b,1-b)$,
\[\frac{\tilde{h}^t_\varepsilon(x)+\delta_\varepsilon\nabla\bar{h}(0)\cdot x}{\left(1+\delta_\varepsilon^2|\nabla\bar{h}(0)|^2\right)^{1/2}}=h_\varepsilon^t(\Pi\circ T_\varepsilon(x,\tilde{h}^t_\varepsilon(x))).\]

In fact, because $\tilde{u}_\varepsilon(X)=t$ if and only if
$T_\varepsilon X\in u_\varepsilon^{-1}(t)$,
$x_{n+1}=\tilde{h}^t_\varepsilon(x)$ if and only if
\[\left(T_\varepsilon X\right)_{n+1}=h_\varepsilon^t\left(\Pi\circ T_\varepsilon X\right),\]
which can be written as
\[\frac{x_{n+1}+\delta_\varepsilon \nabla \bar{h}(0)\cdot x}{\left(1+\delta_\varepsilon^2|\nabla\bar{h}(0)|^2\right)^{1/2}}=h_\varepsilon^t\left(x_1,\cdots,x_{n-1},\frac{x_n-\delta_\varepsilon |\nabla \bar{h}(0)|x_{n+1}}{\left(1+\delta_\varepsilon^2|\nabla\bar{h}(0)|^2\right)^{1/2}}\right).\]
From this we deduce that
\begin{eqnarray*}
\tilde{h}^t_\varepsilon(x)&=&h_\varepsilon^t\left(x_1,\cdots,x_{n-1},x_n-\delta_\varepsilon |\nabla \bar{h}(0)|\tilde{h}^t_\varepsilon(x)+O(\delta_\varepsilon^2)\right)-\delta_\varepsilon \nabla\bar{h}(0)\cdot x+O(\delta_\varepsilon^2)\\
&=&h_\varepsilon^t(x)-\delta_\varepsilon \nabla\bar{h}(0)\cdot x+o(\delta_\varepsilon).
\end{eqnarray*}
Here we have used the facts that the Lipschitz constant of
$h_\varepsilon^t$ is smaller than $1/2$ (by its construction), and
the sup bound of $\tilde{h}^t_\varepsilon$ goes to $0$ as
$\varepsilon\to0$ (by Proposition \ref{prop exponential decay}).

Hence
$\tilde{\lambda}_\varepsilon-\lambda_\varepsilon=o_\varepsilon(1)$,
and
\begin{eqnarray*}
\tilde{h}(x)=\lim_{\varepsilon\to0}\left[\frac{\tilde{h}_\varepsilon^t}{\delta_\varepsilon}-\tilde{\lambda}_\varepsilon\right]
&=&\lim_{\varepsilon\to0}\left[\frac{h_\varepsilon^t}{\delta_\varepsilon}-\lambda_\varepsilon-\nabla\bar{h}(0)\cdot x\right]\\
&=&\bar{h}(x)-\nabla\bar{h}(0)\cdot x.
\end{eqnarray*}
Combined with Proposition \ref{prop harmonic limit}, this implies
that $\tilde{h}$ is a harmonic function in $B_{3/4}$ satisfying
$\nabla\tilde{h}(0)=0$. Then we can proceed as in the previous
subsection. By choosing a smaller $\theta$ to incorporate the
constant $C$ appearing in Lemma \ref{lem 7.2}, for all $\varepsilon$
small,
\begin{equation}\label{8.3.2}
\theta^{-n}\int_{\mathcal{C}_\theta}\left[1-(\tilde{\nu_\varepsilon}\cdot
e_{n+1})^2\right]
\varepsilon|\nabla\tilde{u}_\varepsilon|^2\leq\frac{\theta}{2C}\delta_\varepsilon^2.
\end{equation}
Here $C$ is the constant appearing in Lemma \ref{lem 7.2}, due to a
change of variable associated to the rotation $T_\varepsilon$. After
rotating back, this contradicts \eqref{absurd assumption 2} and
finishes the proof of Theorem \ref{thm tilt excess decay}.

\part{Uniform $C^{1,\alpha}$ regularity of intermediate layers}

\section{Statement}
\numberwithin{equation}{section}
 \setcounter{equation}{0}

In this part we prove the following local uniform $C^{1,\alpha}$
regularity for intermediate layers. This parallels Allard's
$\varepsilon$-regularity theorem for stationary varifolds.
\begin{thm}\label{main result loc}
For any $b\in(0,1)$, there exist five universal constants
$\varepsilon_A,\tau_A, \alpha_A\in(0,1)$ and $R_A,K_A$ so that the
following holds. Let $u_\varepsilon$ be a solution of
\eqref{equation} with $\varepsilon\leq \varepsilon_A$, defined in
$\mathcal{B}_{R_A}$, satisfying $|u_\varepsilon(0)|\leq 1-b$ and
\begin{equation}\label{close to plane}
R_A^{-n}\int_{\mathcal{B}_{R_A}}\frac{\varepsilon}{2}|\nabla u_\varepsilon|^2+\frac{1}{\varepsilon}W(u_\varepsilon)\leq \left(1+\tau_A\right)\omega_n\sigma_0.
\end{equation}
Then there exists a hyperplane, say $\R^n$ (after a suitable rotation), such that for any $t\in(-1+b,1-b)$,
 $\{u_\varepsilon=t\}\cap\mathcal{C}_1$ is a $C^{1,\alpha_A}$ hypersurface, which
is represented by the graph of the function
$x_{n+1}=h^t_\varepsilon(x)$, with
\[\|h_\varepsilon^t\|_{C^{1,\alpha_A}(B_1)}\leq K_A.\]
\end{thm}

Assume the limit varifold $V$ of $u_\varepsilon$ satisfies the
assumptions in Allard's $\varepsilon$-regularity theorem at the
origin $0$. Hence it is a smooth minimal hypersurface with unit
density near $0$. By enlarging this minimal hypersurface around $0$,
the assumptions in this theorem are fulfilled and this theorem
applies, which says, in a neighborhood of $0$, intermediate layers
of $u_\varepsilon$ are hypersurfaces with uniformly $C^{1,\alpha_A}$
bound and they converge to the minimal hypersurface in a
$C^{1,\alpha_A}$ manner.

To prove this theorem, we first use Theorem \ref{thm tilt excess
decay} to obtain a Morrey type bound. As explained in Section 1, due
to the assumption $\delta_\varepsilon\geq K_0\varepsilon$ in Theorem
\ref{thm tilt excess decay}, this Morrey type bound does not give
the required $C^{1,\alpha_A}$ regularity. It only says that at every
scale up to $O(\varepsilon)$, $\{u_\varepsilon=t\}$ is close to a
{\em fixed} hyperplane, i.e. a kind of Lipschitz regularity for
$\{u_\varepsilon=t\}$ up to $O(\varepsilon)$ scales. This is already
sufficient for the proof of Theorem \ref{main result}, which is
given in Section 11. The proof of Theorem \ref{main result loc} will
be completed in Section 12, and it uses the intermediate results
established in Section 11.

\section{A Morrey type bound}
\numberwithin{equation}{section}
 \setcounter{equation}{0}

In this section $u_\varepsilon$ denotes a fixed solution satisfying all of the assumptions in Theorem \ref{main result loc}.
Here we prove
\begin{lem}\label{lem 11.1}
There exist two universal constants $K_1$ and $K_2$ so that the
following holds. For any $X_0\in
\{|u_\varepsilon|\leq1-b\}\cap\mathcal{B}_1$ and ball
$\mathcal{B}_r(X_0)$ with $r\in (K_1\varepsilon,\theta)$, we can
find a unit vector $e_r(X_0)$ so that
\begin{equation}\label{8.1}
r^{-n}\int_{\mathcal{B}_r(X_0)}\left[1-\left(\nu_\varepsilon\cdot
e_r(X_0)\right)^2\right]\varepsilon|\nabla u_\varepsilon|^2\leq
K_2^2\max\{\varepsilon^2r^{-2}, \delta_0^2r^\alpha\}.
\end{equation}
Here $\alpha=\frac{|\log \theta/2|}{|\log\theta|}\in(1,2)$.
\end{lem}

For convenience, we shall replace the cylinders $\mathcal{C}_2$ and $\mathcal{C}_\theta$ in Theorem \ref{thm tilt excess decay}
by balls $\mathcal{B}_1$ and $\mathcal{B}_\theta$ respectively. This may change the constants in that theorem by a factor, which however only depends on the dimension $n$ and does not affect our argument too much.

By the monotonicity formula (Proposition \ref{monotonicity formula}) and \eqref{close to plane}, if $R_A$ is sufficiently large, for any $X\in\mathcal{B}_1$ and $r\in(0,R_A-1)$,
\begin{equation}\label{11.2.1}
r^{-n}\int_{\mathcal{B}_r(X)}\frac{\varepsilon}{2}|\nabla u_\varepsilon|^2+\frac{1}{\varepsilon}W(u_\varepsilon)
\leq (1+2\tau_A)\omega_n\sigma_0.
\end{equation}

If $\tau_A$ is sufficiently small, we can applying Proposition
\ref{prop exponential decay} to $u_\varepsilon(rX)$, which gives
\begin{lem}\label{lem close to plane}
For any $\delta>0$, there exists a $K(\delta)$ so that the following
holds. For any $X\in\{|u_\varepsilon|\leq1-b\}\cap\mathcal{B}_1$ and
$r\in( K(\delta)\varepsilon,1)$, there exists a hyperplane $P_r(X)$
such that
\[\mbox{dist}_H\left(\{u_\varepsilon=u_\varepsilon(X)\}\cap\mathcal{B}_r(X),P_r(X)\cap\mathcal{B}_r(X)\right)\leq \delta r.\]
\end{lem}

By Lemma \ref{lem excess small}, if $r\geq K_1\varepsilon$ ($K_1$ a
constant determined by Lemma \ref{lem excess small}) and $\tau_A$ is
sufficiently small, the excess with respect to $P_r(X)$ (with unit
normal vector $e_r(X)$)
\begin{equation}\label{10.0.3}
E(r;X,u_\varepsilon,P_r(X))\leq \delta_0^2,
\end{equation}
with $\delta_0$ as in Theorem \ref{thm tilt excess decay}. Note that
in \eqref{10.0.3} it is integrated on $\mathcal{B}_r(X)$, not on a
cylinder.

Now Theorem \ref{thm tilt excess decay} applies. In the current
setting it reads as
\begin{lem}\label{lem tilt excess decay}
If $E(r;X,u_\varepsilon,P_r(X))\geq K_0^2r^{-2}\varepsilon^2$, there
exists another hyperplane $\tilde{P}_r(X)$ such that
\[E(\theta r;X,u_\varepsilon,\tilde{P}_r(X))\leq\frac{\theta}{2}E(r;X,u_\varepsilon,P_r(X)).\]
Here one of the unit normal vector of $\tilde{P}_r(X)$, $\tilde{e}_r(X)$ satisfies
\[\|\tilde{e}_r(X)-e_r(X)\|\leq CE(r;X,u_\varepsilon,P_r(X))^{\frac{1}{2}}.\]
\end{lem}
The constant $\theta$ may be different from the one in Theorem
\ref{thm tilt excess decay}, but we still have $\theta<1$.

With this lemma in hand we can prove Lemma \ref{lem 11.1}. The
following proof is similar to the one of \cite[Theorem 2.3]{W2}.
\begin{proof}[Proof of Lemma \ref{lem 11.1}]
Assume $X_0=0$. For $k\geq 0$, let $r_k=\theta^{k}$. Define
\[E_k:=\min_{e\in\mathbb{S}^{n}}\varepsilon^{-2}r_k^{2-n}
\int_{\mathcal{B}_{r_k}}\left[1-\left(\nu_\varepsilon\cdot
e\right)^2\right]\varepsilon|\nabla u_\varepsilon|^2.\] Take a unit
vector $\bar{e}_k$ to attain this minima.

As in \eqref{10.0.3}, for all $r_k\geq K_1\varepsilon$,
\begin{equation}\label{11.1}
E_k\leq\delta_0^2\varepsilon^{-2}r_k^2.
\end{equation}
Lemma \ref{lem tilt excess decay} implies that, once $E_k\geq
K_0^2$, then
\begin{equation}\label{11.2}
E_{k+1}\leq\frac{\theta^3}{2}E_k.
\end{equation}
Moreover, by the definition of $E_k$, we always have
\begin{equation}\label{11.3}
E_{k+1}\leq\theta^{2-n}E_k.
\end{equation}

Let $k_1\in\mathbb{N}$ be the unique number satisfying
$\theta^{k_1}\in[K_1\varepsilon, K_1\theta^{-1}\varepsilon)$.

Now we derive the claimed bound on $E_k$ from \eqref{11.1}-\eqref{11.3}, for $k\leq k_1$. Let
$k_0$ be the smallest number such that, for all $k>k_0$,
\begin{equation}\label{11.0.1}
E_k\leq K_0^2\theta^{2-n}.
\end{equation}
As we will see below, this is well-defined.

If $k_0=0$, for any $0\leq k\leq k_1$,
\begin{equation}\label{11.0.0.1}
\int_{\mathcal{B}_{r_k}}\left[1-\left(\nu_\varepsilon\cdot
\bar{e}_k\right)^2\right]\varepsilon|\nabla u_\varepsilon| \leq
K_0^2\theta^{2-n}\varepsilon^2r_k^{n-2}.
\end{equation}
 This can be extended to
those $r\in[K_1\varepsilon,\theta)$ by choosing a (unique) $k$ so
that $r\in[r_{k+1},r_k)$.

Next we assume there exists a $\tilde{k}>0$ such that
$E_{\tilde{k}}\geq K_0^2\theta^{2-n}$.  By \eqref{11.3},
$E_{\tilde{k}-1}\geq K_0^2$. Then \eqref{11.2} applies, which says
\[E_{\tilde{k}-1}\geq \frac{2}{\theta^3}E_{\tilde{k}}.\]
In particular,
\[E_{\tilde{k}-1}\geq E_{\tilde{k}}\geq K_0\theta^{2-n}.\]
With this estimate we can repeat the above procedure to obtain that,
for all $i\in[0,\tilde{k})$,
\[E_i\geq \frac{2}{\theta^3}E_{i+1}\geq K_0^2\theta^{2-n}.\]
From this we see $k_0$ is well defined.

The above decay estimate implies that, for all $i\leq k_0$,
\[E_i\leq \left(\frac{\theta^3}{2}\right)^iE_0,\]
in other words,
\begin{equation}\label{11.0.2}
\int_{\mathcal{B}_{r_i}}\left[1-\left(\nu_\varepsilon\cdot \bar{e}_i\right)^2\right]\varepsilon|\nabla u_\varepsilon|
\leq \delta_0^2r_i^{n+\alpha}.
\end{equation}
This estimate can also be extended to those $r\in[r_{k_0},\theta)$ by choosing an $i$ so that $r\in[r_{i+1},r_i)$.

In conclusion, for $r\in[r_{k_0},\theta)$, we have the estimate
\eqref{11.0.2}, and for $r\in(K_1\varepsilon,r_{k_0})$
\eqref{11.0.0.1} applies. By choosing a suitable universal constant
$K_2$, \eqref{8.1} follows from these two estimates.
\end{proof}

Next we show that $e_r(X_0)$ can be replaced by a fixed unit vector
(independent of $r$).
\begin{lem}\label{lem 11.2}
For any $\sigma>0$, there exist two constants $K_3:=K_3(\sigma)$ and
$K_4$ ($K_4$ universal, independent of $\sigma$) so that the
following holds. For any $X_0\in
\{|u_\varepsilon|\leq1-b\}\cap\mathcal{B}_1$ and ball
$\mathcal{B}_r(X_0)$ with $r\in(K_3\varepsilon,\theta)$, there
exists a unit vector $e(X_0)$ such that
\begin{equation}\label{11.6}
r^{-n}\int_{\mathcal{B}_r(X_0)}\left[1-\left(\nu_\varepsilon\cdot
e(X_0)\right)^2\right]\varepsilon|\nabla u_\varepsilon|^2 \leq
\sigma+K_4\delta_0r^{\alpha/2}.
\end{equation}
Here $e(X_0)$ is independent of $r\in(K_3\varepsilon,\theta)$.
\end{lem}
\begin{proof}
Keep notations as in the proof of Lemma \ref{lem 11.1}.

For any $r\in(K_1\varepsilon,\theta)$, combining Remark \ref{rmk 1}
and Lemma \ref{lem lower bound for energy}, we get
\begin{eqnarray*}
&&\int_{\mathcal{B}_{2r}(X_0)}\left[1-\left(\nu_\varepsilon\cdot
e_{2r}(X_0)\right)^2\right]\varepsilon|\nabla u_\varepsilon|^2 +
\int_{\mathcal{B}_r(X_0)}\left[1-\left(\nu_\varepsilon\cdot e_r(X_0)\right)^2\right]\varepsilon|\nabla u_\varepsilon|^2\\
&\geq&c\int_{\mathcal{B}_r(X_0)}\left[\mbox{dist}_{\mathbb{RP}^n}
(\nu_\varepsilon, e_r(X_0))^2+\mbox{dist}_{\mathbb{RP}^n}(\nu_\varepsilon, e_{2r}(X_0))^2\right]\varepsilon|\nabla u_\varepsilon|^2\\
&\geq&c\mbox{ dist}_{\mathbb{RP}^n}(e_{2r}(X_0), e_r(X_0))^2\int_{\mathcal{B}_r(X_0)}\varepsilon|\nabla u_\varepsilon|^2\\
&\geq&c\mbox{ dist}_{\mathbb{RP}^n}(e_{2r}(X_0), e_r(X_0))^2r^n.
\end{eqnarray*}

For $k<k_0$, by Lemma \ref{lem 11.1} this gives
\[\mbox{dist}_{\mathbb{RP}^n}(e_{k+1}(X_0), e_k(X_0))\leq K_2\delta_
0r_k^{\frac{\alpha}{2}}=K_2\delta_0\theta^{\frac{\alpha}{2}k}.\]
Summing in $i$ from $k$ to $k_0$, we see
\begin{equation}\label{11.0.3}
\mbox{dist}_{\mathbb{RP}^n}(e_{k_0}(X_0), e_k(X_0))\leq
\frac{K_2}{1-\theta^{\alpha/2}}\delta_0\theta^{\frac{\alpha}{2}k}=\frac{K_2}{1-\theta^{\alpha/2}}\delta_0r_k^{\frac{\alpha}{2}},
\ \ \ \ \forall k<k_0.
\end{equation}

For $k\in[k_0, k_1)$, we have
\begin{equation}\label{10.1.1}
\mbox{dist}_{\mathbb{RP}^n}(e_{k+1}(X_0), e_k(X_0))\leq
K_2\varepsilon r_k^{-1}=K_2\varepsilon\theta^{-k}.
\end{equation}
Let $k_2\leq k_1$ be the largest number satisfying
\begin{equation}\label{10.1.2}
\frac{K_2}{\theta^{-1}-1}\varepsilon\theta^{-k_2-1}+K_2^2\varepsilon^2\theta^{-2k_2-2}\leq
\frac{\theta^n}{4}\sigma.
 \end{equation}
Note that there exists a constant $K_3(\sigma)$ such that
\[r_{k_2}=\theta^{k_2}\leq K_3(\sigma)\varepsilon.\]

 Summing \eqref{10.1.1} from $k$ to $k_2$, we get
\begin{equation}\label{11.0.4}
\mbox{dist}_{\mathbb{RP}^n}(e_{k_2}(X_0), e_k(X_0))\leq
\varepsilon\frac{K_2}{\theta^{-1}-1}\theta^{-k_2-1} \leq
\frac{\theta^n}{4}\sigma, \ \ \ \ \forall k_0\leq k\leq k_2.
\end{equation}
In particular,
\begin{equation}\label{11.0.5}
\mbox{dist}_{\mathbb{RP}^n}(e_{k_2}(X_0), e_{k_0}(X_0))\leq
\frac{\theta^n}{4}\sigma.
\end{equation}

Let $e(X_0)=e_{k_2}(X_0)$, by \eqref{11.0.3}-\eqref{11.0.5} we
obtain, for any $k\in(0,k_2)$,
$$\mbox{dist}_{\mathbb{RP}^n}(e_k(X_0),e(X_0))\leq \frac{\theta^n}{4}\sigma+\frac{K_2}{1-\theta^{\alpha/2}}\delta_0r_k^{\frac{\alpha}{2}}.$$

For any $k\geq 0$, similar to Remark \ref{rmk 1}, we have
\[1-\left(\nu_\varepsilon\cdot e(X_0)\right)^2\leq\left[1-\left(\nu_\varepsilon\cdot e_k\right)^2\right]
+2\mbox{dist}_{\mathbb{RP}^n}(e_k,e(X_0)).\]
 Together with
\eqref{8.1} and \eqref{10.1.2}, this gives
\[r_k^{-n}\int_{\mathcal{B}_{r_k}}\left[1-\left(\nu_\varepsilon\cdot e(X_0)\right)^2\right]\varepsilon|\nabla u_\varepsilon|^2
\leq
\frac{3\theta^n}{4}\sigma+\left(\frac{2K_2}{1-\theta^{\alpha/2}}+K_2^2\right)\delta_0
 r_k^{\alpha/2}.\]
For any $r\in(K_3\varepsilon,\theta)$, by choosing a $k$ so that
$r\in(r_k,r_{k+1}]$, we obtain
\begin{equation}\label{11.7}
r^{-n}\int_{\mathcal{B}_r}\left[1-\left(\nu_\varepsilon\cdot
e(X_0)\right)^2\right]\varepsilon|\nabla u_\varepsilon|^2 \leq
\sigma+\theta^{-n}\left(\frac{2K_2}{1-\theta^{\alpha/2}}+K_2^2\right)\delta_0
 r_k^{\alpha/2}.
\end{equation}
By taking
\[K_4:=\theta^{-n}\left(\frac{2K_2}{1-\theta^{\alpha/2}}+K_2^2\right),\]
which is indeed a universal constant and does not depend on
$\sigma$, we get \eqref{11.6}.
\end{proof}

The following result will be used in the proof of Lipschitz regularity of $\{u_\varepsilon=0\}$.
\begin{coro}\label{coro 11.3}
For any $X_0\in\{u_\varepsilon=0\}\cap\mathcal{B}_1$,
\[|e(X_0)-e_{n+1}|\leq C\left(\sigma^{1/2}+\delta_0^{1/2}\right).\]
\end{coro}
\begin{proof}
By taking $r=\theta$ in \eqref{11.6}, we have
\[
\theta^{-n}\int_{\mathcal{B}_{\theta}(X_0)}\left[1-\left(\nu_\varepsilon\cdot
e(X_0)\right)^2\right] \varepsilon|\nabla u_\varepsilon|^2\leq
\sigma+K_4\delta_0.
\]
On the other hand, by \eqref{close to plane} and Lemma \ref{lem excess small}, we also have
\[\theta^{-n}\int_{\mathcal{B}_{\theta}(X_0)}\left[1-\left(\nu_\varepsilon\cdot e_{n+1}\right)^2\right]
\varepsilon|\nabla u_\varepsilon|^2\leq C\delta_0^2.\] Similar to
the proof of the previous lemma, combining these two and using Lemma
\ref{lem lower bound for energy}, we get
\begin{eqnarray*}
&&\mbox{dist}_{\mathbb{RP}^n}(e(X_0),e_{n+1})^2\\
&\leq &\theta^{-n}\int_{\mathcal{B}_{1/4}(X_0)}\left[1-\left(\nu_\varepsilon\cdot e(X_0)\right)^2\right]
\varepsilon|\nabla u_\varepsilon|^2+\theta^{-n}\int_{\mathcal{B}_{1/4}(X_0)}\left[1-\left(\nu_\varepsilon\cdot e_{n+1}\right)^2\right]
\varepsilon|\nabla u_\varepsilon|^2\\
&\leq & C\left(\sigma+\delta_0\right).
\end{eqnarray*}

Finally, we can fix $e(X_0)$ so that it points to the above. Thus
the estimate on the distance in $\mathbb{RP}^n$ can be lifted to an
estimate in $\mathbb{S}^n$.
\end{proof}

What we have proved can be roughly stated as follows: level sets of
$u_\varepsilon$ are Lipschitz graphs in the form of
$x_{n+1}=h_\varepsilon(x)$ up to the scale $K_3\varepsilon$.
However, this may break down for smaller scales, because in Lemma
\ref{lem 11.2} $K_3$ depends on $\sigma$. To obtain further control
on the scale smaller than $K_3\varepsilon$, we first give a direct
proof of Theorem \ref{main result} and then use this to prove the
full regularity of level sets of $u_\varepsilon$.

\section{A direct proof of Theorem \ref{main result}}
\numberwithin{equation}{section}
 \setcounter{equation}{0}

This section is devoted to a direct proof of Theorem \ref{main result}. In fact, we prove something more.
\begin{thm}\label{main result 3}
Suppose that $u$ is a smooth solution of \eqref{equation 0} on $\R^{n+1}$, satisfying
\begin{equation}\label{limit at infinity}
\lim_{R\to+\infty}R^{-n}\int_{\mathcal{B}_R}\frac{1}{2}|\nabla u|^2+W(u)\leq\left(1+\tau_A\right)\omega_n\sigma_0.
\end{equation}
Then there exists a unit vector $e$ and a constant $t\in\R$ such
that $u(X)\equiv g(e\cdot X+t)$.
\end{thm}
In the following we will show that if $u$ is a minimizing solution of \eqref{equation 0} on $\R^{n+1}$, where $n\leq 6$, then \eqref{limit at infinity} is satisfied. Thus Theorem \ref{main result} is a corollary of this theorem.

Since $u$ is an entire solution, by the main result of
\cite{Modica}, $u$ satisfies the Modica inequality and hence the
monotonicity formula, Proposition \ref{monotonicity formula} for any
$X\in\R^{n+1}$ and $r>0$. This monotonicity ensures the existence of
the limit in \eqref{limit at infinity}. It also implies that, for
any ball $\mathcal{B}_R(X)\subset\R^{n+1}$,
\[R^{-n}\int_{\mathcal{B}_R(X)}\frac{1}{2}|\nabla u|^2+W(u)\leq\left(1+\tau_A\right)\omega_n\sigma_0.\]

With this bound, we can study the asymptotic behavior of $u$ through
the scaling
\[u_\varepsilon(X):=u(\varepsilon^{-1}X).\]
As before, by Hutchinson-Tonegawa theory, the varifolds
$V_\varepsilon$ associated to $u_\varepsilon$ converge to a
stationary varifold $V$ with integer multiplicity.

Furthermore, we claim that
\begin{prop}\label{prop 11.2}
$V$ is a cone with respect to the origin $0$.
\end{prop}
\begin{proof}
This is because for any $R>0$, by the convergence of $\|V_\varepsilon\|$ and \eqref{discrepancy},
\begin{eqnarray}\label{12.1.1}
R^{-n}\|V\|(\mathcal{B}_R)&=&\lim_{\varepsilon\to 0}R^{-n}\|V_\varepsilon\|(\mathcal{B}_R)  \nonumber\\
&=&\lim_{\varepsilon\to 0}R^{-n}\int_{\mathcal{B}_R}\frac{\varepsilon}{2}|\nabla u_\varepsilon|^2+\frac{1}{\varepsilon}W(u_\varepsilon)
\quad\mbox{(by the definition of $V_\varepsilon$)}\\
&=&\lim_{\varepsilon\to
0}\left(\varepsilon^{-1}R\right)^{-n}\int_{\mathcal{B}_{\varepsilon^{-1}R}}\frac{1}{2}|\nabla
u|^2+W(u). \quad\mbox{(by the definition of $u_\varepsilon$)}
\nonumber
\end{eqnarray}
In the last line, the existence of the limit follows from the energy
bound \eqref{limit at infinity} and the monotonicity formula,
Proposition \ref{monotonicity formula}. Note that this limit is
independent of $R$. Then by the monotonicity formula for stationary
varifolds (cf. \cite[Theorem 6.3.2]{Lin}), we deduce that $V$ is a
cone with respect to the origin.
\end{proof}

By \eqref{limit at infinity} and \eqref{12.1.1},
\[\|V\|(\mathcal{B}_1)\leq \left(1+\tau_A\right)\omega_n\sigma_0.\]
Hence we can apply Allard's $\varepsilon$-regularity theorem to
deduce that $\mbox{spt}\|V\|$ is a smooth hypersurface in a
neighborhood of the origin. Then by the previous proposition,
$\mbox{spt}\|V\|$ must be a hyperplane and $V$ is the standard
varifold associated to this plane with unit density.

Let
\[
\Phi_\varepsilon:=g_\varepsilon^{-1}\circ u_\varepsilon\] be the
distance type function (see Appendix A). Combining this blowing down
analysis and Proposition \ref{prop limit of distance}, we get
\begin{prop}\label{prop blowing down the distance fct}
As $\varepsilon\to0$, $\Phi_\varepsilon$ converges to (up to a
subsequence of $\varepsilon\to0$) a linear function in the form
$e\cdot X$ in $C_{loc}(\R^{n+1})$, where $e$ is a unit vector.
\end{prop}

However, this argument does not show the uniqueness of this limit. Different subsequences of $\varepsilon\to0$ may lead to different limits. To obtain the uniqueness of the blowing down limit, we use the following lemma.
\begin{lem}\label{lem 12.3}
There exists a universal constant $C$, such that for any ball
$\mathcal{B}_R(X)$ with $R\geq 1$, we can find a unit vector $e_R$
to satisfy
\begin{equation}\label{12.0.1}
\int_{\mathcal{B}_R(X)}\left[1-\left(\nu\cdot
e_R\right)^2\right]|\nabla u|^2\leq CR^{n-2}.
\end{equation}
\end{lem}
The proof is similar to the one of Lemma \ref{lem 11.1}, see also the proof of \cite[Theorem 2.3]{W2}.

Note that $e_R$ in this theorem may not be unique. In the following
we assume that for each $R>1$, such a vector $e_R$ has been fixed.

If $n=1$, as $R\to+\infty$, since $e_R$ are unit vectors, we can
take a subsequence of $R_i\to+\infty$ so that $e_{R_i}\to
e_\infty\in\mathbb{S}^1$. Assume $e_\infty=e_2$. Then by taking
limit in \eqref{12.0.1},  we get
\[\int_{\mathcal{B}_R(X)}\left(\frac{\partial u}{\partial x_1}\right)^2=0,\quad\forall R>0.\]
Thus $u(x_1,x_2)\equiv u(x_2)$.

Now consider the case $n\geq 2$.  Similar to Lemma \ref{lem 11.2}, we also have
\begin{lem}
There exists a unit vector $e_\infty$ and a universal constant $C$
such that
\begin{equation}\label{12.2}
\int_{\mathcal{B}_R(X)}\left[1-\left(\nu\cdot
e_\infty\right)^2\right]|\nabla u|^2\leq CR^{n-2},\ \ \forall R>1.
\end{equation}
\end{lem}

For the blowing down sequence $u_\varepsilon$, \eqref{12.2} implies that
\begin{equation}\label{11.1.1}
\int_{\mathcal{B}_1}\left[1-\left(\nu_\varepsilon\cdot
e_\infty\right)^2\right]\varepsilon|\nabla u_\varepsilon|^2\leq
C\varepsilon^{2}.
\end{equation}
 Note that this estimate just says
it does not satisfy the assumption
$\delta_\varepsilon\gg\varepsilon$ in Theorem \ref{thm tilt excess
decay}.

 For any $\eta\in C_0^\infty(\R^{n+1})$, let
$\Phi(X,S)=\eta(X)^2<Se_\infty, e_\infty>\in
C_0^\infty(\R^{n+1}\times G(n))$. Passing to the limit in
\eqref{11.1.1} gives
\[0=\lim_{\varepsilon\to0}<V_\varepsilon,\Phi>=<V,\Phi>.\]
Thus for $\|V\|$ a.a. $X$, the tangent plane of $V$ at $X$ is the hyperplane orthogonal to $e_\infty$.  It can be directly checked that $V$ must be the standard varifold associated to this hyperplane. (This can also be seen by noting that we have proved that $\mbox{spt}\|V\|$ is a hyperplane.)

The uniqueness of $V$ also implies that, the limit of
$\Phi_\varepsilon$ in Proposition \ref{prop blowing down the
distance fct} is independent of the choice of subsequences of
$\varepsilon\to 0$, i.e.,
\[\Phi_\varepsilon\rightarrow e_\infty\cdot X,\ \ \ \mbox{in}\ C_{loc}(\R^{n+1}).\]
Without loss of generality, assume $e_\infty=e_{n+1}$.

Then by Theorem \ref{thm C1 convergence interior}, for any $\delta>0$,
\[\nabla\Phi_\varepsilon\to e_{n+1}, \ \ \ \mbox{uniformly on}\ \mathcal{B}_1\cap\{|x_{n+1}|>\delta\}.\]

By compactness, this still holds true if the base point is replaced
by any point $X_0\in\{u=0\}$. Thus we arrive at
\begin{lem}
For any $\delta>0$, there exists an $L(\delta)$ such that, for any
$X\in\{|\Phi|\geq L(\delta)\}$,
\[|\nabla\Phi(X)-e_{n+1}|\leq\delta.\]
\end{lem}

In particular, in $\{|\Phi|>L(\delta)\}$, $u$ is increasing along
directions in the cone
\[\{e: e\cdot e_{n+1}\geq \delta\}.\]

Then we can proceed as in \cite{F} to deduce that $u$ is increasing
along directions in this cone everywhere in $\R^{n+1}$. After
letting $\delta\to 0$, we deduce that for any unit vector $e$
orthogonal to $e_{n+1}$,
\[e\cdot\nabla u\geq 0,\ \ \ -e\cdot\nabla u\geq 0,\ \ \ \mbox{in}\ \R^{n+1}.\]
Thus $\frac{\partial u}{\partial x_i}\equiv 0$ in $\R^{n+1}$, for all $1\leq i\leq n$. This then implies that $u$ depends only on $x_{n+1}$.

Finally, by using \eqref{limit at infinity}, it can be checked
directly that we must have $u(X)\equiv g(x_{n+1}+t)$ for some
$t\in\R$ (see again the proof of Lemma \ref{lem B2}).

Next we prove Theorem \ref{main result}. Let $u$ be a minimizing
solution of \eqref{equation 0} on $\R^{n+1}$, where $n\leq 6$. First
we can use standard comparison functions to deduce an energy bound.
\begin{lem}
There exists a universal constant $C$ such that
\begin{equation}\label{energy bound 1}
\int_{\mathcal{B}_R(X)}\frac{1}{2}|\nabla u|^2+W(u)\leq CR^n,
\end{equation}
for any ball $\mathcal{B}_R(X)$.
 \end{lem}
As before, consider the blowing down sequence $u_\varepsilon$ and
the associated varifold $V_\varepsilon$. By \cite[Theorem 2]{H-T},
its limit varifold $V$ has unit density. In fact, in this case
$\mbox{spt}\|V\|=\partial\Omega$, where $\Omega$ has minimizing
perimeter, see \cite{Modica 2}.

Moreover, by Proposition \ref{prop 11.2}, $\partial\Omega$ is a
cone. Because the dimension $n\leq 6$, $\partial\Omega$ must be a
hyperplane, see Simons \cite{Simons}. Then \eqref{12.1.1} gives
\[\lim_{R\to +\infty}R^{-n}\int_{\mathcal{B}_R}\frac{1}{2}|\nabla u|^2+W(u)=\|V\|(\mathcal{B}_1)=\omega_n\sigma_0.\]
Hence $u$ satisfies all of the assumptions in Theorem \ref{main
result 3}. By applying Theorem \ref{main result 3} we get Theorem
\ref{main result}.

\section{The Lipschitz regularity of intermediate layers}
\numberwithin{equation}{section}
 \setcounter{equation}{0}

Now we continue the proof of Theorem \ref{main result loc}. In this section we first
prove that $\{u_\varepsilon=t\}$ can be represented by a Lipschitz graph in the $x_{n+1}$ direction. This is the consequence of Corollary \ref{coro 11.3} and
the following Lemma \ref{lem 11.5}.

Before coming to Lemma \ref{lem 11.5}, we need the following lemma,
which is an easy consequence of Theorem \ref{main result 3}.
\begin{lem}
Let $v$ be a solution of \eqref{equation 0}
in $\mathbb{R}^{n+1}$. Assume there exists a constant $\sigma$ small so that for all $r$ large,
\begin{equation}\label{11.8}
\int_{\mathcal{B}_r}\left[1-\left(\nu\cdot e_{n+1}\right)^2\right]|\nabla v|^2\leq \sigma^2 r^n,
\end{equation}
and
\begin{equation}\label{11.9}
\lim_{r\to+\infty}r^{-n}\int_{\mathcal{B}_r}\frac{1}{2}|\nabla v|^2+W(v)\leq\left(1+\tau_A\right)\sigma_0\omega_n.
\end{equation}
Then there exists a constant $t\in\R$ and a unit vector $e$
satisfying
\begin{equation}\label{11.10}
|e-e_{n+1}|\leq C\sigma,
\end{equation}
so that $v(X)\equiv g(e\cdot X+t)$.
\end{lem}
\begin{proof}
The only thing we need to check is that \eqref{11.8} implies
\eqref{11.10}. This can be directly verified by substituting
$u(X)\equiv g(e\cdot X+t)$ into \eqref{11.8}.
\end{proof}

\begin{lem}\label{lem 11.5}
For any $b\in(0,1)$, $R>1$ and $\sigma>0$ small, there exists
$\bar{R}>R$ so that the following holds. Let $v$ be a solution of
\eqref{equation 0} in $\mathcal{B}_{\bar{R}}$, satisfying
$|v(0)|\leq1-b$, the Modica inequality \eqref{Modica inequality} and
\[\bar{R}^{-n}\int_{\mathcal{B}_{\bar{R}}}\frac{1}{2}|\nabla v|^2+W(v)\leq\left(1+\tau_A\right)\sigma_0\omega_n.\]
 Suppose that for any $r\in(R,\bar{R})$,
\[\int_{\mathcal{B}_r}\left[1-\left(\nu\cdot e_{n+1}\right)^2\right]|\nabla v|^2\leq \sigma^2r^n.\]
Assuming that $v>0$ when $x_{n+1}\gg 0$. Then by denoting
$\Phi:=g^{-1}\circ v$,
\[\sup_{\mathcal{B}_{R}}|\nabla\Phi-e_{n+1}|\leq \frac{1}{4}.\]
\end{lem}
\begin{proof}
Assume by the contrary, there exists an $R>0$, a sequence of
$R_i\to+\infty$ and a sequence of solutions $v_i$ to \eqref{equation
0} defined on $\mathcal{B}_{R_i}$, satisfying $|v_i(0)|\leq1-b$, the
Modica inequality \eqref{Modica inequality},
\begin{equation}\label{13.1}
R_i^{-n}\int_{\mathcal{B}_{R_i}}\frac{1}{2}|\nabla v_i|^2+W(v_i)\leq\left(1+\tau_A\right)\sigma_0\omega_n,
\end{equation}
and
\begin{equation}\label{13.2}
\int_{\mathcal{B}_r}\left[1-\left(\nu_i\cdot e_{n+1}\right)^2\right]|\nabla v_i|^2\leq \sigma^2r^n,
\ \ \ \forall r\in(R,R_i).
\end{equation}
But
\begin{equation}\label{13.3}
\sup_{\mathcal{B}_R}|\nabla\Phi_i-e_{n+1}|>\frac{1}{4}.
\end{equation}

Then we can assume $v_i$ converges to a smooth solution $v_\infty$ on any compact set of $\R^{n+1}$. By the monotonicity formula
and \eqref{13.1}, for any $r>0$,
\[r^{-n}\int_{\mathcal{B}_r}\frac{1}{2}|\nabla v_\infty|^2+W(v_\infty)\leq\left(1+\tau_A\right)\sigma_0\omega_n.\]
Passing to the limit in \eqref{13.2} we also have
\[\int_{\mathcal{B}_r}\left[1-\left(\nu_\infty\cdot e_{n+1}\right)^2\right]|\nabla v_\infty|^2\leq \sigma^2r^n,
\ \ \ \forall r>R.\] Then by the previous lemma (noting that
$v_\infty(0)=\lim_{i\to+\infty}v_i(0)$ and $v_\infty>\gamma$ in the
part far above $\R^n$), $v_\infty(X)\equiv g(e\cdot
X+g^{-1}(v_\infty(0)))$ for some unit vector $e$ satisfying
\[|e-e_{n+1}|\leq\frac{1}{8}.\]

Consequently,
\[\Phi_i(X):=g^{-1}\circ v_i(X)\to e\cdot X\ \ \ \mbox{in}\ C^1(\mathcal{B}_R).\]
In particular, for all $i$ large,
\[\sup_{\mathcal{B}_R}|\nabla\Phi_i-e|\leq\frac{1}{4}.\]
This is a contradiction with \eqref{13.3} and we finish the proof.
\end{proof}

We can apply this lemma to $v(X):=u_\varepsilon(X_0+\varepsilon X)$,
where $u_\varepsilon$ is as in Theorem \ref{main result loc} and
$X_0\in \{u_\varepsilon=u_\varepsilon(0)\}\cap\mathcal{B}_1$.
Combined with Corollary \ref{coro 11.3} (provided $\sigma$, and then
$\varepsilon$, are sufficiently small), this results in
\begin{lem}\label{lem 13.3}
For any $X_0\in \{u_\varepsilon=u_\varepsilon(0)\}\cap\mathcal{B}_1$, $\nabla u_\varepsilon\neq 0$ in $\mathcal{B}_{K_3\varepsilon}(X_0)$ and
\[|\nu_\varepsilon-e_{n+1}|\leq \frac{1}{2} \quad \mbox{in } \mathcal{B}_{K_3\varepsilon}(X_0).\]
\end{lem}
Here we only need to note that, at the beginning we have assumed that $u_\varepsilon>u_\varepsilon(0)$ in $\{x_{n+1}>1/2\}\cap\mathcal{B}_1$.
Then by Lemma \ref{lem close to plane}, $u_\varepsilon<u_\varepsilon(0)$ in $\{x_{n+1}<-1/2\}\cap\mathcal{B}_1$.

Next by combining Lemma \ref{lem close to plane} and Lemma \ref{lem
11.2}, for any $r\geq K_3\varepsilon$,
\begin{equation}\label{13.4}
\{\left(X-X_0\right)\cdot e(X_0)\geq\frac{r}{2}\}\cap\mathcal{B}_r(X_0)\subset\{u_\varepsilon>u_\varepsilon(0)\},
\end{equation}
thanks to the continuous dependence on $r$.

By \eqref{13.4}, for any $x\in B_1$, there exists a unique
$x_{n+1}\in(-1,1)$ so that
$(x,x_{n+1})\in\{u_\varepsilon=u_\varepsilon(0)\}$. Combined with
the previous lemma, this then implies that
\[\{u_\varepsilon=u_\varepsilon(0)\}\cap\mathcal{B}_1=\{x_{n+1}=h_\varepsilon(x)\}, \ \ \ x\in B_1.\]
Here $h_\varepsilon$ is a function with its Lipschitz constant bounded by $4$. (This constant can be made as small as possible
by decreasing $\varepsilon$, $\tau_A$ and $\sigma$.)

To complete the proof of Theorem \ref{main result loc}, we directly apply the main result in \cite{C-C 3}.
Note that instead of the minimizing condition assumed in that paper, with our assumption \eqref{close to plane}
the argument still goes through.

\begin{appendices}
\section{A distance type function}
\numberwithin{equation}{section}
 \setcounter{equation}{0}

In this appendix $u_\varepsilon$ always denotes a solution of
\eqref{functional}, satisfying the Modica inequality \eqref{Modica inequality}.
 Here we introduce a distance
type function associated to $u_\varepsilon$ and study its
convergence as $\varepsilon\to0$. This is perhaps well known (see
for example \cite{ESS} for the parabolic Allen-Cahn case). However
we do not find an exact reference, so we include some details here.

Recall that $g_\varepsilon(t):=g(\varepsilon^{-1}t)$ is a one dimensional solution of \eqref{equation}.
Define
\[\Phi_\varepsilon(X)=g_\varepsilon^{-1}\left(u_\varepsilon(X)\right).\]
It satisfies
\begin{equation}\label{A1}
-\varepsilon\Delta\Phi_\varepsilon=f(\varepsilon^{-1}\Phi_\varepsilon)(1-|\nabla\Phi_\varepsilon|^2),
\end{equation}
where $f(t):=-\frac{W^\prime(g(t))}{\sqrt{2W(g(t))}}\in C^2(\R)$.
Note that
\[\lim_{t\to\pm\infty}f(t)=\pm\sqrt{W^{\prime\prime}(\pm1)},\]
where the convergence rate is exponential.

 The following result is a consequence of the Modica
inequality.
\begin{prop}\label{prop distance}
$|\nabla\Phi_\varepsilon|\leq 1$.
\end{prop}
\begin{proof}
Since $u_\varepsilon=g_\varepsilon(\Phi_\varepsilon)$,
\[|\nabla u_\varepsilon|=g_\varepsilon^\prime(\Phi_\varepsilon)|\nabla\Phi_\varepsilon|.\]
Thus $|\nabla\Phi_\varepsilon|\leq 1$ is equivalent to
\[|\nabla u_\varepsilon|^2\leq g_\varepsilon^\prime(\Phi_\varepsilon)^2.\]
The first integral for $g_\varepsilon$ is
\[\frac{\varepsilon}{2}\left(g_\varepsilon^\prime\right)^2=\frac{1}{\varepsilon}W(g_\varepsilon).\]
Then the final equivalent statement is exactly the Modica inequality for $u_\varepsilon$.
\end{proof}

By using \eqref{A1}, the limit of $\Phi_\varepsilon$, $\Phi_0$ can be characterized as a viscosity solution of the eikonal equation:  In $\{\Phi_0>0\}$, $\Phi_0$ is a viscosity solution of
\[|\nabla\Phi_0|^2-1=0.\]
In $\{\Phi_0<0\}$, $\Phi_0$ is a viscosity solution of
\[1-|\nabla\Phi_0|^2=0.\]
 This is similar to the {\it vanishing viscosity} method (see for example
Fleming-Souganidis \cite{S-W}). However, here we would like to give
a direct proof in our special setting.
\begin{prop}\label{prop limit of distance}
For any $\delta>0$, there exist three constants
$\varepsilon_\sharp,\tau_\sharp, R_\sharp>0$ so that the following
holds. Let $u_\varepsilon$ satisfy all of the assumptions in Theorem
\ref{main result loc} (with $R_A,\varepsilon_A$ and $\tau_A$
replaced by $R_\sharp$, $\varepsilon_\sharp$ and $\tau_\sharp$
repsectively),
 then there exists a set $\Omega\subset\mathcal{B}_1$, with $0\in\partial\Omega$ and $\partial\Omega$
  being a smooth minimal hypersurface, such that
\begin{equation}\label{8.1.1}
\sup_{\mathcal{B}_1}|\Phi_\varepsilon-d_{\partial\Omega}|\leq\delta.
\end{equation}
Here $d_{\partial\Omega}$ is the signed distance function to $\partial\Omega$, which is positive in $\Omega$.
\end{prop}
\begin{proof}
 By Hutchinson-Tonegawa \cite{H-T},
 the varifolds $V_\varepsilon$
  converge to a stationary rectifiable varifold $V$ with integer multiplicity. Moreover, \eqref{close to plane} and
the monotonicity formula implies that
\[2^{-n}\|V\|(\mathcal{B}_2(X))\leq\left(1+2\tau_\sharp\right)\sigma_0\omega_n,\quad\forall X\in\mathcal{B}_2,\]
provided $R_\sharp$ has been chosen large enough.

If $\tau_\sharp$ is sufficiently small, Allard's
$\varepsilon$-regularity theorem
 implies that
$\mbox{spt}\|V\|\cap \mathcal{B}_2$ is a smooth hypersurface and
$V\cap \mathcal{B}_2$ is the standard varifold associated to this
hypersurface with unit density. This hypersurface divides
$\mathcal{B}_1$ into two parts (see Remark \ref{rmk Hausdorff
convergence}), say $\Omega$ and $\mathcal{B}_1\setminus\Omega$. As
in Remark \ref{rmk Hausdorff convergence}, $u_\varepsilon$ converges
to $1$ uniformly in any compact set of $\Omega$, and to $-1$
uniformly in any compact set of $\Omega^c$.

Thus if $\varepsilon_\sharp$ is small enough, we can assume that there exists a set $\Omega$ with $0\in\partial\Omega$ and $\partial\Omega$ being a smooth minimal hypersurface, such that
\[\mbox{dist}_H\left(\{u_\varepsilon>0\}\cap\mathcal{B}_1,\Omega\cap\mathcal{B}_1\right)\leq\frac{\delta}{8}.\]
In particular,
\[\sup_{\mathcal{B}_1}\big|\mbox{dist}_{\{u_\varepsilon=0\}}-d_{\partial\Omega}\big|\leq\frac{\delta}{8}.\]

By Proposition \ref{prop distance}, in $\{u_\varepsilon>0\}\cap\mathcal{B}_1$,
\[\Phi_\varepsilon(X)\leq \mbox{dist}_{\{u_\varepsilon=0\}}(X)\leq\mbox{dist}_{\partial\Omega}(X)+\delta.\]

Similarly, in $\{X:|\mbox{dist}_{\partial\Omega}(X)|\leq\frac{\delta}{4}\}$, $|\Phi_\varepsilon|\leq\delta/2$. Thus in this part,
\[\big|\Phi_\varepsilon-d_{\partial\Omega}\big|\leq\big|\Phi_\varepsilon\big|+\big|d_{\partial\Omega}\big|\leq\delta.\]

In order to prove \eqref{8.1.1}, it remains to show that if
$X\in\Omega\cap\{X:\mbox{dist}_{\partial\Omega}(X)\geq\frac{\delta}{4}\}$,
where $\mbox{dist}_{\{u_\varepsilon=0\}}\geq\delta/8$, we have
\[\Phi_\varepsilon(X)\geq\mbox{dist}_{\{u_\varepsilon=0\}}(X)-\frac{\delta}{16}.\]
However, if we have chosen $\varepsilon_\sharp$ sufficiently small
(compared to $\delta$), this can be proved directly by constructing
a comparison function in the ball
$\mathcal{B}_{\mbox{dist}_{\{u_\varepsilon=0\}}(X)-\frac{\delta}{16}}(X)$,
by noting that $u_\varepsilon$ is close to $1$ in this ball.
\end{proof}

In the following we assume that as $\varepsilon\to0$, $\Phi_\varepsilon$ converges to a distance function $\Phi_0$ uniformly.
Now we present a fact about the $C^1$ convergence of $\Phi_\varepsilon$ near a $C^1$ point of $\Phi_0$.

First we establish the uniform semi-concavity of $\Phi_\varepsilon$.
\begin{lem}\label{lem uniform semi-concavity}
Let $\Phi_\varepsilon$ satisfy \eqref{A1} in $\mathcal{B}_1$. Assume
$\Phi_\varepsilon>1/2$ and $|\nabla\Phi_\varepsilon|\leq1$ in
$\mathcal{B}_1$. Then
\[\nabla^2\Phi_\varepsilon(0)\leq C,\]
where $C$ is a constant depending only on the dimension $n$.
\end{lem}
The constant $1/2$ is not essential here. It can be replaced by any positive constant.

\begin{proof}
We shall work in the setting where $\varepsilon=1$ and the ball is
$\mathcal{B}_R$, where $R=\varepsilon^{-1}$. For simplicity, all
subscripts will be dropped.

Take a unit vector $\xi$. By directly differentiating \eqref{A1} in the direction $\xi$, we get
\[-\Delta\Phi_\xi=f^\prime(\Phi)(1-|\nabla\Phi|^2)\Phi_\xi-2f(\Phi)\sum_{k=1}^{n+1}\Phi_k\Phi_{\xi k},\]
\begin{eqnarray*}
-\Delta\Phi_{\xi\xi}&=&f^\prime(\Phi)(1-|\nabla\Phi|^2)\Phi_{\xi\xi}+f^{\prime\prime}(\Phi)(1-|\nabla\Phi|^2)\Phi_{\xi}^2\\
&&-4f^\prime(\Phi)\sum_{k=1}^{n+1}\Phi_k\Phi_{\xi k}\Phi_\xi-2f(\Phi)\sum_{k=1}^{n+1}\Phi_k\Phi_{\xi\xi k}
-2f(\Phi)\sum_{k=1}^{n+1}\Phi_{\xi k}^2.
\end{eqnarray*}

Take an $\eta\in C_0^\infty(\mathcal{B}_{R/2})$ such that $\eta\equiv 1$ in $\mathcal{B}_{R/4}$, $0\leq\eta\leq 1$, $|\nabla\eta|\leq 8R^{-1}$ and $\eta^{-1}|\nabla\eta|^2+|\Delta\eta|\leq 100R^{-2}$. Denote $w:=\eta\Phi_{\xi\xi}$. Since $w=0$ on $\partial \mathcal{B}_{R/2}$, it attains its maxima at an interior point $X_0$, where
\begin{equation}\label{A.3}
\nabla w=\eta\nabla\Phi_{\xi\xi}+\Phi_{\xi\xi}\nabla\eta=0,
\end{equation}
\begin{eqnarray*}
0\geq\Delta w&=&\Delta\Phi_{\xi\xi}\eta+2\nabla\Phi_{\xi\xi}\nabla\eta+\Phi_{\xi\xi}\Delta\eta\\
&\geq&-f^\prime(\Phi)(1-|\nabla\Phi|^2)w-f^{\prime\prime}(\Phi)(1-|\nabla\Phi|^2)\Phi_{\xi}^2\eta\\
&&+4f^\prime(\Phi)\sum_{k=1}^{n+1}\Phi_k\Phi_{\xi k}\Phi_\xi\eta+2f(\Phi)\sum_{k=1}^{n+1}\Phi_k\Phi_{\xi\xi k}\eta
+2f(\Phi)\eta\sum_{k=1}^{n+1}\Phi_{\xi k}^2\\
&&+2\nabla\Phi_{\xi\xi}\nabla\eta+w\eta^{-1}\Delta\eta.
\end{eqnarray*}
Substituting \eqref{A.3} into this, and applying the Cauchy inequality to the third term, we obtain
\begin{eqnarray*}
&&4\frac{f^\prime(\Phi)^2}{f(\Phi)}|\nabla\Phi|^2\Phi_\xi^2\eta+f^{\prime\prime}(\Phi)(1-|\nabla\Phi|^2)\Phi_{\xi}^2\eta\\
&\geq&-f^\prime(\Phi)(1-|\nabla\Phi|^2)w-2f(\Phi)\nabla\Phi\nabla\eta\eta^{-1}w+f(\Phi)\eta^{-1}w^2
+w\eta^{-1}\Delta\eta-2w\eta^{-2}|\nabla\eta|^2.
\end{eqnarray*}
This can be written as
\[Aw(X_0)^2+Bw(X_0)\leq D,\]
where $A>0$, $B$ and $D$ are constants. From this we deduce that
\[w(X_0)\leq \frac{|B|}{A}+\sqrt{\frac{|D|}{A}}.\]
More precisely,
\begin{eqnarray}\label{A.4}
w(X_0)&\leq&\frac{\big|f^\prime(\Phi)\big|}{f(\Phi)}(1-|\nabla\Phi|^2)\eta+2|\nabla\Phi||\nabla\eta|+\frac{1}{f(\Phi)}|\Delta\eta|
+\frac{2}{f(\Phi)}\eta^{-1}|\nabla\eta|^2\\\nonumber
&&+\frac{1}{\sqrt{f(\Phi)}}
\left(4\frac{f^\prime(\Phi)^2}{f(\Phi)}|\nabla\Phi|^2\Phi_\xi^2\eta+f^{\prime\prime}(\Phi)(1-|\nabla\Phi|^2)\Phi_{\xi}^2\eta
\right)^{\frac{1}{2}}
\end{eqnarray}

Since $\Phi\geq R/2$ in $\mathcal{B}_{R/2}$, by the definition of
$f$ and some standard estimates on $g(t)$,
\[f(\Phi)>c\ \ \ \ \mbox{in}\ \mathcal{B}_{R/2},\]
\[|f^\prime(\Phi)|+|f^{\prime\prime}(\Phi)|\leq Ce^{-cR}\ \ \ \ \mbox{in}\ \mathcal{B}_{R/2}.\]
Substituting these into \eqref{A.4}, by using the condition
$|\nabla\Phi|\leq 1$ and our assumptions on $\eta$, we obtain
\[\sup_{\mathcal{B}_{R/4}}\Phi_{\xi\xi}\leq w(x_0)\leq CR^{-1}.\]
Rescaling back we get the claimed estimate.
\end{proof}

\begin{thm}\label{thm C1 convergence interior}
Assume that $\Phi_\varepsilon$ converges to $\Phi_0$ in $C^0(\Omega)$,
where $\Omega\subset\R^{n+1}$ is an open set and $\Phi_0>0$ in $\Omega$.
 If $\Phi_0\in C^1(\Omega)$, then $\Phi_\varepsilon$ converges to $\Phi_0$ in $C^1_{loc}(\Omega)$.
\end{thm}
\begin{proof}
Fix an open set $\Omega_0\subset\subset\Omega$.
Take an arbitrary sequence $X_\varepsilon\in \Omega_0$ such that $X_\varepsilon\to X_0\in\Omega_0$ as $\varepsilon\to 0$.
By the uniform semi-concavity of $\Phi_\varepsilon$ in $\Omega_0$, there exists a constant $C(\Omega_0)$ such that, for all $\varepsilon>0$
\[\tilde{\Phi}_\varepsilon(X):=\Phi_\varepsilon(X)-C(\Omega_0)|X-X_\varepsilon|^2\]
are concave in $\Omega_0$. In particular, for any unit vector $e$
and $h<\mbox{dist}(\Omega_0,\partial\Omega)$,
\begin{equation}\label{A5}
\tilde{\Phi}_\varepsilon(X_\varepsilon+he)\leq\tilde{\Phi}_\varepsilon(X_\varepsilon)+h\nabla\Phi_\varepsilon(X_\varepsilon)e.
\end{equation}

Because $|\nabla\Phi_\varepsilon(X_\varepsilon)|\leq 1$, assume
$\nabla\Phi_\varepsilon(X_\varepsilon)$ converges to a vector $\xi$.
By the uniform convergence of $\Phi_\varepsilon$ in $\Omega$,
passing to the limit in \eqref{A5} leads to
\[\Phi_0(X_0+he)\leq \Phi_0(X_0)+h\left(\xi\cdot e\right)+\frac{C(\Omega_0)h^2}{2}, \ \ \forall h>0.\]
Since $\Phi_0$ is differentiable at $X_0$, letting $h\to0$ gives
$\xi=\nabla \Phi_0(X_0)$. From this argument we get the uniform
convergence of $\nabla\Phi_\varepsilon$ in $\Omega_0$.
\end{proof}

\section{Several technical results}
\numberwithin{equation}{section}
 \setcounter{equation}{0}

Here we collect some technical results used in this paper.

The first one is an exponential decay estimate. This has been used
in many places and can be proved by various methods (see for example
\cite[Section 2]{C-C 3}), so here we only state the result.
\begin{lem}\label{lem B1}
If in the ball $\mathcal{B}_{2R}(0)$, $u\in C^2$ satisfies
\begin{equation} \label{1.1.1}
\left\{ \begin{aligned}
&\Delta u\geq M u , \\
&0\leq u\leq 1,
                          \end{aligned} \right.
\end{equation}
then
$$\sup_{\mathcal{B}_R(0)}u\leq Ce^{-cRM^{\frac{1}{2}}}.$$
\end{lem}

The next one gives a control of the discrepancy using the excess.
\begin{lem}\label{lem B2}
Given $M,L$ and $\tau>0$, there exist two constants $\delta>0$ and
$R(M,L,\tau)>2L$ so that the following holds. Suppose that $u$ is a
solution of \eqref{equation 0} in $\mathcal{B}_R$ with $R\geq
R(M,L,\tau)$, satisfying
\[R^{-n}\int_{\mathcal{B}_R}\frac{1}{2}|\nabla u|^2+W(u)\leq M,\]
\[\left(2L\right)^{-n}\int_{\mathcal{B}_{2L}}\left[1-\left(\nu\cdot e_{n+1}\right)^2\right]|\nabla u|^2\leq\delta.\]
Then
\[L^{-n}\int_{\mathcal{B}_L}\Big|W(u)-\frac{1}{2}|\nabla u|^2\Big|\leq\tau.\]
\end{lem}
\begin{proof}
Assume by the contrary, there exist two constants $M$ and $\tau$,
and a sequence of solutions $u_i$, defined in $\mathcal{B}_{R_i}$
with $R_i\to+\infty$, satisfying
\begin{equation}\label{B.1.1}
R_i^{-n}\int_{\mathcal{B}_{R_i}}\frac{1}{2}|\nabla u_i|^2+W(u_i)\leq
M,
\end{equation}
\begin{equation}\label{B.1.2}
\left(2L\right)^{-n}\int_{\mathcal{B}_{2L}}\left[1-\left(\nu_i\cdot
e_{n+1}\right)^2\right]|\nabla u_i|^2\to 0,
\end{equation}
 but
 \begin{equation}\label{B.1.3}
L^{-n}\int_{\mathcal{B}_L}\Big|W(u_i)-\frac{1}{2}|\nabla
u_i|^2\Big|\geq\tau.
\end{equation}

Denote the limit of $u_i$ by $u_\infty$. By passing to the limit in
\eqref{B.1.2} and the unique continuation principle, $u_\infty$
depends only on the $x_{n+1}$ variable. By the monotonicity formula,
for any $R\in(0,R_i)$,
\[R^{-n}\int_{\mathcal{B}_{R}}\frac{1}{2}|\nabla
u_i|^2+W(u_i)\leq M.\]
This also holds for $u_\infty$ by passing to
the limit. Because $u_\infty$ is one dimensional, this implies
\[\int_{-\infty}^{+\infty}\frac{1}{2}\Big|\frac{d
u_\infty}{dx_{n+1}}\Big|^2+W(u_\infty)\leq M.\] Note that except the
heteroclinic solution $g$, all the other solutions of
\eqref{equation 0} in $\R^1$ are periodic, and hence their energy on
$\R$ is infinite. By this fact we see $u_\infty\equiv g(x_{n+1}+t)$
for some constant $t\in\R$. Hence
\[W(u_\infty)\equiv \frac{1}{2}\Big|\frac{d
u_\infty}{dx_{n+1}}\Big|^2.\]
 Consequently,
\[\lim_{i\to+\infty}\int_{\mathcal{B}_L}\Big|W(u_i)-\frac{1}{2}|\nabla
u_i|^2\Big|=0.\]
However, this is a contradiction with
\eqref{B.1.3}.
\end{proof}

The following result says the energy $\varepsilon|\nabla u_\varepsilon|^2$ is mostly concentrated on the transition part $\{|u_\varepsilon|\leq 1-b\}$.
\begin{lem}\label{lem B3}
 Let $u_\varepsilon$ be a solution of \eqref{equation} defined in $\mathcal{B}_2$,
 satisfying the Modica inequality and
  \[\int_{\mathcal{B}_2}\frac{\varepsilon}{2}|\nabla u_\varepsilon|^2+\frac{1}{\varepsilon}W(u_\varepsilon)\leq M.\]
  For any $\delta>0$, there exists a constant $b\in(0,1)$ such that
\[\int_{\{|u_\varepsilon|>1-b\}\cap\mathcal{B}_1}\varepsilon|\nabla u_\varepsilon|^2\leq\delta.\]
\end{lem}
This is essentially \cite[Proposition 5.1]{H-T}. We just need to note that by the Modica inequality, we can bound $\varepsilon|\nabla u_\varepsilon|^2$ by $\varepsilon^{-1}W(u_\varepsilon)$.

Here we give a different proof. More precisely, we prove
\begin{lem}\label{lem excess concentrated}
Let $u_\varepsilon$ be as in the previous lemma. For any $1\leq
i\leq n+1$ and $\delta>0$, there exists a constant $b\in(0,1)$ such
that,
\[\int_{\{|u_\varepsilon|>1-b\}\cap\mathcal{B}_1}\varepsilon\left(\frac{\partial u_\varepsilon}{\partial x_i}\right)^2\leq\delta
\int_{\{|u_\varepsilon|<1-b\}\cap\mathcal{B}_2}\varepsilon\left(\frac{\partial
u_\varepsilon}{\partial x_i}\right)^2.\]
\end{lem}
\begin{proof}
For simplicity, denote $\xi:=\varepsilon\left(\frac{\partial
u_\varepsilon}{\partial x_i}\right)^2$, which satisfies
\begin{equation}\label{B.0.1}
\Delta\xi\geq\frac{c}{\varepsilon^2}\xi,\quad \mbox{in}\
\{|u_\varepsilon|>\gamma\}.
\end{equation}

By the gradient bound on the distance type function
$\Phi_\varepsilon$, we know for any $M>0$, there exists $0<b<1$ such
that,
\[|X_1-X_2|\geq M\varepsilon,\quad \forall X_1\in\{|u_\varepsilon|<1-2b\},
X_2\in\{|u_\varepsilon|>1-b\}.\] In other words, for any
$X\in\{|u_\varepsilon|>1-b\}$,
$B_{M\varepsilon}(X)\subset\{|u_\varepsilon|>1-2b\}$.

By \eqref{B.0.1}, if $1-2b>\gamma$, for any
$X\in\{|u_\varepsilon|>1-b\}$,
\begin{enumerate}
\item because $\xi$ is subharmonic,
\[\sup_{\mathcal{B}_{M\varepsilon/2}(X)}\xi\leq CM^{-n-1}\varepsilon^{-n-1}\int_{\mathcal{B}_{M\varepsilon}(X)}\xi(Y)dY;\]
\item by Lemma \ref{lem B1},
\[\xi(X)\leq Ce^{-cM}\sup_{\mathcal{B}_{M\varepsilon/2}(X)}\xi.\]
\end{enumerate}
Thus
\[\xi(X)\leq Ce^{-cM}M^{-n-1}\varepsilon^{-n-1}\int_{\mathcal{B}_{M\varepsilon}(X)}\xi(Y)dY.\]
Integrating this on $\{|u_\varepsilon|>1-b\}$ and then using Fubini
theorem, we obtain
\begin{eqnarray*}
\int_{\{|u_\varepsilon|>1-b\}\cap\mathcal{B}_1}\xi(X)dX&\leq&
Ce^{-cM}M^{-n-1}\varepsilon^{-n-1}\int_{\{|u_\varepsilon|>1-b\}\cap\mathcal{B}_1}\int_{B_{M\varepsilon}(0)}\xi(X+Y)dYdX\\
&=&Ce^{-cM}M^{-n-1}\varepsilon^{-n-1}\int_{B_{M\varepsilon}(0)}\int_{\{|u_\varepsilon|>1-b\}\cap\mathcal{B}_1}\xi(X+Y)dXdY\\
&\leq&Ce^{-cM}M^{-n-1}\varepsilon^{-n-1}\int_{B_{M\varepsilon}(0)}\int_{\{|u_\varepsilon|>1-2b\}\cap\mathcal{B}_2}\xi(X)dXdY\\
&\leq&Ce^{-cM}\int_{\{|u_\varepsilon|>1-2b\}\cap\mathcal{B}_1}\xi(X)dX.
\end{eqnarray*}
Hence by choosing $b$ small enough, which implies that $M$ is
sufficiently large, we get the claimed estimate.
\end{proof}

Finally, we give a lower bound of the energy in balls.
\begin{lem}\label{lem lower bound for energy}
For any $b\in(0,1)$ and $M>0$, there exist two constants $c(M,b)$
and $R(M,b)$ so that the following holds. Assume $u$ to be a
solution of \eqref{equation 0} in $\mathcal{B}_R$ where $R\geq
2R(M,b)$, satisfying $|u(0)|\leq1-b$, the Modica inequality and the
energy bound
\[\int_{\mathcal{B}_R}\frac{1}{2}|\nabla u|^2+W(u)\leq MR^n.\]
Then for any $r\in[1,R/2]$,
\begin{equation}\label{B5}
\int_{\mathcal{B}_r}|\nabla u|^2\geq c(M,b)r^n.
\end{equation}
\end{lem}
By the monotonicity formula, it is easy to get a constant $c(b)$ so
that,
\begin{equation}\label{energy lower bound}
\int_{\mathcal{B}_r}\frac{1}{2}|\nabla u|^2+W(u)\geq
c(b)r^n,\quad\forall 1<r<R.
\end{equation}
 However, this is weaker than our
statement.
\begin{proof}
We first prove that, under the assumptions of this lemma (with a
different constant $R_1(M,b)$),
\begin{equation}\label{B3}
\int_{\mathcal{B}_{R/2}}|\nabla u|^2\geq c(M,b)R^n,
\end{equation}
if $R\geq R_1(M,b)$.

Assume this is not true, that is, the claimed $R_1(M,b)$ does not
exist. Then there exists an $M>0$ and a sequence of $u_i$, which are
solutions of \eqref{equation} in $\mathcal{B}_{R_i}$ where
$R_i\to+\infty$, satisfying $|u_i(0)|\leq1-b$, the Modica inequality
and the energy bound
\[\int_{\mathcal{B}_{R_i}}\frac{1}{2}|\nabla u_i|^2+W(u_i)\leq MR_i^n,\]
but
\begin{equation}\label{B4}
R_i^{-n}\int_{\mathcal{B}_{R_i/2}}|\nabla u|^2\to 0.
\end{equation}

Let $\varepsilon_i=R_i^{-1}$ and $u_{\varepsilon_i}(X):=u_i(R_iX)$. Then
\[\int_{\mathcal{B}_1}\frac{\varepsilon_i}{2}|\nabla u_{\varepsilon_i}|^2+\frac{1}{\varepsilon_i}W(u_{\varepsilon_i})\leq M,\]
\begin{equation}\label{B6}
\int_{\mathcal{B}_{1/2}}\varepsilon_i|\nabla u_{\varepsilon_i}|^2\to 0.
\end{equation}
By the main result in \cite{H-T}, $\varepsilon_i|\nabla
u_i|^2dX\rightharpoonup\mu$ as measures, where $\mu$ is a positive
Radon measure (in fact, the weight measure associated to the limit
varifold, as in Section 4). Moreover, $u_{\varepsilon_i}\to\pm 1$
locally uniformly outside $\mbox{spt}\mu$. However, \eqref{B6}
obviously implies that $\mu(\mathcal{B}_{1/2})=0$. Thus for all $i$
large, $|u_{\varepsilon_i}|>1-b$ in $\mathcal{B}_{1/2}$. This
contradicts our assumption that $|u_{\varepsilon_i}(0)|\leq1-b$ and
proves \eqref{B3}.

By the monotonicity formula (recall that we have assumed the validation of the  Modica inequality), for any $r\in(1,R)$
\begin{equation}\label{B7}
\int_{\mathcal{B}_r}\frac{1}{2}|\nabla u|^2+W(u)\leq Mr^n.
\end{equation}
Thus the above discussion covers the case $r\in[R_1(M,b), R]$ in
\eqref{B5}, that is, \eqref{B3} holds for every $r\in[R_1(M,b), R]$.

For the remaining case, we only need to note that it is impossible to have $u\equiv u(0)$, because otherwise
\[\int_{\mathcal{B}_r}\frac{1}{2}|\nabla u|^2+W(u)=\omega_{n+1}W(u(0))r^{n+1}\geq Mr^n,\]
provided $r\geq
\frac{M}{\omega_{n+1}\left(\inf_{s\in[-1+b,1-b]}W(s)\right)}$. Then
it can be directly verified that
\[\int_{\mathcal{B}_1}|\nabla u|^2\geq c(M),\]
by using \eqref{B7} with
$r=\frac{M}{\omega_{n+1}\left(\inf_{s\in[-1+b,1-b]}W(s)\right)}$.

Choosing $R(M,b)=\max\{R_1(M,b),
\frac{M}{\omega_{n+1}\left(\inf_{s\in[-1+b,1-b]}W(s)\right)}\}$ we
 finish the proof.
\end{proof}

\end{appendices}

\bigskip

\noindent {\bf Acknowledgments.} The author is grateful to referees
for their careful reading and useful suggestions. This work was
supported by NSFC No. 11301522.

\end{document}